\documentclass[11pt,reqno]{amsproc}

\usepackage[utf8]{inputenc}
\usepackage[margin=1in]{geometry}
\usepackage{amsmath}
\usepackage{amssymb}
\usepackage{amsfonts}
\usepackage{amsthm}
\usepackage{mathrsfs}
\usepackage{mathtools}
\usepackage{enumitem}
\usepackage[usenames,dvipsnames]{xcolor}
\usepackage{hyperref}
\hypersetup{colorlinks=true, citecolor=magenta, linkcolor=blue, filecolor=teal, urlcolor=cyan, pdfpagemode=FullScreen}
\usepackage{graphicx}

\numberwithin{equation}{section}

 \newtheorem{theorem}{Theorem}[section]
 \newtheorem{proposition}[theorem]{Proposition}
 \newtheorem{lemma}[theorem]{Lemma}
 \newtheorem{conjecture}[theorem]{Conjecture}
 \newtheorem{corollary}[theorem]{Corollary}
 \theoremstyle{definition}

 \theoremstyle{remark}
 \newtheorem{remark}[theorem]{Remark}
\def\p{\partial}
\newcommand{\les}{\lesssim}

\newcommand{\mean}[1]{\langle#1 \rangle}

\newcommand*{\supp}{\ensuremath{\mathrm{supp\,}}}
\newcommand*{\tr}{\ensuremath{\mathrm{tr\,}}}

\newcommand*{\RR}{\ensuremath{\mathbb{R}}}

\newcommand{\eps}{\varepsilon}

\renewcommand*{\tilde}{\widetilde}

\usepackage{etoolbox}
\makeatletter
\patchcmd{\@maketitle}
  {\global\topskip8pc\relax}
  {\global\topskip5pc\relax}   % reduce to taste
  {}{\ClassWarning{amsproc}{Could not patch topskip in \string\@maketitle}}
\makeatother
\begin{document}

\title[RSS for 3D Navier-Stokes]{On rotated backwards self-similar solutions of the incompressible 3D Navier-Stokes equations}

\author{Ben Pineau} 
\address{Courant Institute School of Mathematics, Computing, and Data Science, New York University, New York, NY, 10012}
\email{\href{https://bpineau2.github.io/}{brp305@nyu.edu}}

\author{Vlad Vicol}
\address{Courant Institute School of Mathematics, Computing, and Data Science, New York University, New York, NY, 10012}
\email{\href{https://cims.nyu.edu/~vicol/}{vicol@cims.nyu.edu}}

\begin{abstract}
We consider backwards globally self-similar solutions of the 3D incompressible Navier-Stokes equations 
which are invariant under the \emph{joint action of scaling} (the natural parabolic scaling) \emph{and rotation}  about a given axis, at a constant angular speed $\alpha$ in self-similar time. 

For these so-called \emph{rotated self-similar} solutions (RSS), we prove that if they satisfy a Type~I upper bound, and if the rotation parameter  \emph{$\alpha$ is either too small, or too large}, then they must be trivial. This Liouville-type result extends the classical works of Ne\v{c}as-R\r{u}\v{z}i\v{c}ka-\v{S}ver\'ak ('96) and Tsai ('98), which consider $\alpha=0$, to the case of similarity profiles which experience nontrivial rotation. Our results partially answer a question posed by Perelman.

For backwards globally self-similar solutions which are invariant under the discrete action of scaling and rotation, the so-called \emph{rotated discretely self-similar} solutions (RDSS), we obtain similar Liouville-type results under a Type~I upper bound, assuming extreme values of the rotation parameter $\alpha$, and assuming that the scaling factor $\lambda$ is sufficiently close to $1$. 

We also establish a new regularity criterion for 3D Navier-Stokes which is \emph{local} in nature: if the solution satisfies a Type~I upper bound in a unit parabolic cylinder, and there is a single time-slice at which the solution is locally \emph{approximately} self-similar, then the top-center of the parabolic cylinder is a regular point of the Navier-Stokes flow.

The proof of all these results rests on the introduction of a robust weighted-$L^2$ framework. In particular, our method is quantitative and is not sensitive to whether the Bernoulli head pressure satisfies a maximum principle, which was a key obstruction in previous works. 
\end{abstract}

\maketitle
%\tableofcontents
 
\section{Introduction}

We consider the Cauchy problem for the 3D incompressible Navier-Stokes equations
\begin{subequations}\label{eq:NSE}
\begin{align}
\p_t u - \Delta u + (u\cdot\nabla)u + \nabla p 
&= 0,
\\
\nabla\cdot u 
&= 0,
\end{align}
\end{subequations}
for the unknown velocity field $u=u(x,t)$ and scalar pressure $p=p(x,t)$, with $x\in \RR^3$ and $t\geq -1$. We supplement~\eqref{eq:NSE} with a \textit{smooth}, \textit{divergence-free}, and sufficiently fast \textit{decaying} initial datum $u(\cdot,-1)=u_0$. Every such datum $u_0$ produces a unique, smooth, decaying solution $(u,p)$ of~\eqref{eq:NSE}, on a maximal time interval $[-1,T_*)$, for some $T_* = T_*(u_0) \in (-1,\infty]$, see~\cite{CF88,Lemarie02,Tsai18,BV22}. Whether $T_*$ can be finite is a central open problem in the mathematical theory of viscous incompressible fluids~\cite{Leray34,Fefferman06}. 

Suppose $T_*$ is finite, and without loss of generality, let $T_*=0$. The classical candidate for a finite-time singularity is a \textit{self-similar} solution: zooming-in and rescaling the solution at just the right rate in time, one sees the same object as $t \to 0^-$, a \textit{similarity profile}. 
For many nonlinear equations of mathematical physics this picture is not merely a heuristic, but a theorem.\footnote{Self-similar singularity formation has been rigorously established, for instance in: the semilinear heat equation~\cite{GigaKohn85,HerreroVelazquez93,BricmontKupiainen94,MerleZaag97}, the incompressible Euler equations~\cite{Elgindi21,ChenHou25}, shocks for compressible Euler~\cite{BuckmasterShkollerVicol23}, implosions for the compressible Euler and Navier-Stokes ~\cite{MRRS22a,MRRS22b,BuckmasterCaoLaboraGomezSerrano25,ChenShkollerVicol26}, Prandtl boundary-layer separation~\cite{DalibardMasmoudi19,CollotGhoulMasmoudi21}, Euler--Poisson gravitational collapse~\cite{GuoHadzicJang21},  chemotactic aggregation in reaction--diffusion systems~\cite{HerreroVelazquez96,HerreroMedinaVelazquez98}, the mass-critical nonlinear Schr\"odinger equation~\cite{MerleRaphael04}, the complex Ginzburg--Landau equation~\cite{PlechacSverak01,MasmoudiZaag08}, the harmonic map heat flow and wave maps~\cite{RaphaelSchweyer13,RaphaelRodnianski12}, and the gravitational collapse of a self-gravitating scalar field in general relativity~\cite{Christodoulou94}. This is a highly non-exhaustive list, which for instance does not cover the many interesting results on self-similar blowup in 1D toy-models arising in fluid dynamics.}

Short of proving unconditional global-in-time regularity for~\eqref{eq:NSE}, it is thus a natural question to ask whether smooth self-similar solutions of the 3D incompressible Navier-Stokes equations exist. 

\subsection{Backwards self-similar solutions}
\label{sec:SS}
In~\cite{Leray34} Leray proposed to seek solutions of~\eqref{eq:NSE} that are invariant under the natural Navier-Stokes scaling:  
\[
u(x,t) = \lambda u(\lambda x,\lambda^2 t), 
\qquad \mbox{for all } \; 
\lambda>0.
\] 
For a solution defined on $[-1,0)$ and which blows up as $t \to 0^-$ at $x=0$, one may take $\lambda = \frac{1}{\sqrt{-t}}$, thus arriving at a \textit{backwards}\footnote{The adjective \emph{backwards} is used to distinguish these solutions from \emph{forward} self-similar solutions, which satisfy the ansatz $u(x,t)=\frac{1}{\sqrt{t}}U (\frac{x}{\sqrt{t}} )$, for $t>0$, and which arise from scale-invariant initial data. Forward self-similar and forward discretely self-similar solutions play an important role in the Cauchy theory of the Navier-Stokes equations for scale-invariant data, and in questions related to non-uniqueness; see, for instance, \cite{JiaSverak14,Tsai14,BradshawTsai17_2,KorobkovTsai16,JiaSverakTsai18,JiaSverak15, GuillodSverak23, AlbrittonBrueColombo22,GLX26,AGKR26} and references therein. Since our goal is to rule out possible finite-time blowup profiles for smooth Navier-Stokes solutions, all self-similar, rotated self-similar, and discretely self-similar solutions considered in this paper are backwards in time. See also~\cite{GuillodWittwer15} for related stationary solutions.} \textit{globally}\footnote{We call the solution \textit{globally} self-similar to emphasize that the profile $U$ is defined on all of $\RR^3$ and does not depend on time.} \textit{self-similar solution} (SS):
\begin{equation}\label{eq:SS}
u(x,t)=\frac{1}{\sqrt{-t}}\,U \Big(\frac{x}{\sqrt{-t}}\Big),
\qquad x\in \RR^3, \quad t<0.
\end{equation}
Here  $U\colon\RR^3\to\RR^3$ is a time-independent \textit{similarity profile}.  Inserting ansatz~\eqref{eq:SS} into the Navier-Stokes system~\eqref{eq:NSE} shows that the profile $U = U(y)$ solves the stationary Leray system
\begin{subequations}\label{eq:Leray}
\begin{align}
\tfrac12 U + \tfrac12 (y\cdot\nabla)U   -\Delta U +  (U\cdot\nabla)U + \nabla P &= 0,\\
\nabla\cdot U &= 0,
\end{align}
\end{subequations}
on $\RR^3$. Here we have denoted the self-similar space variable as $y= \frac{x}{\sqrt{-t}}$.

Leray's question, whether~\eqref{eq:Leray} admits a nontrivial finite-energy solution $U$, was answered in the negative by Ne\v{c}as, R\r{u}\v{z}i\v{c}ka, and \v{S}ver\'ak~\cite{NRS96}: if $U\in L^3(\RR^3)$, then $U\equiv 0$.\footnote{Results which show that globally defined solutions of an elliptic system must vanish under a mild integrability or growth constraint are sometimes called \textit{Liouville-type} theorems.} The proof rests on a maximum principle for the head-pressure, or Bernoulli function, defined by
\begin{equation}\label{eq:Bernoulli}
\Pi := P +  \tfrac12 |U|^2 + \tfrac12\, (y\cdot U).
\end{equation}
For a solution of~\eqref{eq:Leray}, $\Pi$ satisfies the elliptic equation
\begin{equation}\label{eq:Bernoulli:ineq}
-\Delta \Pi + \big(U+\tfrac12 y\big)\cdot\nabla\Pi = -|\Omega|^2 \le 0, \qquad \Omega = \nabla \times U,
\end{equation}
and hence obeys a maximum principle. The triviality of $U$ then follows by comparison against exponentially growing radial supersolutions of~\eqref{eq:Bernoulli:ineq}.

Tsai~\cite{Tsai98} established the same conclusion for profiles $U\in L^p(\RR^3)$ with $p\in(3,\infty]$, as well as for weak solutions satisfying only local energy estimates. Further Liouville-type theorems for self-similar solutions were obtained by Chae and Wolf~\cite{ChaeWolf17b}; for instance, they prove that a globally self-similar solution with profile $U \in L^{p,\infty}(\RR^3)$ with $p \in (\frac 32,\infty)$, must vanish identically. 

\subsection{Rotated self-similar solutions (RSS)}
\label{sec:RSS}
In addition to the scaling invariance $u(x,t)\mapsto \lambda u(\lambda x,\lambda^2 t)$ for $\lambda>0$, the Navier-Stokes equations are also invariant under rotations $u(x,t) \mapsto \mathcal{O} u(\mathcal{O}^\intercal x,t)$, for $\mathcal{O} \in SO(3)$. Thus, a larger class of self-similar solutions may be obtained by composing scaling and rotation. Denote by $R(s)\in SO(3)$ the matrix representing rotation by angle $s$ about the $\vec{e}_3$-axis,\footnote{The choice of the $\vec{e}_3$-axis as the axis of rotation is made here only for convenience.} and let $J$ be the associated antisymmetric generator:
\begin{equation}\label{eq:RJ}
R(s)=\begin{pmatrix} \cos s & -\sin s & 0\\ \sin s & \cos s & 0\\ 0 & 0 & 1\end{pmatrix},
\qquad
J=R'(0)=\begin{pmatrix} 0 & -1 & 0\\ 1 & 0 & 0\\ 0 & 0 & 0\end{pmatrix}.
\end{equation}
Letting $\mathcal{O} = R(2\alpha \log \lambda)$ for some fixed $\alpha \in \RR$ leads to the class of \textit{rotated self-similar} solutions (RSS), as being those for which 
\[
u(x,t) = \lambda R(2\alpha \log \lambda) u \bigl( \lambda R(-2\alpha \log \lambda) x,\lambda^2 t \bigr),
\qquad \mbox{for all } \; \lambda >0.
\] 
For a solution of~\eqref{eq:NSE} defined on $[-1,0)$ and which blows up as $t \to 0^-$ at $x=0$, one may take $\lambda = \frac{1}{\sqrt{-t}}$ so that  $\mathcal{O} = R(2\alpha \log \lambda) = R(\alpha s)$, where $s=- \log(-t)$; this corresponds to a similarity profile that is spun at a constant angular speed $\alpha\in \RR$. The \textit{backwards rotated globally self-similar}  ansatz thus reads
\begin{equation}\label{eq:RSS}
u(x,t)=\frac{1}{\sqrt{-t}}\,R(\alpha s)\,
U\!\left(\frac{R(-\alpha s) \, x}{\sqrt{-t}}\right),
\qquad
s=-\log(-t),
\end{equation}
for some $\alpha \in \RR$ and a time-independent similarity profile $U\colon\RR^3\to\RR^3$. When $\alpha=0$, \eqref{eq:RSS} reduces to the self-similar ansatz~\eqref{eq:SS} because $R(0) = {\rm Id}$. 

Substituting~\eqref{eq:RSS} into~\eqref{eq:NSE}, the profile $U\colon\RR^3\to\RR^3$ is found to solve
\begin{subequations}
\label{eq:profile}
\begin{align}
\alpha\big(JU - (Jy\cdot\nabla)U\big) + \tfrac12 U + \tfrac12 (y\cdot\nabla)U - \Delta U +  (U\cdot\nabla)U + \nabla P &= 0,
\label{eq:profile:a}\\
\nabla\cdot U &= 0,
\label{eq:profile:b}
\end{align}
\end{subequations}
on $\RR^3$. Here,  the self-similar space variable is $y= R(-\alpha s) \frac{x}{\sqrt{-t}}$, and $J$ is defined in~\eqref{eq:RJ}.
According to Tsai's book~\cite[Chapter 8]{Tsai18}, the backward rotated self-similar ansatz~\eqref{eq:RSS} was proposed by G.~Perelman in a private communication to G.~Seregin; see also~\cite{BradshawTsai17,JiaSverakTsai18}.

In contrast to the non-rotated case ($\alpha=0$), it is currently an \emph{open problem} to rule out nontrivial rotated backward self-similar singularities of the 3D Navier-Stokes equations. The conjecture (attributed to G.~Perelman) stated in~\cite[Conjecture 8.9]{Tsai18}, or equivalently in~\cite[Open Problem 5.2]{BradshawTsai17}, is the following:
\begin{conjecture}\label{conj:Perelman}
Let $\alpha \neq 0$, and let $U$ be a smooth solution of~\eqref{eq:profile} on $\RR^3$. If there exists a constant $C_{U,0}>0$ such that 
\begin{equation}\label{eq:decay}
|U(y)|\le \frac{C_{U,0}}{1+|y|},
\end{equation}
for all $y\in\RR^3$, then $U\equiv 0$.
\end{conjecture}

\begin{remark}[Type~I blowup rate]
\label{rem:type:I}
Using the ansatz~\eqref{eq:RSS}, it is immediate to show the global bound~\eqref{eq:decay} is equivalent (within this class of rotated backwards self-similar solutions) to the borderline Type~I blowup rate\footnote{The Type~I blowup rate~\eqref{eq:type:I} is precisely the rate at which the
dissipation balances the material derivative, and it is the scaling-critical threshold underlying partial regularity theory of~\cite{CKN82}. In the axisymmetric setting, Type~I blowup for 3D Navier-Stokes has been ruled out~\cite{ChenStrainTsaiYau08,SereginSverak09}; consequently, any putative singularity must be of Type~II. For recent results on potential Type~II blowups, we refer to~\cite{Seregin23,Seregin24,Seregin26} and references therein. Note that in a Type~II blowup scenario  for 3D Navier-Stokes, the self-similar scaling is expected to be dictated by the 3D Euler equations; self-similar solutions of 3D Euler can however be shown to not be viable for a ``lift'' to 3D Navier-Stokes, under a local outgoing property~\cite{CIV26}.}
\begin{equation}
 |u(x,t)| \le \frac{C_{U,0}}{|x|+\sqrt{-t}}
\label{eq:type:I}
\end{equation}
for all $(x,t) \in \RR^3 \times [-1,0)$. The constants in~\eqref{eq:decay} and~\eqref{eq:type:I} are the same. Thus, an equivalent way to state Conjecture~\ref{conj:Perelman} is: \emph{Assume that $u$ is a smooth solution of the 3D Navier-Stokes equations~\eqref{eq:NSE} on $\RR^3 \times [-1,0)$ which satisfies the bound~\eqref{eq:type:I}. Furthermore, assume that there exists $\alpha>0$ such that $u$ satisfies the (global) rotated backwards self-similar ansatz~\eqref{eq:RSS}. Then, $u\equiv 0$.}
\end{remark}

\begin{remark}[$C_{U,0}$ is independent of $\alpha$]
\label{rem:C:indep:alpha}
The equivalent formulation Conjecture~\ref{conj:Perelman} that we have given in Remark~\ref{rem:type:I} stresses an important point. The constant $C_{U,0}$ appearing in \eqref{eq:type:I} is just an a-priori Type~I upper bound for the magnitude of the Navier-Stokes solution; this constant does not know anything about rotations being present or not. The ansatz~\eqref{eq:RSS} is separate, and this is where the rotation parameter $\alpha$ comes in. Moreover, at the initial time $t=-1$ we have $s= - \log(1) = 0$, and $R(0) = {\rm Id}$; the initial data $u_0$ does not know about the rotation either. As such, \emph{throughout this paper we assume that the constant $C_{U,0}$ in~\eqref{eq:decay} (or~\eqref{eq:type:I}) depends on $u_0$, but not on $\alpha$.}
\end{remark}

Why is Conjecture~\ref{conj:Perelman} open for $\alpha \neq 0$? Assumption~\eqref{eq:decay} only implies that the profile $U$ belongs to the weak-$L^3$ class $L^{3,\infty}(\RR^3)$. Had we assumed that the profile decays a little bit faster, to ensure that $U \in L^3(\RR^3)$, then we would have $u\in L^\infty([-1,0);L^3(\RR^3))$, and by the theory of Escauriaza, Seregin, and \v{S}ver\'ak~\cite{ESS} this would imply regularity (and hence the triviality of $U$). This argument applies to rotated self-similar solutions since~\eqref{eq:RSS} implies that $\|u(\cdot,t)\|_{L^3}=\|U\|_{L^3}$ for all $t<0$ and $\alpha\in\RR$. The genuine difficulty therefore lies in the borderline decay rate~\eqref{eq:decay}, which guarantees $U\in L^p(\RR^3)$ for every $p>3$, but not membership in the scaling-critical space $L^3$. When $\alpha=0$, the Liouville theorem established in~\cite{Tsai98} fundamentally relies on the maximum principle for the head pressure $\Pi$~\eqref{eq:Bernoulli} recalled above. When $\alpha\neq 0$, the issue is that the antisymmetric rotation terms $\alpha\big(JU-(Jy\cdot\nabla)U\big)$ present in~\eqref{eq:profile} destroy this structure. We are not aware (see also~\cite{BradshawTsai17,Tsai18}) of any scalar quantity (an $\alpha$-dependent modification of~\eqref{eq:Bernoulli}) playing the role of the Bernoulli function, and which satisfies a maximum principle. This is the central obstruction, and the reason that no Liouville theorem has so far been available for rotated self-similar solutions, when $\alpha\neq0$.

The purpose of this paper is to prove that if $U$ is a RSS solution of 3D Navier-Stokes and $|\alpha|$ is either sufficiently small, or sufficiently large when compared to $C_{U,0}$, then $U\equiv 0$. 

\begin{theorem}[Main result for RSS solutions]
\label{thm:main}
Let $u$ be a solution of the 3D incompressible Navier-Stokes equations on $\RR^3 \times [-1,0)$, which satisfies the Type~I upper bound~\eqref{eq:type:I} for some $C_{U,0}>0$. Assume that there exists $U\in C^2(\RR^3)$ and $\alpha\in \RR$ such that $u$ is a backwards rotated globally self-similar solution of~\eqref{eq:NSE} with profile $U$ (the RSS ansatz~\eqref{eq:RSS} holds). There exists $0 < \underline{\alpha} = \underline{\alpha} (C_{U,0}) \ll 1$ and $1 \ll \overline{\alpha} = \overline{\alpha} (C_{U,0}) <\infty$, such that if \emph{either} $|\alpha| < \underline{\alpha}$ \emph{or} $|\alpha| > \overline{\alpha}$, then $U\equiv 0$.
\end{theorem}
 
Theorem~\ref{thm:main} resolves Conjecture~\ref{conj:Perelman} for $|\alpha|\ll 1$ and $|\alpha|\gg 1$, but leaves open the case   $\alpha \approx 1$. 

The main ideas in the proof of Theorem~\ref{thm:main} are as follows. When $|\alpha| \ll 1$, a natural expectation is that \emph{if} one finds a new proof of the results in~\cite{NRS96,Tsai98} for $\alpha=0$, which is quantitative and does not rely on the maximum principle, then one should obtain the desired Liouville-type result for small $|\alpha|$, with this new proof. This new proof is given in Section~\ref{sec:alpha:small}, and is based on a weighted $L^2$-based enstrophy estimate, which relies on two independent facts.  
First, a naive energy estimate shows that there exists a large enough radius $\bar{R}>0$, such that if the vorticity $\Omega$ has small $L^2$ norm when restricted to the ball $B_{\bar{R}}$,\footnote{Here and throughout the paper, we denote by $B_R$ the ball of radius $R>0$ centered at the origin in $\RR^3$.}  then $U \equiv 0$; see~Proposition~\ref{prop:local:vort}.
Second, we observe that if $U$ satisfies suitable smoothness and decay properties, then the non-normal operator $L = - \Delta + (U + \frac 12 y) \cdot \nabla$  present on the left side of~\eqref{eq:Bernoulli:ineq} is such that its $L^2$-adjoint $L^*$ has an element $w = w(y)$ in its kernel, which is smooth and obeys Gaussian upper and lower bounds; see~Proposition~\ref{prop:weight:construction}. We may then use this kernel element $w$ as a weight in a weighted-$L^2$ energy estimate for $\Omega$ to conclude that $\int_{\RR^3} |\Omega|^2 w dy \lesssim |\alpha|$, where the implicit constant depends only on $C_{U,0}$. Due to the $\alpha$-independent Gaussian lower bound for $w$, we deduce that $\int_{B_{\bar {R}}} |\Omega|^2 dy \lesssim |\alpha|$. 
Combining the above two facts implies Theorem~\ref{thm:main} for $|\alpha| \ll 1$.
 
For the case $|\alpha| \gg 1$ the proof idea is different. For large $|\alpha|$ one may expect that the term $\alpha\big(JU - (Jy\cdot\nabla)U\big)$ present in \eqref{eq:profile} is dominant, and so a smooth solution $U$ wants to stay close to the kernel of the operator $\mathcal{R} = J   - (J y) \cdot \nabla  $. Note that the kernel of $\mathcal{R}$ consists of axisymmetric vector fields (with respect to rotations around the $\vec{e}_3$ axis). In Section~\ref{sec:alpha:large} we are able to quantify this intuition by showing that if $|\alpha|$ is sufficiently large, then a weighted-$L^2$ norm of $\mathcal{R}U$ is $\ll \frac{1}{\alpha}$, where the implicit constant depends on $C_{U,0}$; see Proposition~\ref{prop:large:alpha}. Here it is important that we obtain a weighted-$L^2$ bound on $\mathcal{R}U$ which is $\ll \frac{1}{\alpha}$ as opposed to merely $\lesssim \frac{1}{\alpha}$, and that the weight is bounded independently of $\alpha$. With this bound, we return to the $L^2$-weighted energy estimate for the vorticity with weight $w$ (the element in the kernel of $L^*$ described earlier) and show that if $|\alpha|$ is sufficiently large, then $\int_{B_{\overline{R}}} |\Omega|^2 dy \lesssim \alpha \|e^{-|y|^2/8} \mathcal{R}U\|_{L^2(\RR^3)} \ll 1$. The proof of Theorem~\ref{thm:main} for $|\alpha| \gg 1$   follows as before, because there is not enough enstrophy in a large ball around the origin.

\subsection{Discretely self-similar solutions (DSS)}
The self-similar solutions of 3D Navier-Stokes discussed in Sections~\ref{sec:SS} and~\ref{sec:RSS} are invariant under the scaling symmetry $u(x,t) \mapsto \lambda u (\lambda x , \lambda^2 t)$ \emph{for all} $\lambda > 1$. One can however consider a more general class of solutions, which are only assumed to be invariant under the Navier-Stokes scaling symmetry for \textit{one particular}\footnote{Clearly if $u$ is invariant under $u(x,t) \mapsto \lambda u (\lambda x , \lambda^2 t)$ for one value of $\lambda>1$, then it is also invariant under $u(x,t) \mapsto \lambda^k u (\lambda^k x , \lambda^{2k} t)$ for all $k\in\mathbb{N}$. For consistency, throughout this paper we shall refer to the discrete self-similarity factor $\lambda>1$ as being the one which is \textit{smallest possible}.} value of $\lambda>1$, namely
\[
u(x,t) = \lambda u(\lambda x, \lambda^2 t), \qquad\mbox{for a fixed } \; \lambda > 1.
\]
Such solutions are called \emph{discretely self-similar} (DSS). To emphasize the value $\lambda>1$ of the similarity factor, this class of solutions is often denoted as $\lambda$-DSS; see~\cite[Section~8]{Tsai18}.

We emphasize that discretely self-similar solutions do not correspond to a time-independent similarity profile; instead,  $\lambda$-DSS  solutions representing a finite time blowup as $t \to 0^-$ at $x=0$ are given by the \emph{backwards globally\footnote{These solutions are still referred to as \textit{globally} self-similar because~\eqref{eq:DSS:ansatz} is required to hold for all $y \in \RR^3$.} discretely self-similar} ansatz 
\begin{subequations}
\label{eq:DSS:total}
\begin{equation}
\label{eq:DSS:ansatz}
u(x,t)=\frac{1}{\sqrt{-t}}\, 
U(y,s),
\qquad
y= \frac{x}{\sqrt{-t}},
\qquad 
s=-\log(-t),
\end{equation}
for a \emph{periodic-in-self-similar-time} profile $U\colon\RR^3 \times [0,S] \to\RR^3$, where the time-period is given by
\begin{equation}
\label{eq:DSS:period}
S := 2 \log(\lambda) >0.
\end{equation}
\end{subequations}
Note that as $\lambda \to 1^+$ the profile $U$ formally becomes time-independent since $S \to 0^+$ in~\eqref{eq:DSS:period}, and thus the ansatz~\eqref{eq:DSS:ansatz} becomes the Leray self-similar ansatz~\eqref{eq:SS}.

\begin{remark}[RSS $\subseteq$ DSS]
\label{rem:RSS:DSS}
Note that because $R(s+2\pi) = R(s)$ and $R(0) = {\rm Id}$, rotated self-similar (RSS) solutions as defined in Section~\ref{sec:RSS} for a given $\alpha \in \RR$ are automatically discretely self-similar ($\lambda$-DSS) for any factor $\lambda> 1$ such that  $2\alpha\log(\lambda)\in 2\pi\mathbb{Z}$. Put differently, if $u$ is given by ansatz~\eqref{eq:RSS} with a time-independent profile $\bar U$ and rotation parameter $\alpha \in \RR$, then $u$ satisfies the ansatz~\eqref{eq:DSS:ansatz} with the periodic-in-time profile $U(y,s) = R(\alpha s) \bar U(R(-\alpha s) y)$, and period $S = \frac{2\pi}{|\alpha|}$. Due to~\eqref{eq:DSS:period}, the discrete similarity factor $\lambda$ and the rotation strength $\alpha$ are related via 
\[
\lambda = e^{\frac{\pi}{|\alpha|}}.
\] 
Importantly, we note that \emph{the regime $|\alpha|\ll 1$ corresponds to $\lambda\gg 1$}, while \emph{the regime $|\alpha| \to \infty$ corresponds to $\lambda \to 1^+$}.
\end{remark}

Substituting~\eqref{eq:DSS:ansatz} into~\eqref{eq:NSE}, we discover that the profile $U\colon\RR^3 \times [0,S] \to\RR^3$  is a time-periodic solution of 
\begin{subequations}
\label{eq:profile:DSS}
\begin{align}
\partial_s U + \tfrac12 U + \tfrac12 (y\cdot\nabla)U - \Delta U +  (U\cdot\nabla)U + \nabla P &= 0,
\label{eq:profile:DSS:a}\\
\nabla\cdot U &= 0,
\label{eq:profile:DSS:b}
\end{align}
\end{subequations}
on $\RR^3$, with period $S = 2 \log(\lambda)$ as in~\eqref{eq:DSS:period}.  

Not much seems to be known  about the existence of solutions to \eqref{eq:profile:DSS}, other than the following Liouville-type theorem of Chae and Wolf~\cite{ChaeWolf17}, which concerns the case of discrete self-similarity factors $\lambda$ such that $0< \lambda - 1 \ll 1$. Precisely,~\cite{ChaeWolf17} shows that for every $C_{U,0}>0$ there exists $\lambda_* = \lambda_*(C_{U,0}) > 1$ such that the following holds: if $u\in C^{\infty}((-\infty,0)\times\mathbb{R}^3)$ is a solution of 3D Navier-Stokes which satisfies the Type~I bound~\eqref{eq:type:I} on $(-\infty,0)\times\mathbb{R}^3$ for some $C_{U,0}>0$, and if $u$ is $\lambda$-DSS (meaning that $u$ is given by the ansatz~\eqref{eq:DSS:total}) for some $\lambda \in (1,\lambda_*)$, then $U\equiv 0$.
 
The Chae and Wolf result~\cite{ChaeWolf17} is obtained by a \textit{compactness argument}, which reduces the problem to the known non-existence results for Leray self-similar solutions in~\cite{NRS96,Tsai98}. In light of Remark~\ref{rem:RSS:DSS} which provides a direct analogy between the $0<\lambda-1\ll 1$ regime in $\lambda$-DSS solutions and the $|\alpha|\gg 1$ regime in RSS solutions, it is worth noting that~\cite{ChaeWolf17} gives an \textit{alternative proof} of the large $|\alpha|$ case in Theorem~\ref{thm:main}.  The advantage of the proof we present in this paper (see Section~\ref{sec:alpha:large}) is that it is \emph{quantitative} and provides an explicit a-priori bound for the vorticity. In fact, a slight modification of the proof of Theorem~\ref{thm:main} immediately recovers the main result in~\cite{ChaeWolf17}, which we re-state here for convenience:
\begin{theorem}[Main result for DSS solutions]
\label{thm:main:DSS}
Let $u$ be a solution of the 3D incompressible Navier-Stokes equations on $\RR^3 \times [-1,0)$, which satisfies the Type~I upper bound~\eqref{eq:type:I} for some $C_{U,0}>0$. Assume that there exists $\lambda >1$ and $U\in C^2(\RR^3 \times [0,S])$ with $S= 2 \log(\lambda)$, such that $u$ is a backwards  globally discretely self-similar solution of~\eqref{eq:NSE} with profile $U$ (the DSS ansatz~\eqref{eq:DSS:ansatz} holds). There exists $\underline{\lambda} = \underline{\lambda} (C_{U,0}) >1$, such that if  $1 < \lambda < \underline{\lambda}$, then $U\equiv 0$.
\end{theorem}

The proof of Theorem~\ref{thm:main:DSS} is given in Section~\ref{sec:DSS:RDSS} and appears as a particular case ($\alpha=0$) of the results discussed in the next subsection.

\subsection{Rotated Discretely self-similar solutions (RDSS)}
Analogously to how the discretely self-similar (DSS) solutions generalize the Leray self-similar (SS) solutions by requiring that scale invariance holds for a  minimal $\lambda>1$ instead of for all $\lambda>0$, one may generalize the class of rotated self-similar (RSS) solutions by requiring that invariance under the combined effect of rotation (with self-similar angular speed $\alpha$) and scaling holds for a  minimal $\lambda>1$, instead of for all $\lambda>0$. This leads one to the concept of \textit{rotated discretely self-similar} (RDSS) solutions, as being those for which\footnote{In~\cite[Equation~(1.9)]{BradshawTsai17} and~\cite[Equation~(8.10)]{Tsai18} the definition replaces the existence of an $\alpha \in \RR$ with the existence of a phase $\phi = 2\alpha \log \lambda \in \RR$. Since $\lambda>1$, these are clearly equivalent definitions; we have chosen this definition for consistency with the notation in Section~\ref{sec:RSS}.}
\[
u(x,t) = \lambda R(2\alpha \log \lambda) u \bigl( \lambda R(-2\alpha \log \lambda)x ,\lambda^2 t \bigr),
\qquad \mbox{for a fixed } \; \lambda >1, \; \mbox{ and a fixed } \; \alpha \in \RR,
\]
for all $x\in \RR^3$ and $t \in [-1,0)$. To emphasize the value $\lambda>1$ of the similarity factor and the value $\alpha \in \RR$ of the angular speed, we shall sometimes denote this class of solutions as $(\alpha,\lambda)$-RDSS.

As with DSS solutions, we emphasize that RDSS solutions do not correspond to a time-independent rotated similarity profile; instead, one may verify that RDSS solutions representing a finite time blowup as $t \to 0^-$ at $x=0$ are given by the \emph{backwards globally\footnote{These solutions are still referred to as \textit{globally} self-similar because~\eqref{eq:RDSS:ansatz} is required to hold for all $y \in \RR^3$.} rotated discretely self-similar} ansatz 
\begin{subequations}
\label{eq:RDSS:total}
\begin{equation}
\label{eq:RDSS:ansatz}
u(x,t)=\frac{1}{\sqrt{-t}}\, 
R(\alpha s) U(y,s),
\qquad
y= \frac{R(-\alpha s) x}{\sqrt{-t}},
\qquad 
s=-\log(-t),
\end{equation}
for a \emph{periodic-in-self-similar-time} profile $U\colon\RR^3 \times [0,S] \to\RR^3$, where the time-period is given by
\begin{equation}
\label{eq:RDSS:period}
S := 2 \log(\lambda) >0.
\end{equation}
\end{subequations}
Note that as $\lambda \to 1^+$ the profile $U$ formally becomes time-independent since $S \to 0^+$ in~\eqref{eq:RDSS:period}, and thus the ansatz~\eqref{eq:RDSS:ansatz} becomes the rotated self-similar ansatz~\eqref{eq:RSS}, with the same value of $\alpha$. Moreover, since $R(0) = {\rm Id}$, setting $\alpha =0$ in \eqref{eq:RDSS:total} recovers the discretely self-similar ansatz~\eqref{eq:DSS:ansatz}, with the same value of $\lambda$. In particular, in light of Remark~\ref{rem:RSS:DSS} we have the hierarchy (see also~\cite{BradshawTsai17,Tsai18})
\begin{equation*}
\text{SS}\subseteq \text{RSS}\subseteq \text{DSS}\subseteq \text{RDSS},
\end{equation*}
and hence Liouville-type theorems should be hardest to establish in the setting of rotated discretely self-similar solutions.

Note that if $u$ is a $(\alpha,\lambda)$-RDSS solution to 3D Navier-Stokes, then by substituting~\eqref{eq:RDSS:ansatz} into \eqref{eq:NSE}, we discover that the profile $U \colon \RR^3 \times [0,S] \to \RR^3$ is a time-periodic solution of
\begin{subequations}
\label{eq:profile:RDSS}
\begin{align}
\partial_s U + \alpha\big(JU - (Jy\cdot\nabla)U\big) + \tfrac12 U + \tfrac12 (y\cdot\nabla)U - \Delta U +  (U\cdot\nabla)U + \nabla P &= 0,
\label{eq:profile:RDSS:a}\\
\nabla\cdot U &= 0,
\label{eq:profile:RDSS:b}
\end{align}
\end{subequations}
on $\RR^3$, with period $S = 2 \log(\lambda)$ as in~\eqref{eq:RDSS:period}.  

A small modification of the adjoint weight construction from the proof of Theorem~\ref{thm:main} immediately leads to the following  generalization of Theorem~\ref{thm:main:DSS}:
\begin{theorem}[Main result for RDSS solutions]
\label{thm:main:RDSS}
Let $u$ be a solution of the 3D incompressible Navier-Stokes equations on $\RR^3 \times [-1,0)$, which satisfies the Type~I upper bound~\eqref{eq:type:I} for some $C_{U,0}>0$. Assume that there exist $\alpha \in \RR$,  $\lambda >1$, and $U\in C^2(\RR^3 \times [0,S])$ with $S= 2 \log(\lambda)$, such that $u$ is a backwards globally rotated discretely self-similar solution of~\eqref{eq:NSE} with profile $U$ (the RDSS ansatz~\eqref{eq:RDSS:ansatz} holds).  There exist $\underline{\alpha} = \underline{\alpha}(C_{U,0})>0$ and $\underline{\lambda} = \underline{\lambda} (C_{U,0}) >1$,  such that if $|\alpha|\leq \underline{\alpha}$ and $1 < \lambda < \underline{\lambda}$, then $U\equiv 0$. Moreover, there exist $\overline{\alpha} = \overline{\alpha}(C_{U,0})>0$ and $\underline{\lambda} = \underline{\lambda} (C_{U,0}) >1$,  such that if $|\alpha|\geq \overline{\alpha}$ and $1< \lambda< \underline{\lambda}^{\frac{1}{1+\alpha^2}}$, then $U\equiv 0$.
\end{theorem}

The proof of Theorem~\ref{thm:main:RDSS} is given in Sections~\ref{sec:DSS:RDSS}--\ref{sec:DSS:RDSS:Ben}. The arguments may be viewed as an extension of the proof of Theorem~\ref{thm:main}, with the additional knowledge that $\p_s U$ is not zero, but is instead \textit{small} in a suitable sense, if the period $S$ is assumed small (equivalently, if $\lambda$ is close to $1$). 

\begin{remark}[Navier-Stokes solitons and breathers]
\label{rem:solitons}
We are indebted to V.~\v{S}ver\'ak for pointing out the following connection. In the terminology of geometric flows, given an evolution $\partial_t u = F[u]$ with a symmetry group $G$ acting via $u\mapsto R(g)u$, a solution is called a \emph{soliton} if
$u(t)=R(g(t))\,U$ for a fixed profile $U$. RSS solutions are solitons of~\eqref{eq:NSE}
in this sense. The symmetry group of~\eqref{eq:NSE} is generated by scalings,
rotations, translations, and Galilean boosts, and the solitons (those with bounded velocity and decaying pressure) may be classified (up to conjugation by Galilean symmetries) as:
steady states; rotating waves $u(x,t)=R(\omega t)\,U(R(-\omega t)x)$; and the backward rotated self-similar solutions~\eqref{eq:RSS}, as well as their forward analogues. In particular,
RSS solutions are the general scaling solitons of~\eqref{eq:NSE}. Backward solitons are
conjecturally trivial (cf.~Conjecture~\ref{conj:Perelman}), while nontrivial forward
self-similar solutions exist for every $(-1)$-homogeneous initial datum which is
locally H\"older continuous away from the origin~\cite{JiaSverak14}, and play a central
role in the non-uniqueness program~\cite{JiaSverak15,GuillodSverak23}. Finally, the DSS and
RDSS solutions  are \textit{periodic} modulo $G$, i.e., the Navier-Stokes analogues of
\emph{breathers}, and Theorems~\ref{thm:main:DSS}--\ref{thm:main:RDSS} may be viewed as
conditional results of \emph{no shrinking breathers} type, in the spirit
of~\cite{Perelman02}.
\end{remark}

\subsection{No Type~I blowup if nearly self-similar on one time slice}
Thus far, the Liouville-type theorems in this paper concern \emph{globally} backwards self-similar solutions of the 3D Navier-Stokes equations, i.e., solutions for which the self-similar ansatz holds on all of $\RR^3$; see~\eqref{eq:SS} for SS solutions, \eqref{eq:RSS} for RSS solutions, \eqref{eq:DSS:total} for DSS solutions, and~\eqref{eq:RDSS:total} for RDSS solutions. The method of proof we have employed is however very flexible, and applies with minor modifications to solutions with \emph{no exact self-similarity assumption}, neither local, nor global. 

In order to highlight this proof strategy, we prove a regularity criterion of \emph{local} type. We assume only that the solution satisfies a Type~I upper bound in a unit parabolic cylinder, and we assume that there is a single time (prior to blowup) at which the solution is locally \emph{approximately} backward self-similar (cf.~\eqref{eq:Ds:bound:xt} below); we then conclude that the space-time (top) center of the parabolic cylinder is a regular point\footnote{We say that $(0,0)$ is a \emph{regular point} of the solution $u$ of~\eqref{eq:NSE} if $u\in L^\infty(B_r\times(-r^2,0))$ for some $r>0$, cf.~\cite{CKN82}.} of 3D Navier-Stokes. Our main result is Theorem~\ref{thm:local:theorem} below. Note that no exact self-similarity (neither local nor global) is assumed.

\begin{theorem}[Approximately self-similar Type~I solutions are regular]
\label{thm:local:theorem}
Let $(u,p)$ be a smooth solution to \eqref{eq:NSE} on $B_1\times[-1,0)$ satisfying the Type~I bound
\begin{equation}
\label{eq:local:typeI}
|u(x,t)|\leq \frac{C_u }{\sqrt{-t}+|x|},
\qquad (x,t)\in B_1 \times [-1,0), 
\end{equation}
for some constant $C_u  >0$. Assume that the pressure is bounded far from $(0,0)$; that is, assume there exists a constant $C_p>0$ such that the pressure satisfies the bound
\begin{equation}
\label{eq:local:pbound}
|p(x,t)|\leq C_p,\qquad (x,t)\in A\times [-1,0), 
\qquad A:= \{\tfrac{1}{2}< |x|< \tfrac{3}{4}\}.
\end{equation}
There exist $\delta_0=\delta_0(C_u )\in(0,1]$ and $s_0=s_0(C_u ,C_p)\geq 1$ such that if, at a single time $\bar t\in(-e^{-s_0},0)$ we have the bound
\begin{equation}
\label{eq:Ds:bound:xt}
\sqrt{-\bar t}\,\bigl\| (-\bar t)\,\partial_t u(\cdot,\bar t) - \tfrac12 u(\cdot,\bar t) - \tfrac12 (x\cdot\nabla) u(\cdot,\bar t) \bigr\|_{L^\infty(B_1)}
\leq \delta_0 ,
\end{equation}
then $(0,0)$ is a regular point.
\end{theorem}

\begin{remark}[The time derivative in self-similar coordinates]
The operator $(-t) \partial_t - \frac 12 - \frac 12 x \cdot \nabla$ appearing in~\eqref{eq:Ds:bound:xt} is $-\tfrac12$ times the generator of the self-similar scaling $u\mapsto\lambda\,u(\lambda x,\lambda^2 t)$ at $\lambda=1$, and so measures the deviation of $u$ from backward self-similarity. In fact,~\eqref{eq:Ds:bound:xt}   is exactly an assumption on the self-similar time derivative of the self-similar velocity. Indeed, upon defining
\[
U(y,s) := \sqrt{-t}\, u(x,t), \qquad \mbox{with} \qquad x = \sqrt{-t}\,y,\quad t = -e^{-s},
\]
we have that
\begin{equation}
\label{eq:dsU:xt}
\partial_s U(y,s) = \sqrt{-t}\,\bigl((-t)\partial_t u - \tfrac12 u - \tfrac12 (x\cdot\nabla)u\bigr)(x,t).
\end{equation}
Since $\{|x|<1\}$ corresponds to $\{|y|<e^{s/2}\}$, condition~\eqref{eq:Ds:bound:xt} at $\bar t = -e^{-\bar s}$ is precisely
\begin{equation}
\label{eq:Ds:bound}
\|\partial_sU(\cdot,\bar s)\|_{L^{\infty}(B_{e^{\bar s/2}})} \leq \delta_0 ,
\qquad 
\mbox{for some} \qquad \bar s\geq s_0.
\end{equation}
In the proof of Theorem~\ref{thm:local:theorem} we work with the restatement~\eqref{eq:Ds:bound} of assumption~\eqref{eq:Ds:bound:xt}.
\end{remark}

\begin{remark}[A weaker form of \emph{approximate self-similarity}]
\label{rem:weak:norm}
The $L^\infty$-based assumption in~\eqref{eq:Ds:bound} enters the proof only once, in estimate~\eqref{eq:goodie:3}. The entire proof of Theorem~\ref{thm:local:theorem} goes through without change if~\eqref{eq:Ds:bound} is replaced by the weaker assumption
\[
\int_{B_{e^{\bar s/2}}}|\partial_sU(y,\bar s)|\, (1+|y|)\,e^{-\frac{1}{8}|y|^2}dy \les \delta_0 ,
\]
for some universal implicit constant in the $\les$ symbol.
\end{remark}

The proof of Theorem~\ref{thm:local:theorem} localizes the weighted Bernoulli-function estimates of Sections~\ref{sec:Bernoulli:original}--\ref{sec:alpha:small} by using the finite radius principal eigenfunction from Lemma~\ref{lem:lambda:conv}; this converts hypothesis~\eqref{eq:Ds:bound} into smallness of the local enstrophy of the profile at the given time, which is then propagated to all later times (Lemma~\ref{lem:local:propagation}) and fed into the Caffarelli-Kohn-Nirenberg $\eps$-regularity criterion~\cite{CKN82}.

\subsection{Organization of the paper}
Sections~\ref{sec:pressure:and:derivative:bounds}--\ref{sec:Bernoulli:original} gather basic properties of RSS solutions. The Liouville-type results for RSS solutions claimed in Theorem~\ref{thm:main} are proven in Section~\ref{sec:alpha:small} ($|\alpha|\ll 1$) and Section~\ref{sec:alpha:large} ($|\alpha|\gg1$). The Liouville-type results for RDSS (and hence also DSS) solutions claimed in Theorem~\ref{thm:main:RDSS} are analyzed in Section~\ref{sec:DSS:RDSS} ($|\alpha|\ll 1$) and Section~\ref{sec:DSS:RDSS:Ben} ($|\alpha|\gg1$). Finally, in Section~\ref{sec:local} we prove the local regularity criterion of Theorem~\ref{thm:local:theorem}.
 
\section{Estimates for the pressure and higher derivatives of the velocity}
\label{sec:pressure:and:derivative:bounds}
We work within the class of solutions to 3D Navier-Stokes which are rotated backward self-similar, meaning that ansatz~\eqref{eq:RSS} holds, and which satisfy the Type~I bound~\eqref{eq:type:I}, which is equivalent to the scaling-critical upper bound~\eqref{eq:decay}. By repeating the arguments in~\cite{NRS96,Tsai98}, we show that~\eqref{eq:type:I} implies Type~I estimates for higher space derivatives of the velocity, and estimates for the pressure. Importantly, \emph{all these bounds are independent of $\alpha$}. We record:
\begin{lemma}\label{lem:gradient}
Let $u$ be a solution of the 3D incompressible Navier-Stokes equations~\eqref{eq:NSE} on $\RR^3 \times [-1,0)$, which satisfies the Type~I upper bound~\eqref{eq:type:I} for some $C_{U,0}>0$. Assume that there exists $U\in C^2(\RR^3)$ and $\alpha\in \RR$ such that $u$ is given by the  ansatz~\eqref{eq:RSS}. Then $U\in C^\infty(\RR^3)$, and there exist \emph{$\alpha$-independent} constants $C_{U,1} = C_{U,1}(C_{U,0})>0$ and $C_{U,2} = C_{U,2}(C_{U,0})>0$ such that
\begin{equation}\label{eq:decay:grad}
|\nabla U(y)|\le \frac{C_{U,1}}{1+|y|^{2}}, \qquad  
|\nabla^2 U(y)|\le \frac{C_{U,2}}{1+|y|^{3}},
\end{equation}
for all $y\in\RR^3$.
Moreover, for any $\sigma \in (0,1)$, there exists an \emph{$\alpha$-independent} constant $C_{P,0} = C_{P,0}(C_{U,0},\sigma) > 0$ such that 
\begin{equation}\label{eq:decay:pressure}
|P(y)|\le \frac{C_{P,0}}{1+|y|^{2-\sigma}},  
\end{equation}
for all $y\in\RR^3$.
\end{lemma}
This result is essentially known, and follows mutatis mutandis from~\cite{NRS96,Tsai98,Lemarie02,Tsai18}. We repeat here a sketch of the proof to emphasize the $\alpha$-independence of the constants $C_{U,1}, C_{U,2}$, and $C_{P,0}$. Note that at time $t=-1$ we have $s=0$ and $R(0) = {\rm Id}$; thus, the initial data $u_0 = u|_{t=-1}$ does not know about the rotation parameter $\alpha$. The Navier-Stokes equations~\eqref{eq:NSE} also do not contain an $\alpha$-dependent term. Thus, it is convenient to prove parts of Lemma~\ref{lem:gradient} in original $(x,t)$ variables.

\begin{proof}[Proof of Lemma~\ref{lem:gradient}]
The bound~\eqref{eq:decay} implies $U \in L^\infty_{\rm loc}(\RR^3) \cap L^q(\RR^3)$ for $q>3$. 
Equation~\eqref{eq:profile:a} may be written as $-\Delta U  + \nabla P = \nabla \cdot F$, $\nabla \cdot U = 0$ with $F \in L^{\infty}_{\rm loc}$. The fact that $U,P \in C^\infty(\mathbb{R}^3)$ follows from standard interior estimates for the Stokes system~\cite[Lemma 2.12]{Tsai18}, differentiating the equation, and iterating. The bounds such obtained are qualitative in nature, as they may a-priori depend on $\alpha$.

Note that $\nabla \cdot  (JU - (Jy\cdot\nabla)U ) =0$ and $\nabla \cdot (U + (y\cdot \nabla U)) = 0$ when $\nabla \cdot U = 0$, and hence $-\Delta P =  \partial_i \partial_j (U^i U^j)$.  From~\cite[Lemma 2.1]{Tsai18} (see also~\cite[Lemma 3.1]{NRS96}) it follows that the pressure $P$ in \eqref{eq:profile:a} is given by $P = R_i R_j (U^i U^j) \; +$ a constant, where $R_i = \partial_i (-\Delta)^{-1/2}$ are Riesz transforms, and we have used the summation convention on repeated indices. Without loss of generality, we take this constant term to vanish, so that $P = R_i R_j (U^i U^j)$; this directly implies $P \in L^{q/2}(\mathbb{R}^3)$ for $q>3$ with an upper bound that only depends on $q$ and $C_{U,0}^2$ (and is hence independent of $\alpha$), by the boundedness of Calderon-Zygmund operators on (non-endpoint) Lebesgue spaces.

The bounds~\eqref{eq:decay:grad} follow as in~\cite{NRS96,Tsai98} by rescaling to an appropriate parabolic cylinder in original $(x,t)$ variables, using the fact $(U,P) \in L^q(\RR^3) \times L^{q/2}(\RR^3)$ implies the smallness of the mean of $|u|^3 + |p|^{3/2}$ if the center of this cylinder is far enough, and then applying the Caffarelli-Kohn-Nirenberg~\cite{CKN82} partial regularity theory (see~\cite{Kukavica09} for a compact summary).

An alternative argument is as follows.
Fix $y\in\RR^3$ arbitrary and set $r=\frac12\max\{|y|,1\}$. Then the function\footnote{Note that although $u$ is a priori a solution of Navier-Stokes only on $\RR^3\times [-1,0)$, the ansatz~\eqref{eq:RSS} allows us to extend it backwards-in-time as a solution of Navier-Stokes on $\RR^3\times [-T,0)$ for any $T\geq 1$. In particular, setting $T = 1 + r^2$ makes $\tilde u$ well defined on $Q_1$.} $\tilde u(\tilde x, \tilde t):=r \,u(y+r \tilde x,\,-1+r^2 \tilde t)$ solves~\eqref{eq:NSE} on the unit parabolic cylinder $Q_1:=B_1\times(-1,0\,]$.
Moreover, the bound~\eqref{eq:type:I} implies that $\|\tilde u\|_{L^\infty(Q_1)}\le C_{U,0}$, a bound which is independent of $y$ and of $\alpha$. By interior regularity for the Navier-Stokes system,\footnote{\label{foot:int:reg}One can follow the proof in the classical paper of Serrin~\cite{Serrin62} and make it quantitative, resulting in a polynomial dependence of $C_1$ and $C_2$ on $C_{U,0}$. Note that no pressure bounds appear in this interior estimate. Equivalently, one may iterate interior regularity for the non-stationary Stokes system, see e.g.~\cite[Lemma A.2]{ChenStrainTsaiYau08}.} there exist  constants $C_1, C_2 >0$ which \emph{depend only on $C_{U,0}$}, such that $|(\nabla_{\tilde x} \tilde u)(0,0)|\le C_1$ and $|(\nabla^2_{\tilde x} \tilde u)(0,0)|\le C_2$. We undo the definition of $\tilde u$ to deduce that $| (\nabla_x u)(y,-1) | \leq C_1 r^{-2} $ and $|(\nabla^2_x u)(y,-1)| \leq C_2 r^{-3} $. Since at $t=-1$ we have $s=0$ and $R(0) = {\rm Id}$, we may then use~\eqref{eq:RSS} to equate 
$(\nabla U)(y)=(\nabla_x u)(y,-1)$ and $\nabla^2 U(y) = (\nabla_x^2 u)(y,-1)$. The bounds~\eqref{eq:decay:grad} now follow by the definition of $r$, with $C_{U,1} = 8 C_1$ and $C_{U,2} = 16 C_2$, two constants that depend only on $C_{U,0}$ and are thus independent of $\alpha$.

The bound~\eqref{eq:decay:pressure} follows from the singular integral kernel representation of $P = R_i R_j (U^i U^j)$, and the previously established bounds for $U$ and $\nabla U$. We have that $C_{P,0}$ depends only on $\sigma, C_{U,0}$, and $C_{U,1}$, and is thus independent of $\alpha$.
\end{proof}

\section{Small enstrophy implies zero velocity}
\label{sec:vorticity}
We reduce the proof of Theorem~\ref{thm:main} to verifying that there exists a large enough radius $\bar R>0$, such that if $\| \Omega \|_{L^2(B_{\bar R}(0))}$ is sufficiently small (independently of $\alpha$), then $U\equiv 0$. This observation is summarized as:
\begin{proposition}
\label{prop:local:vort}
Let $(U,P)$ be a smooth solution of~\eqref{eq:profile} with associated vorticity $\Omega = \nabla \times U$. Let $\bar R := \sqrt{8 C_{U,1}}>0$. 
There exists a \emph{universal} constant $C_{\Omega}   >0$  such that 
\begin{equation}
\label{eq:vort:est}
\|\Omega\|_{L^2(\RR^3)}^2+\|\nabla\Omega\|_{L^2(\RR^3)}^2
\leq C_{\Omega} \|\Omega\|_{L^2(B_{\bar R})} 
\bigl(\|\Omega\|^2_{L^2(\RR^3)}+\|\nabla \Omega\|^2_{L^2(\RR^3)}\bigr) .
\end{equation}
In particular, if 
\begin{equation}
\label{eq:small:enstrophy}
\|\Omega\|_{L^2(B_{\bar R})} < C_{\Omega}^{-1},
\end{equation} 
then $U\equiv 0$.
\end{proposition}
Note that both $C_{\Omega}$ and $\bar R$ are independent of $\alpha$. Also, note that~\eqref{eq:small:enstrophy} automatically holds if $6 C_{U,1} < C_{\Omega}^{-4/7}$, in light of~\eqref{eq:decay:grad}; this rules out the existence of nontrivial \emph{small solutions}  of~\eqref{eq:profile}.
Lastly, note that the smallness condition~\eqref{eq:small:enstrophy} may be re-stated in original $(x,t)$ variables as follows. Undoing the rescaling~\eqref{eq:RSS} gives $\omega(x,t) = (-t)^{-1} R(\alpha s) \Omega (y)$ with $y = R(-\alpha s)x/\sqrt{-t}$, so that $|\omega(x,t)| = (-t)^{-1}|\Omega(y)|$. The change of variables $x = \sqrt{-t} R(\alpha s) y$ then yields $\|\Omega\|_{L^2(B_{\bar R})}^2
= (-t)^{1/2} \int_{|x|\le \bar R\sqrt{-t}} |\omega(x,t)|^2  dx$, for all $t\in(-1,0)$. The right-hand side is invariant under the Navier-Stokes scaling $u\mapsto \lambda\,u(\lambda\,\cdot\,,\lambda^2\,\cdot\,)$, and its smallness over the parabolic balls $\{|x|\le\bar R\sqrt{-t}\}$ shrinking to the origin as $t \to 0^-$ is \textit{reminiscent of} the scale-invariant hypotheses under which the partial regularity theory of Caffarelli, Kohn and Nirenberg~\cite{CKN82} implies that $(0,0)$ is a \emph{regular} point of $u$.

\begin{proof}[Proof of Proposition~\ref{prop:local:vort}]
Taking the curl of~\eqref{eq:profile:a}, using the fact that when $\nabla \cdot U=0$ we have $\nabla \times (JU) = - \partial_3 U$, $[\nabla \times, J y]\cdot \nabla U = - \nabla U_3$, and $J \Omega = \nabla U_3 - \partial_3 U$, we deduce
\begin{equation}
\label{eq:vort}
\alpha \bigl( J\Omega-  Jy\cdot\nabla\Omega\bigr) + \Omega + \tfrac{1}{2}y\cdot\nabla\Omega -\Delta\Omega +U\cdot\nabla \Omega=\Omega\cdot\nabla U.   
\end{equation}
By~\eqref{eq:decay:grad} we have that $|\Omega(y)|\leq \frac{2 C_{U,1}}{1+|y|^2}$, and thus $\Omega \in L^2(\RR^3)$. We may thus take the dot product of~\eqref{eq:vort} with $\Omega$ and integrate over $\RR^3$, use the fact that $\nabla \cdot (J y) = 0$ and $J \Omega \cdot \Omega =0 $,  to obtain\footnote{In two space dimensions the vortex stretching term is absent from~\eqref{eq:vort}, and so we immediately deduce $\frac{1}{2} \| \Omega\|_{L^2(\RR^2)}^2 +\| \nabla \Omega\|_{L^2(\RR^2)}^2 = 0$, without any further assumption on $|\alpha|$ as being small or large.}
\[
\tfrac{1}{4} \| \Omega\|_{L^2(\RR^3)}^2 +\| \nabla \Omega\|_{L^2(\RR^3)}^2 = \int_{\RR^3} \Omega^i \partial_i U^j \Omega^j \, dy.
\]
By~\eqref{eq:decay:grad}, when $|y|\geq \bar R = \sqrt{8 C_{U,1}}$, we have $|\nabla U(y)|\leq 1/8$. We  obtain from the above identity that 
\begin{align}
\tfrac{1}{8} \| \Omega\|_{L^2(\RR^3)}^2 +\| \nabla \Omega\|_{L^2(\RR^3)}^2 
&\leq \int_{|y| < \bar R} |\Omega|^2 \, |\nabla U|\, dy
\notag\\
& \leq \|\Omega\|_{L^2(B_{\bar R})} \|\Omega\|_{L^4(\RR^3)} \|\nabla U\|_{L^4(\RR^3)}
\notag\\
&\leq C^\prime \|\Omega\|_{L^2(B_{\bar R})} \|\Omega\|_{L^4(\RR^3)}^2
\notag \\
&\leq C^\prime \|\Omega\|_{L^2(B_{\bar R})} \|\Omega\|_{L^2(\RR^3)}^{1/2} \|\Omega\|_{L^6(\RR^3)}^{3/2}
\notag\\
&\leq C^\prime (C^{\prime\prime})^{3/2} \|\Omega\|_{L^2(B_{\bar R})} \|\Omega\|_{L^2(\RR^3)}^{1/2} \|\nabla \Omega\|_{L^2(\RR^3)}^{3/2}
\end{align}
where $C^\prime>0$ (the operator norm of $\nabla \nabla \times (-\Delta)^{-1} \colon L^4(\RR^3) \to L^4(\RR^3)$) and $C^{\prime\prime} >0$ (the constant appearing in the Gagliardo-Nirenberg-Sobolev inequality associated to $\dot{H}^1(\RR^3) \subset L^6(\RR^3)$) are universal constants. The bound~\eqref{eq:vort:est} now directly follows with $C_{\Omega} = 6 C^\prime (C^{\prime\prime})^{3/2}$.

Finally, if $\Omega \equiv 0$, then since $\nabla \cdot U = 0$ we have that $U$ is a harmonic function which vanishes at infinity by \eqref{eq:decay}; thus $U\equiv 0$. 
\end{proof}

\section{The equation satisfied by the Bernoulli function}
\label{sec:Bernoulli:original}
As in~\cite{NRS96}, define the head-pressure, or \emph{Bernoulli function}, by
\begin{equation}
\label{eq:Bernoulli:def}
\Pi(y) 
:= P(y) + \tfrac 12 |U(y)|^2 + \tfrac 12 y \cdot U(y).
\end{equation}
Note that~\eqref{eq:decay} and \eqref{eq:decay:pressure} imply
\[
|\Pi(y)|\leq C_{P,0} + \tfrac 12 C_{U,0}^2 +  \tfrac 12   C_{U,0}
\]
for all $y\in \RR^3$.
Also, define 
\begin{equation}
L:=  -\Delta + (U(y) + \tfrac 12 y) \cdot \nabla
\label{eq:L:def}
\end{equation}
which is a non-normal elliptic operator without potential term.

While in the case $\alpha =0$ we know that $L \Pi = -|\Omega|^2$ (see~\eqref{eq:Bernoulli:ineq} above), as soon as $\alpha\neq 0$ this identity fails. Instead, for general $\alpha \in \RR$,  if $(U,P)$ solves~\eqref{eq:profile} then
\begin{align}
L\Pi + |\Omega|^2
&= -\Delta P - U \cdot \Delta U - |\nabla U|^2 - \tfrac 12 y \cdot \Delta U
+ (U+\tfrac 12  y) \cdot \nabla P + |\Omega|^2
\notag\\
&\quad + \bigl( (U +\tfrac 12  y) \cdot \nabla U \bigr) \cdot U
+ \tfrac 12 (U+\tfrac 12  y) \cdot U + \tfrac 12 \bigl(  (U+\tfrac 12  y) \cdot \nabla U \bigr)\cdot y
\notag\\
&= {\rm Tr} \bigl( (\nabla U)^2 \bigr) - (U +\tfrac 12 y) \cdot \Delta U - |\nabla U|^2 
+ |\Omega|^2 
\notag\\
&\quad
- (U+\tfrac 12  y) \cdot \bigl( \alpha J U - \alpha (J y \cdot \nabla)U + \tfrac 12 U   - \Delta U + (U + \tfrac 12 y) \cdot\nabla  U \bigr) 
\notag\\
&\quad + \bigl( (U +\tfrac 12  y) \cdot \nabla U \bigr) \cdot U
+ \tfrac 12 (U+\tfrac 12  y) \cdot U + \tfrac 12 \bigl(  (U+\tfrac 12  y) \cdot \nabla U \bigr)\cdot y
\notag\\
&= {\rm Tr} \bigl( (\nabla U)^2 \bigr) - |\nabla U|^2  + |\Omega|^2
- \alpha (U+\tfrac 12  y) \cdot \bigl(   J U -   (J y \cdot \nabla)U  \bigr) 
\notag\\
&= 
- \alpha (U+\tfrac 12  y) \cdot \bigl(   J U -   (J y \cdot \nabla)U  \bigr) 
\notag\\
&= \alpha E,
\label{eq:L:Pi:alpha}
\end{align}
where the \emph{error term} $E$ is defined as
\begin{equation}
E :=  
\tfrac{1}{2} (y_1 U_2 - y_2 U_1) +  \left(U(y) + \tfrac{1}{2} y \right) \cdot \left(y_1 \partial_2 U - y_2 \partial_1 U \right)
\,.
\label{eq:E:def}
\end{equation}
Note that when $\alpha=0$, identity~\eqref{eq:L:Pi:alpha} recovers~\eqref{eq:Bernoulli:ineq}. 
The main properties of the error term $E$ are:
 
\begin{lemma}
\label{lem:E:properties}
Let $U$ be divergence-free, and define $E$ via~\eqref{eq:E:def}. Then, we have that 
\begin{equation}
|E(y)| \leq C_E:= C_{U,0} + C_{U,1} + 2 C_{U,0} C_{U,1} ,
\label{eq:E:bnd}
\end{equation} 
for all $y \in \RR^3$, and also
\begin{equation}
E=\tfrac{1}{2}\partial_{\theta}(|U|^2+U\cdot y),
\label{eq:E:is:theta:derivative}
\end{equation}
where $\partial_\theta := y_1\partial_2 - y_2\partial_1$ denotes the angular derivative about the $y_3$-axis. 
\end{lemma}
The observation that $E$ is a pure $\p_\theta$ derivative~\eqref{eq:E:is:theta:derivative}  is crucial for dealing with the large $\alpha$ regime below.

\begin{proof}[Proof of Lemma~\ref{lem:E:properties}]
The bound~\eqref{eq:decay} and \eqref{eq:decay:grad} imply 
\[
|E(y)| \leq C_{U,0} + C_{U,1} + \tfrac{2}{1+|y|^2} C_{U,0} C_{U,1}.
\]
The bound~\eqref{eq:E:bnd} then follows.
Next, a direct computation gives
\[
\tfrac{1}{2}\partial_\theta(|U|^2) 
= U\cdot(y_1\partial_2 U - y_2\partial_1 U),
\]
and
\begin{align*}
\tfrac{1}{2}\partial_\theta(U\cdot y) 
&= \tfrac{1}{2}\bigl((y_1\partial_2 U - y_2\partial_1 U)\cdot y + U\cdot(y_1 e_2 - y_2 e_1)\bigr)
\notag\\
&= \tfrac{1}{2}(y_1 U_2 - y_2 U_1) + \tfrac{1}{2}\,y\cdot(y_1\partial_2 U - y_2\partial_1 U).
\end{align*}
Adding these two identities yields
\[
\tfrac{1}{2}\partial_\theta(|U|^2 + U\cdot y) 
= \tfrac{1}{2}(y_1 U_2 - y_2 U_1) + \bigl(U + \tfrac{y}{2}\bigr)\cdot(y_1\partial_2 U - y_2\partial_1 U) = E,
\]
which is the identity claimed in~\eqref{eq:E:is:theta:derivative}.
\end{proof}

\section{RSS: proof of the main result for $\alpha$ small}
\label{sec:alpha:small}
The goal is to prove that if $|\alpha|$ is small enough, then~\eqref{eq:small:enstrophy} holds. We note that the arguments given in this section are new even for $\alpha=0$.

The main technical ingredient is the following observation concerning the kernel of the operator 
\begin{equation}
L^* := - \Delta - (U(y) + \tfrac 12 y) \cdot \nabla - \tfrac 32
\label{eq:L*:def}
\end{equation} 
which is the $L^2$ adjoint of the operator $L$ introduced earlier in~\eqref{eq:L:def}.

\begin{proposition}
\label{prop:weight:construction}
Let $U$ be smooth, divergence-free, which satisfies the bounds~\eqref{eq:decay} and~\eqref{eq:decay:grad}. Given  $\eps \in (0,1)$  there exist constants $m = m(\eps, C_{U,0})>0$ and $M = M(\eps, C_{U,0})>0$, and a $C^2$ smooth strictly positive function $w \colon \RR^3 \to \RR_+$ satisfying 
\begin{equation}
 L^* w = 0
 \label{eq:weight:kernel}
\end{equation}
and  
\begin{equation}
m \, e^{-\frac{1+\eps}{4} |y|^2} 
\leq w(y) \leq 
M \, e^{-\frac{1-\eps}{4} |y|^2}, 
\qquad \mbox{for all} \qquad y \in \RR^3.
\label{eq:weight:decay}
\end{equation}
\end{proposition}

Assuming that Proposition~\ref{prop:weight:construction} holds, we give the proof of Theorem~\ref{thm:main} for $|\alpha| \ll 1$.
\begin{proof}[Proof of Theorem~\ref{thm:main}: the small $|\alpha|$  case]
We multiply~\eqref{eq:L:Pi:alpha} by $w$ and integrate over $\RR^3$, an operation justified by the Gaussian-type decay of the weight function $w$ established in~\eqref{eq:weight:decay},\footnote{Technically speaking, to justify this computation we'd also need to have some mild decay information on $\nabla w$, such as $\lim_{R\to\infty} R^{-1} \int_{R\leq |y| \leq 2 R} |\nabla w| dy =0$. The weight $w$ satisfies this property; in fact, we could work a bit harder in the proof of Proposition~\ref{prop:weight:construction} and establish pointwise Gaussian-type upper bounds for $|\nabla w|$, analogously to the upper bound in~\eqref{eq:weight:decay}. The aforementioned mild decay property, may however be obtained via the Caccioppoli inequality for $L^* w = 0$, which together with~\eqref{eq:weight:decay} yields $\| \nabla w\|_{L^2(R\leq |y|\leq 2R)} \leq C(C_{U,0}) \| w\|_{L^2(R/2\leq |y|\leq 4R)} \leq C^\prime(C_{U,0},\eps) R^{\frac 32} e^{-\frac{1-\eps}{16} R^2}$, for all $R\geq 1$.} to obtain
\begin{align}
\int_{\RR^3} |\Omega|^2 \, w \, dy
&= - \int_{\RR^3} L \Pi \,   w \, dy
+ \alpha \int_{\RR^3} E \, w \, dy
\notag\\
&= - \int_{\RR^3}\Pi \,   L^* w  \, dy
+ \alpha \int_{\RR^3} E \, w \, dy
= \alpha \int_{\RR^3} E \, w \, dy.
\label{eq:vorticity:error}
\end{align}
In the last equality we have appealed to~\eqref{eq:weight:kernel}. Then, using~\eqref{eq:E:bnd} and~\eqref{eq:weight:decay} with $\eps = \frac 12$, we obtain
\begin{align*}
\|\Omega\|_{L^2(B_{\bar R})}^2
&\leq \frac{1}{\inf_{B_{\bar R}} w } \int_{B_{\bar R}} |\Omega|^2 \, w \, dy
\notag\\
&\leq m^{-1}  e^{\frac 38 \bar R^2} \cdot |\alpha| \, \|E\|_{L^\infty(\RR^3)} \|w\|_{L^1(\RR^3)}
\leq|\alpha| \, \cdot \bigl( m^{-1}  e^{\frac 38 \bar R^2}  C_E \, 200 \, M  \bigr)
.
\end{align*}
Note that $m,M,C_E,\bar R$ only depend on $C_{U,0}$ and $C_{U,1}= C_{U,1}(C_{U,0})$, these constants are independent of $\alpha$. Moreover, the constant $C_{\Omega}$ appearing in~\eqref{eq:small:enstrophy} also does not depend on $\alpha$. Therefore,   letting
\[
\underline{\alpha}:=\tfrac{1}{200} m  e^{- \frac 38 \bar R^2}  C_E^{-1} \, M^{-1} \,  C_{\Omega}^{-2} ,
\]
we obtain that for $|\alpha| < \underline{\alpha}$ we have $\|\Omega\|_{L^2(B_{\bar R})} < C_{\Omega}^{-1}$. Proposition~\ref{prop:local:vort} then implies $U\equiv 0$.
\end{proof}

\begin{remark}[A probabilistic interpretation]
\label{rem:probabilistic} 
Since $\nabla \cdot (U + \tfrac12 y) = \frac 32$, the equation $L^* w = 0$ may be
rewritten as $\Delta w + \nabla\cdot((U+\frac12 y)\,w) = 0$, which is the stationary
Fokker-Planck equation for the diffusion ${\rm d}X_t = -(U(X_t) + \frac12 X_t)\,{\rm d}t + \sqrt{2}\,{\rm d}B_t$. When $U \equiv 0$ this is the Ornstein-Uhlenbeck process, whose invariant density (suitably normalized) is precisely the Gaussian $\mu$
from~\eqref{eq:L2:mu:def}. Proposition~\ref{prop:weight:construction}
states that this equilibrium survives, as a deformation, under a divergence-free drift
perturbation $U$ vanishing at infinity. The balance between the drift and the diffusion is altered significantly only on a
compact set, and at infinity only through an arbitrarily small loss $\eps$ in the Gaussian rate. The existence of the invariant probability measure may also be obtained by
probabilistic arguments, via the Lyapunov function $V(y) = |y|^2$, which satisfies
$\Delta V - (U + \tfrac12 y)\cdot \nabla V \leq 6 + 2 C_{U,0} - |y|^2 < 0$ outside a
compact set, in the sense of Khasminskii~\cite{Khasminskii12}; see
also~\cite{BogachevKrylovRockner15}. Here we give a self-contained construction via standard elliptic theory, to emphasize the two-sided pointwise Gaussian bounds~\eqref{eq:weight:decay}, with
constants that depend only on $\eps$ and $C_{U,0}$.
\end{remark}

The remainder of this section is dedicated to the proof of Proposition~\ref{prop:weight:construction}, which clearly did the heavy lifting in the above argument.  
The argument consists of the following steps. First, along an increasing sequence of radii $R_k \to \infty$, we construct strictly positive smooth eigenfunctions $w_{R_k}$ for the  operator $L^*$ on the ball $B_{R_k}$, with homogeneous Dirichlet boundary conditions. We show that the corresponding eigenvalue $\lambda_{R_k}$ is non-negative and converges to zero as $k\to\infty$; see Lemma~\ref{lem:lambda:conv}. Second, we show by means of a barrier argument that after a suitable normalization, namely $w_{R_k}(0)=1$, the sequence $w_{R_k}$ satisfies the Gaussian upper bound stated in~\eqref{eq:weight:decay} on each $B_{R_k}$, uniformly as $R_k \to\infty$; see Lemma~\ref{lem:Gaussian:upper}. Third, we pass to the limit (in say, $C^2_{loc}$) to obtain the existence of a smooth positive function $w$ which lies in the kernel of $L^*$ and satisfies the same Gaussian upper bound; see Lemma~\ref{lem:pass:R:infty}. Fourth, using a second barrier argument we show that  $w$ satisfies the Gaussian lower bound stated in~\eqref{eq:weight:decay}; see Lemma~\ref{lem:Gaussian:lower}. Together, these  yield the proof of Proposition~\ref{prop:weight:construction}, which is given at the end of this section.

\begin{lemma}[Existence and basic properties of the principal eigenpair]
\label{lem:lambda:conv}
For every $R>0$ the homogeneous Dirichlet eigenvalue problem
\begin{subequations}
\label{eq:Dirichlet:EV:L*}
\begin{alignat}{2}
L^*w_R &= \lambda_R w_R,\qquad && \mbox{in}\qquad B_R,
\label{eq:Dirichlet:EV:L*:a}\\
w_R &= 0,\qquad && \mbox{on}\qquad \partial B_R,
\end{alignat}
\end{subequations}
admits a principal eigenpair $(\lambda_R,w_R)$, normalized with $w_R(0)=1$, such that:
\begin{enumerate}[label=(\roman*)]
\item $w_R\in C^\infty(B_R) \cap C(\overline{B_R})$ and $w_R>0$ in $B_R$;
\item $\lambda_R\in\RR$, and every other eigenvalue $\sigma\in\mathrm{spec}(L^*)$ satisfies $\mathrm{Re}(\sigma)\ge\lambda_R$;
\item $\lambda_R\ge 0$;
\item there exists $R_*=R_*(C_{U,0})>0$ such that $\lambda_R\le (R_*/ R)^{2}$ for every $R\ge R_*$.
\end{enumerate}
\end{lemma}
\begin{proof}[Proof of Lemma~\ref{lem:lambda:conv}]
We recall a few facts concerning the spectrum of non-symmetric elliptic operators, that sometimes go  under the umbrella name of  Krein-Rutman or Donsker-Varadhan theory~\cite{BNV94,GilbargTrudinger98,Evans10}. Let $\mathcal{D} \subset \RR^3$ be an open  bounded  connected domain, with $\partial  \mathcal{D}$ that is smooth. Let $\mathcal{L}:= - a^{ij} \partial_{ij} + b^i \partial_i + c$ be a uniformly elliptic operator with smooth symmetric coefficients $a=\{a^{ij}\}$, with smooth and bounded drift $b=\{b^i\}$ and potential $c$; assume furthermore that $c\geq 0$ in $\mathcal{D}$. Consider the spectrum of $\mathcal{L}$ with homogeneous Dirichlet boundary conditions on $\partial \mathcal{D}$.
It is well-known (see~\cite[Theorem 3, \S~6.5]{Evans10}) that $\mathcal{L}$ admits a principal eigenvalue $\lambda_1 \in \RR$, such that if $\sigma \in \mathbb{C}$ is any other eigenvalue of $\mathcal{L}$, we have $\mathrm{Re}(\sigma) \geq \lambda_1$. The corresponding eigenfunction $w$ is positive within $\mathcal{D}$, and the eigenvalue $\lambda_1$ is simple, so that $w$ is uniquely selected by normalizing $w(x_0) =1$ for some $x_0\in \mathcal{D}$. Moreover, the principal eigenfunction is smooth up to the boundary of the domain~\cite{GilbargTrudinger98}. The principal eigenvalue of $\mathcal{L}$ and its formal adjoint $\mathcal{L}^*$ (taken with homogeneous Dirichlet boundary conditions) agree, namely $\lambda_1(\mathcal{L}) = \lambda_1(\mathcal{L}^*)$. 
Lastly, if $\mu \in \RR$, then the operator $\mathcal{L}$ and the shifted operator $\mathcal{L}_{\mu}:= \mathcal{L} + \mu {\rm Id}$ have the same (normalized) positive principal eigenfunction, and $\lambda_1(\mathcal{L}_\mu) = \mu + \lambda_1(\mathcal{L})$.

The operator $L$ defined in~\eqref{eq:L:def} and the operator $L^*_\mu$ for $\mu \geq 3/2$ (the shift by $\mu$ ensures that the constant term in~\eqref{eq:L*:def} is non-negative) fall under the class of operators $\mathcal{L}$ discussed in the above paragraph. In particular, $L^*_\mu$ admits a principal eigenpair which we denote as $(\mu+\lambda_R, w_R)$. Note that  $L^*_{\mu}$ and $L^*$ have the same \emph{positive} principal and smooth eigenfunction, so that  $w_R>0$ on $B_R$. Moreover $\mu + \lambda_R = \lambda_1(L^*_\mu) = \mu + \lambda_1(L^*)$. This proves part (i) of Lemma~\ref{lem:lambda:conv}. Since $\lambda_R: = \lambda_1(L^*) = \lambda_1(L)$, we also have that the principal eigenvalue is \emph{real}, thereby proving part (ii) of Lemma~\ref{lem:lambda:conv}. 

Denote the principal eigenpair for the operator $L$ on $B_R$ (with homogeneous Dirichlet boundary conditions) by $(\lambda_R,\varphi_R)$. Since $L$ has no zeroth order term, and since $\varphi_R>0$ in $B_R$ is smooth and vanishes on $\partial B_R$, it must attain its \emph{nonzero} maximum on $\overline{B_R}$ at a point $y_* \in B_R$. At this point, $-\Delta \varphi_R (y_*) \geq 0$ and $\nabla \varphi_R(y_*)=0$, so that $\lambda_R \varphi_R(y_*) = L \varphi_R(y_*) = - \Delta \varphi_R(y_*) \geq 0$. This shows that $\lambda_R \geq 0$, thereby proving part (iii) of Lemma~\ref{lem:lambda:conv}.

In order to prove part (iv) of Lemma~\ref{lem:lambda:conv}, we let $\psi(y):=R^2-|y|^2$, so that $\psi>0$ in $B_R$ and $\psi=0$ on $\partial B_R$. Since $\nabla\psi=-2y$, $\Delta\psi=-6$, and~\eqref{eq:decay} holds, we have
\[
L\psi=-\Delta\psi+(U+\tfrac12 y)\cdot\nabla\psi=6+(U+\tfrac12 y)\cdot(-2y)=6-2U\cdot y-|y|^2
\leq 6 + 2 C_{U,0} - |y|^2.
\]
Let $R_* = 2 \sqrt{6 + 2 C_{U,0}}$, and assume that $R\geq R_*$. If $R/2 \leq |y| \leq R$, then the above inequality shows that $L\psi(y) \leq 0$, while $\psi(y) >0$ by construction. In the complementary regime $|y| < R/2$, we have $L\psi(y) \leq 6 + 2 C_{U,0} = R_*^2/4$, while $\psi(y) = R^2 - |y|^2 \geq 3 R^2/4$. Thus, we have that 
\[
L\psi(y) \leq R_*^2 R^{-2} \psi(y)
\]
for all $y \in B_R$. Then, since both $w_R$ and $\psi$ are positive in $B_R$ and vanish on $\partial B_R$, we obtain
\[
\lambda_R \int_{B_R} w_R \, \psi \, dy 
= \int_{B_R} L^* w_R \, \psi \, dy 
= \int_{B_R} w_R \, L \psi \, dy 
\leq R_*^2 R^{-2} \int_{B_R} w_R \,  \psi \, dy .
\]
This concludes the proof of $\lambda_R \leq R_*^2 R^{-2}$ for $R\geq R_*$, and thus of the Lemma.
\end{proof}

\begin{lemma}[Gaussian upper bound]
\label{lem:Gaussian:upper}
Let $(\lambda_R,w_R)$ be as in Lemma~\ref{lem:lambda:conv}, with $R\geq R_*$. 
For any $\eps \in (0,1)$ there exists a constant $M = M(\eps, C_{U,0})>0$ such that 
\begin{equation*}
w_R(y) \leq 
M \, e^{-\frac{1-\eps}{4} |y|^2} 
\end{equation*}
for all $y\in B_R$.
In particular, $M$ is independent of $R$ and $\alpha$.
\end{lemma}
\begin{proof}[Proof of Lemma~\ref{lem:Gaussian:upper}]
Let $G_\eps(y):= e^{-\frac{1-\eps}{4} |y|^2}$, and define $\phi_R(y):= w_R(y) / G_\eps(y)$. Our goal is to prove $\phi_R \leq M$ on $B_R$, for some $M>0$ independent of $R$ and $\alpha$. Note that $\phi_R=0$ on $\partial B_R$.

Using the definition of $L^*$ in~\eqref{eq:L*:def} and the PDE satisfied by $w_R$ in~\eqref{eq:Dirichlet:EV:L*:a}, we obtain
\begin{align*}
&-\Delta \phi_R - (U + \tfrac 12 y)\cdot \nabla \phi_R 
\notag\\
&\quad 
= G_\eps^{-1} \bigl( -\Delta w_R - (U + \tfrac 12 y)\cdot\nabla w_R)
- 2 \nabla w_R  \cdot \nabla G_\eps^{-1} 
- w_R \bigl( \Delta G_{\eps}^{-1} + (U+\tfrac 12 y)\cdot\nabla G_{\eps}^{-1} \bigr)
\notag\\
&\quad 
= G_\eps^{-1} \bigl( L^* w_R + \tfrac 32 w_R)
- 2 G_\eps \nabla \phi_R  \cdot \nabla G_\eps^{-1} 
- 2 \phi_R \nabla G_\eps  \cdot \nabla G_\eps^{-1} 
+ \phi_R G_\eps \bigl( L^* G_{\eps}^{-1} + \tfrac 32 G_{\eps}^{-1} \bigr)
\notag\\
&\quad 
=   - \nabla \phi_R  \cdot \nabla \log G_\eps^{-2} 
+ \phi_R G_\eps \bigl( L^* G_{\eps}^{-1} + (3+\lambda_R) G_{\eps}^{-1} - 2 G_\eps^{-1} \nabla G_\eps  \cdot \nabla G_\eps^{-1} \bigr).
\end{align*}
Using the fact that $\nabla \log G_\eps^{-2} = (1-\eps) y$ and 
\begin{align*}
&G_\eps \bigl( L^* G_{\eps}^{-1} + (3+\lambda_R) G_{\eps}^{-1} - 2 G_\eps^{-1} \nabla G_\eps  \cdot \nabla G_\eps^{-1} \bigr)
\\
&=  (\tfrac 32 +\lambda_R) - \tfrac 32 (1-\eps) - \tfrac 14 (1-\eps)^2 |y|^2  - \tfrac{1}{2} (1-\eps)(U\cdot y + \tfrac 12 |y|^2)  + \tfrac 12 (1-\eps)^2 |y|^2 
\\
&= \tfrac 32  \eps  +\lambda_R - \tfrac{1}{2} (1-\eps) U\cdot y  - \tfrac{1}{4} \eps (1-\eps) |y|^2 ,
\end{align*}
we thus obtain that 
\begin{align}
 -\Delta \phi_R - \bigl(U + (\eps - \tfrac 12) y\bigr)\cdot \nabla \phi_R 
 +\underbrace{ \big(\tfrac{\eps (1-\eps)}{4}  |y|^2 - \tfrac{3\eps}{2} - \lambda_R + \tfrac{1-\eps}{2}  U\cdot y  \bigr) }_{=: V_\eps(y)} \phi_R = 0
 \label{eq:zR:a}
\end{align}
in $B_R$. By~\eqref{eq:decay}, if $R \geq R_*=R_*(C_{U,0})$ (as defined in item (iv) of Lemma~\ref{lem:lambda:conv}), then
\[
V_\eps(y)\geq 
\tfrac{\eps (1-\eps)}{4}  |y|^2 - \tfrac{3\eps}{2} - 1 - \tfrac{1-\eps}{2}  C_{U,0}.
\]
Thus, for all 
\[
R \geq R_{\eps}:= \max \left\{ R_* , \sqrt{ \tfrac{4+6\eps}{\eps(1-\eps)} + \tfrac{2 C_{U,0}}{\eps}} \right\}
\]
the zeroth order term multiplying $\phi_R$ in~\eqref{eq:zR:a} is positive,  
namely $V_\eps(y) \geq 0$, for all $ R_\eps < |y| <R$.
Thus, we can apply the weak maximum principle in the annulus $R_\eps < |y| < R$, which together with the boundary condition $\phi_R = 0$ on $\partial B_R$ implies
\begin{equation}
\max_{R_\eps \leq |y| \leq R} \phi_R = \max_{|y| = R_\eps} \phi_R.
\label{eq:zR:b} 
\end{equation}
Therefore, we are left to bound $\phi_R$ on $\partial B_{R_\eps}$, where we recall that  $R_\eps$ is independent of $R$ and $\alpha$.

Since $\phi_R >0$ in $B_{2R_\eps}$, and both the drift and the potential term of the uniformly elliptic operator acting on $\phi_R$ in~\eqref{eq:zR:a} are uniformly bounded\footnote{To be precise, we have $\| U + (\eps - \frac 12 ) y\|_{L^\infty(B_{2R_\eps})} \leq C_{U,0} + R_\eps$ and $\|V_\eps\|_{L^\infty(B_{2R_\eps})} \leq R_\eps^2+ \frac 52 +  \frac 12 C_{U,0}  $.} solely in terms of $C_{U,0}$ and $\eps$, by the Harnack inequality (see e.g.~\cite[Corollary 9.25]{GilbargTrudinger98}) we obtain that there exists $M = M (C_{U,0},\eps)> 0$ such that 
\[
\sup_{B_{R_\eps}} \phi_R \leq M \inf_{B_{R_\eps}} \phi_R.
\]
To conclude, we recall that $w_R$ was normalized such that $w_R(0)=1$. Since $G_\eps(0)=1$, it thus follows that $\phi_R(0)=1$ and so $\inf_{B_{R_\eps}} \phi_R \leq 1$. The above Harnack estimate and the continuity of $\phi_R$ thus imply
\begin{equation}
\max_{|y| \leq R_\eps } \phi_R \leq M.
\label{eq:zR:c} 
\end{equation}
Combining \eqref{eq:zR:b} and~\eqref{eq:zR:c} implies that $\phi_R \leq M$ on $B_R$, with $M=M(\eps,C_{U,0})$. 
\end{proof}

\begin{lemma}[Passing $R\to \infty$]
\label{lem:pass:R:infty}
Let $(\lambda_R,w_R)$ be as in Lemma~\ref{lem:lambda:conv} and fix $\eps \in (0,1)$ and $M=M(\eps,C_{U,0})$ as in Lemma~\ref{lem:Gaussian:upper}. 
There exists a sequence $R_k \to \infty$ as $k\to \infty$ such that $w_{R_k}$ converges in $C^{2}_{\rm loc}$ to a strictly positive solution $w$ of 
\[
L^* w = 0
\]
in $\RR^3$. 
 Additionally, we have $w(0)=1$ and  $w$ satisfies the Gaussian upper bound
\begin{equation*}
w(y) \leq  M \, e^{-\frac{1-\eps}{4} |y|^2},
\qquad \mbox{for all} \qquad y \in \RR^3.
\end{equation*}
\end{lemma}
\begin{proof}[Proof of Lemma~\ref{lem:pass:R:infty}]
This is a standard diagonal limiting argument, which we include for completeness. 
Fix $j\in \mathbb{N}$, and let $R\geq \max\{ j+2, R_*\}$ be arbitrary. Lemma~\ref{lem:Gaussian:upper} shows that the family $\{w_R\}_{R\geq j+2}$ is uniformly bounded in $L^\infty(B_{j+2})$; the bounds do not depend on $j$ either. Since $R\geq R_*$, by item (iv) of Lemma~\ref{lem:lambda:conv}, we have $\lambda_R \leq 1$, and thus by \eqref{eq:Dirichlet:EV:L*:a} the family $\{L^* w_R\}_{R\geq j+2}$ is uniformly bounded in $L^\infty(B_{j+2})$; the bounds do not depend on $j$ either. The drift $U + \frac 12 y$ appearing in the definition of the uniformly elliptic operator $L^*$ (recall~\eqref{eq:L*:def}) is independent of $R$ and lies in $C^2(B_{j+2})$; the bounds do depend on $j$. By interior regularity and Schauder estimates, for each fixed $\mu \in(0,1)$ we may find a constant $C_j = C_j(j, \mu,\eps,C_{U,0},C_{U,1},C_{U,2})>0$ such that $\|w_R\|_{C^{2,\mu}(B_j)} \leq C_j$, uniformly for all $R\geq \max\{ j+2, R_*\}$. By Arzela-Ascoli (as $R \to \infty$, at fixed $j$) and a diagonal argument (as $j\to \infty$), we may extract a diverging sequence $\{ R_k  \}_{k\geq 1}$ and identify a function $w \in C^2(\RR^3)$ such that $w_{R_k} \to w$ in $C^{2}_{\rm loc}(\RR^3)$ as $k\to \infty$. The normalization $w(0)=1$ follows since $w_{R_k}(0)= 1$ for all $k$. By construction, we inherit also the Gaussian upper bound and the lower bound $w\geq 0$. 

Again, fix $j\in \mathbb{N}$, and let $k$ be large enough so that $R_k \geq \max\{ j+2, R_*\}$. By the $C^{2}_{\rm loc}$ convergence $w_{R_k} \to w$ and the definition of the operator $L^*$, we have that $L^* w_{R_k} \to L^* w$ uniformly on $B_j$, as $k \to \infty$. Moreover, by item (iv) of Lemma~\ref{lem:lambda:conv}, we have $\lambda_{R_k} \to 0$ as $k\to \infty$, and thus $\lambda_{R_k} w_{R_k} \to 0$ uniformly on $B_j$, as $k \to \infty$. We deduce $L^* w = 0$ on $B_j$; since $j$ was arbitrary, we deduce that $L^* w = 0$ on $\RR^3$.

We are left to show the strict positivity of $w$. Note that the operator $L^*+ \frac 32$ is uniformly elliptic, has smooth locally bounded coefficients, and \emph{contains no zeroth order term}; cf.~\eqref{eq:L*:def}. Moreover, $(L^*+ \frac 32) w = \frac 32 w \geq 0$ in $\RR^3$, since $w$ is non-negative. By the strong minimum principle (see e.g.~\cite[Theorem 3.5]{GilbargTrudinger98}), if there would exist $y_*\in \RR^3$ such that $w(y_*)=0$, then $w$ would be the constant zero function, contradicting $w(0)=1$.
\end{proof}

\begin{lemma}[Gaussian lower bound]
\label{lem:Gaussian:lower}
Let $w$ be as in Lemma~\ref{lem:pass:R:infty}. 
For any $\eps \in (0,1)$ there exists a constant $m = m(\eps, C_{U,0})>0$ such that 
\begin{equation*}
w(y) \geq 
m \, e^{-\frac{1+\eps}{4} |y|^2},
\qquad \mbox{for all} \qquad y \in \RR^3.
\end{equation*}
In particular, $m$ is independent of $\alpha$.
\end{lemma}
\begin{proof}[Proof of Lemma~\ref{lem:Gaussian:lower}]
For $m \in (0,1]$ to be determined, define the function 
\[
\phi(y):= m e^{-\frac{\eps}{2}|y|^2} - w(y) e^{\frac{1-\eps}{4}|y|^2} 
=  e^{-\frac{\eps}{2}|y|^2} \bigl( m - w(y) e^{\frac{1+\eps}{4}|y|^2} \bigr).
\]
The desired lower bound on $w$ is thus equivalent to the inequality $\phi\leq 0$ in $\RR^3$.
Note that the upper bound for $w$ established in Lemma~\ref{lem:pass:R:infty}, together with $m\leq 1$  directly implies that 
\[
- M \leq \phi(y) <  e^{-\frac{\eps}{2}|y|^2} 
\]
for all $y\in \RR^3$, and thus $\limsup_{|y|\to \infty}\phi(y) \leq 0$ since $\eps>0$.

We next verify that $\phi$ is a sub-solution to an elliptic PDE, on the complement of a large ball. Inspired by the computations performed in the proof of Lemma~\ref{lem:Gaussian:upper} for the function $\phi_R(y) = w_R(y) e^{\frac{1-\eps}{4}|y|^2}$, by repeating the same computations that have led to~\eqref{eq:zR:a}  we obtain that 
\begin{align*}
& -\Delta \phi - \bigl(U + (\eps-\tfrac 12) y \bigr) \cdot \nabla \phi 
+  \big(\tfrac{\eps (1-\eps)}{4}  |y|^2 - \tfrac{3\eps}{2}  + \tfrac{1-\eps}{2}  U\cdot y  \bigr) \phi
\notag\\
&\qquad = m \Bigl( -\Delta e^{-\frac{\eps}{2}|y|^2} - \bigl(U + (\eps-\tfrac 12 ) y \bigr) \cdot \nabla e^{-\frac{\eps}{2}|y|^2} 
+  \big(\tfrac{\eps (1-\eps)}{4}  |y|^2 - \tfrac{3\eps}{2}  + \tfrac{1-\eps}{2}  U\cdot y  \bigr) e^{-\frac{\eps}{2}|y|^2}\Bigr)
\notag\\
&\qquad =- m e^{-\frac{\eps}{2}|y|^2} \bigl(   \tfrac{\eps(1+ \eps)}{4}   |y|^2  
- \tfrac{1+\eps}{2}  U\cdot y -   \tfrac{3\eps}{2}   \bigr)
.
\end{align*}
We write the above equation as 
\begin{equation}
\label{eq:phi:elliptic:with:force}
-\Delta \phi - \bigl(U + (\eps -\tfrac 12) y \bigr) \cdot \nabla \phi  + V  \phi = - m e^{-\frac{\eps}{2}|y|^2} F,
\end{equation}
where $V(y) := \tfrac{\eps (1-\eps)}{4}  |y|^2 - \tfrac{3\eps}{2}  + \tfrac{1-\eps}{2}  U\cdot y $ and $F(y):= \tfrac{\eps(1+ \eps)}{4}   |y|^2  - \tfrac{1+\eps}{2}  U\cdot y - \tfrac{3\eps}{2}$. The main observation is that the coefficient of $|y|^2$ in both $V$ and $F$ are positive when $\eps \in (0,1)$. In particular, due to \eqref{eq:decay}, we have that  
\begin{equation}
\label{eq:phi:elliptic:with:force:sign}
V(y)\geq 0 \quad \mbox{and} \quad F(y)\geq 0, \qquad \mbox{for all} \quad |y| \geq R_\eps^\prime,
\end{equation} 
where 
\[
R_\eps^\prime := \sqrt{ \tfrac{ 6 }{1-\eps} + \tfrac{2 C_{U,0}}{\eps}}.
\]

With $R_\eps^\prime=R_\eps^\prime(\eps,C_{U,0})>0$ fixed, we choose the parameter $m$ as follows. Since $L^* w = 0$ in $B_{2R_{\eps}^\prime}$ and the coefficients of the operator $L^*$ are uniformly bounded on $B_{2R_{\eps}^\prime}$ solely in terms of $\eps$ and $C_{U,0}$, by the Harnack inequality we obtain that there exists $m = m(\eps,C_{U,0}) \in (0,1] $ such that 
\[
 \sup_{B_{R_\eps^\prime}} w \leq m^{-1} \inf_{B_{R_\eps^\prime}} w.
\]
Since $w(0)=1$ and $w$ is continuous, we deduce that
\[
w(y) \geq m \qquad \mbox{for all} \qquad y \in \overline{B_{R_\eps^\prime}}.
\]
Therefore, by the definition of $\phi$ we have that 
\[
\phi(y) 
=
e^{-\frac{\eps}{2}|y|^2} \bigl( m - w(y) e^{\frac{1+\eps}{4}|y|^2} \bigr)
\leq 
e^{-\frac{\eps}{2}|y|^2}  ( m - w(y))
\leq 
0
\qquad \mbox{for all} \qquad y \in \overline{B_{R_\eps^\prime}}.
\]

At last, we use~\eqref{eq:phi:elliptic:with:force} and~\eqref{eq:phi:elliptic:with:force:sign} to prove that $\phi\leq 0$ also for $|y|> R_\eps^\prime$. Assume by contradiction that there exists $y_*$ with $|y_*|> R_\eps^\prime$ such that $\phi(y_*) > 0$. Pick $R$ large enough so that $e^{-\frac{\eps}{2} R^2} < \phi(y_*)$ and $R > |y_*|$. On the annulus $R_\eps^\prime< |y| < R$, \eqref{eq:phi:elliptic:with:force}--\eqref{eq:phi:elliptic:with:force:sign} shows that $\phi$ is a sub-solution of the uniformly elliptic operator 
$-\Delta - \bigl(U + (\eps -\tfrac 12) y \bigr) \cdot \nabla + V$, and $V\geq 0$ on this set. By the weak maximum principle and our a priori global bound for $\phi$, we have that 
\[
\phi(y_*) \leq \max_{R_\eps^\prime \leq |y| \leq R} \phi(y) = \max_{\{ |y| = R_{\eps}^\prime \} \cup \{|y|=R \} } \phi(y) 
\leq e^{-\frac{\eps}{2} R^2}
< \phi(y_*),
\]
a contradiction. This proves that $\phi \leq 0$ on $B_{R_\eps^\prime}^\complement$, and thus concludes the proof of the Lemma.
\end{proof}

\begin{proof}[Proof of Proposition~\ref{prop:weight:construction}]
The existence of a $C^2$ smooth function $w$ lying in the kernel of $L^*$ is established in Lemma~\ref{lem:pass:R:infty}. The upper bound stated in~\eqref{eq:weight:decay} is established as part of Lemma~\ref{lem:pass:R:infty}, while the lower bound stated  in~\eqref{eq:weight:decay} is established as part of Lemma~\ref{lem:Gaussian:lower}. 
\end{proof}

\section{RSS: proof of the main result for $\alpha$ large}
\label{sec:alpha:large}
The goal is to prove that if $|\alpha|$ is large enough,  $U$ is close to the kernel of the operator $J  -  Jy\cdot\nabla$. That is, $U$ is almost axisymmetric. If $U$ were truly axisymmetric, then because of~\eqref{eq:E:is:theta:derivative} we'd have $E \equiv 0$. Instead, when $|\alpha|\gg 1$ we show that $\alpha E$ is small, from which~\eqref{eq:small:enstrophy} follows, and we conclude the proof of $U\equiv 0$ with Proposition~\ref{prop:local:vort}.

\subsection{Cylindrical coordinates, rotation operator, angular mean}
We write a Cartesian point $y=(y_1,y_2,y_3) \in \RR^3$ in cylindrical coordinates as $y = (r\cos\theta,r\sin\theta,z)$, where $r = \sqrt{y_1^2+y_2^2} \geq 0$, $z = y_3 \in \RR$, and $\theta \in [0,2\pi)$. As usual, we denote the cylindrical coordinate frame by $e_r = (\cos\theta,\sin\theta,0)$, $e_\theta=(-\sin\theta,\cos\theta,0)$, and $e_z=(0,0,1)$.\footnote{It is convenient to record the identities: $\p_\theta e_r = e_\theta$, $\p_\theta e_\theta = - e_r$, $\p_\theta e_z = 0$.} 
Vector fields $V \colon \RR^3 \to \RR^3$ are written as 
\[
V(y) = V^r (r,\theta,z) e_r + V^\theta(r,\theta,z) e_\theta + V^z(r,\theta,z) e_z.
\]

We recall that the operator $J$ defined in \eqref{eq:RJ} acts on vector fields $V = (V_1,V_2,V_3)$ as $J V = (-V_2,V_1,0)$. In cylindrical coordinates, this identity is
\[
JV  = - V^\theta   e_r + V^r  e_\theta .
\]
Specializing to $V(y) = y = r e_r + z e_z$, and using that $\nabla = e_r \p_r + \frac{1}{r} e_\theta \p_\theta + e_z \p_z$, we deduce that 
\[
Jy = r e_\theta,
\qquad
(Jy) \cdot \nabla = \p_\theta.
\]
In particular, for a vector field $V \colon \RR^3 \to \RR^3$, we have 
\[
\bigl( (Jy) \cdot \nabla \bigr) V  = \bigl( \p_\theta V^r - V^\theta \bigr) e_r + \bigl( \p_\theta  V^\theta + V^r \bigr)   e_\theta + \p_\theta V^z   e_z.
\]

We define  
\begin{equation}
\mathcal{R} V := J V - (J y) \cdot \nabla V, 
\label{eq:cal:R:def}
\end{equation}
to be the operator appearing in~\eqref{eq:profile} and~\eqref{eq:vort} with amplification coefficient $\alpha$. In cylindrical coordinates, the operator $\mathcal{R}$ acts as
\begin{equation}
 \mathcal{R} V = - \p_\theta V^r e_r - \p_\theta V^\theta e_\theta - \p_\theta V^z e_z.
 \label{eq:rotation:operator}
\end{equation}
That is, $\mathcal{R}$ acts on vector fields by applying $-\p_\theta$ to each cylindrical component separately. 

Next, for any scalar function $f\colon \RR^3 \to \RR$ we define 
\begin{equation}
\label{eq:f:mean:theta}
\mean{f}_\theta := \mean{f}_\theta(r,z) = \frac{1}{2\pi} \int_0^{2\pi} f(r,\theta,z) d\theta.
\end{equation}
For a vector field $V\colon \RR^3 \to \RR^3$, we accordingly define 
\begin{equation}
\label{eq:V:mean:theta}
\mean{V}_\theta := \mean{V^r}_\theta(r,z) e_r + \mean{V^\theta}_\theta(r,z) e_\theta + \mean{V^z}_\theta(r,z) e_z.
\end{equation}
Note that a vector field $V$ is axisymmetric \emph{if and only if} $V = \mean{V}_\theta$. Thus, in order to measure how far a vector field is from being axisymmetric, we introduce the notation
\begin{equation}
\label{eq:V:away:from:mean}
(V)_a := V - \mean{V}_\theta = (V^r)_a(r,\theta,z) e_r + (V^\theta)_a(r,\theta,z) e_\theta + (V^z)_a(r,\theta,z) e_z,
\end{equation}
where for any scalar function $f\colon \RR^3 \to \RR$, we have denoted 
\begin{equation}
\label{eq:f:away:from:mean}
(f)_a = f - \mean{f}_\theta.
\end{equation}
With this notation $V$ is axisymmetric \emph{if and only if} $(V)_a = 0$. We will say that  \emph{$V$ is nearly axisymmetric} if $ (V)_a \ll 1$ (in a suitable topology). We record a few useful facts.
\begin{lemma}
\label{lem:axisym:properties}
Let $f,g\colon \RR^3\to \RR$ be smooth scalar functions and $V\colon \RR^3 \to \RR^3$ be a smooth vector field.  Then, we have:
\begin{enumerate}[label=(\roman*)]
\item $\mathcal{R} \mean{V}_\theta = 0$ and hence  $\mathcal{R} V = \mathcal{R}(V)_a$; moreover, $\mean{\mathcal{R} V}_\theta = 0$ and hence $(\mathcal{R} V)_a = \mathcal{R}(V)_a$.
\item $\mean{\p_r f}_\theta = \p_r \mean{f}_\theta$ and $\mean{\p_z f}_\theta = \p_z \mean{f}_\theta$; therefore, $(\p_r f)_a = \p_r (f)_a$ and $(\p_z f)_a = \p_z (f)_a$.
\item $(\p_\theta f)_a = \p_\theta (f)_a$ and hence $(\nabla f)_a = \nabla (f)_a$.
\item $(\nabla \cdot V)_a = \nabla \cdot (V)_a$ and hence  $(\Delta f)_a = \Delta (f)_a$.
\item $(f \, g)_a = (f)_a \mean{g}_\theta + \mean{f}_\theta (g)_a + \bigl( (f)_a (g)_a \bigr)_a$.
\item $\bigl( (V\cdot \nabla) f \bigr)_a
= \mean{V}_\theta \cdot \nabla (f)_a 
+ (V^r)_a \p_r \mean{f}_\theta + (V^z)_a \p_z \mean{f}_\theta
+ \bigl( ( (V)_a \cdot \nabla) (f)_a \bigr)_a$.
\end{enumerate}
\end{lemma}
\begin{proof}[Proof of Lemma~\ref{lem:axisym:properties}]
Since $\mathcal{R}$ acts on vector fields by applying $-\p_\theta$ to each of its cylindrical components separately, item \emph{(i)} follows directly from the definition~\eqref{eq:V:mean:theta}.

The claims in item \emph{(ii)} follow directly from the definitions~\eqref{eq:f:mean:theta} and~\eqref{eq:f:away:from:mean}.

Note that $\mean{\p_\theta f}_\theta =0$, and thus $(\p_\theta f)_a = \p_\theta f$. Since $\p_\theta
\mean{f}_\theta = 0$, it follows that $\p_\theta (f)_a = \p_\theta f$, proving the first claim in \emph{(iii)}. 
Since $\nabla f= e_r \p_r f+ e_\theta \frac{1}{r}  \p_\theta f + e_z \p_z f $, from item \emph{(ii)},  \eqref{eq:V:away:from:mean}, and the previously established identity,  we deduce $(\nabla f)_a = e_r (\p_r f)_a + e_\theta (\frac{1}{r}  \p_\theta f)_a + e_z (\p_z f)_a =  e_r \p_r (f)_a + e_\theta \frac{1}{r}  \p_\theta (f)_a + e_z\p_z  (f)_a = \nabla (f)_a$.

In cylindrical coordinates we have $\nabla \cdot V = \p_r V^r + \frac{1}{r} V^r + \frac{1}{r} \p_\theta V^\theta + \p_z V^z$, and thus $\mean{\nabla \cdot V}_\theta = \mean{\p_r V^r}_\theta + \mean{\frac{1}{r} V^r}_\theta  + \mean{\p_z V^z}_\theta =\p_r   \mean{V^r}_\theta + \frac{1}{r} \mean{ V^r}_\theta  + \p_z  \mean{V^z}_\theta = \nabla \cdot \mean{V}_\theta$; in the last equality we have used that $\p_\theta \mean{f}_\theta = 0$. The first claim in item \emph{(iv)} now follows from~\eqref{eq:V:away:from:mean} and \eqref{eq:f:away:from:mean}. The second claim follows by letting $V = \nabla f$, and using the second claim in item~\emph{(iii)}.

The  claim in item~\emph{(v)} follows from $(f \, g)_a 
= \bigl( \mean{f}_\theta \mean{g}_\theta + (f)_a \mean{g}_\theta + \mean{f}_\theta (g)_a + (f)_a (g)_a \bigr)_a$ together with the fact that functions independent of $\theta$ lie in the kernel of $(\cdot)_a$ and the fact that multiplication by a function independent of $\theta$ commutes with the operation $(\cdot)_a$.

Finally, item~\emph{(vi)} follows from the identity $V\cdot\nabla f = V^r \p_r f+ \frac 1r V^\theta \p_\theta f + V^z \p_z f$, and  item~\emph{(v)}.
\end{proof}

\subsection{Gaussian weighted $L^2$}
The velocity~\eqref{eq:profile} and vorticity~\eqref{eq:vort} equations both contain the elliptic operator $- \Delta + \frac 12 y \cdot \nabla$, which is not self-adjoint on $L^2(\RR^3)$. In order to recover this favorable property (which will be useful in the analysis of this section), we work with the Gaussian weighted space $L^2_\mu(\RR^3)$, defined as follows. Let
\[
\mu = \mu(y) = e^{-\frac 14 |y|^2} = e^{-\frac 14 (r^2 + z^2)}.
\] 
For smooth vector fields $V,G\colon \RR^3 \to \RR^3$ we define  \begin{equation}
\langle V, G\rangle_{L^2_\mu} : = \int_{\RR^3} V(y) \cdot G(y) \, \mu(y) \, dy 
=  \int_{\RR} \int_0^\infty  \mean{V\cdot G}_\theta(r,z) \mu(r,z) 2\pi r dr dz,
\quad 
\| V\|_{L^2_\mu}^2 = \langle V, V\rangle_{L^2_\mu},
\label{eq:L2:mu:def}
\end{equation}
while for smooth scalar functions $f,g \colon \RR^3 \to \RR$ we define
\begin{equation}
\langle f , g\rangle_{L^2_\mu} : = \int_{\RR^3} f(y) g(y) \, \mu(y) \, dy 
=  \int_{\RR} \int_0^\infty  \mean{f \, g}_\theta(r,z) \, \mu(r,z) \, 2\pi r \, dr dz,
\quad 
\| f\|_{L^2_\mu}^2 = \langle f, f\rangle_{L^2_\mu}.
\label{eq:L2:mu:def:scalar}
\end{equation}
With these definitions, we record a few useful facts.

\begin{lemma}
\label{lem:useful:inner:product}
Let $f,g \colon \RR^3 \to \RR$ be  smooth scalar functions and $V, G \colon \RR^3 \to \RR^3$ be  smooth vector fields. Then we have:
\begin{enumerate}[label=(\roman*)]
\item The operator $- \Delta + \frac 12 y \cdot \nabla$ is self-adjoint on $L^2_\mu(\RR^3)$ and 
$\langle (- \Delta + \tfrac 12 y \cdot \nabla) V , V\rangle_{L^2_\mu} 
= 
\|\nabla V\|_{L^2_\mu}^2$.
\item $\langle \mathcal{R} V, G \rangle_{L^2_\mu} = - \langle  V, \mathcal{R} G \rangle_{L^2_\mu}$ and $\langle \mathcal{R} V, G \rangle_{L^2_\mu} = \langle \mathcal{R} V, (G)_a \rangle_{L^2_\mu}$.
\item $\langle (f)_a, \mean{g}_\theta\rangle_{L^2_\mu} = 0$ and $\langle (V)_a, \mean{G}_\theta \rangle_{L^2_\mu} = 0$; hence 
\[
\|f\|_{L^2_\mu}^2 = \|\mean{f}_\theta \|_{L^2_\mu}^2 + \|(f)_a\|_{L^2_\mu}^2, 
\qquad \mbox{and} \qquad 
\|V\|_{L^2_\mu}^2 = \|\mean{V}_\theta \|_{L^2_\mu}^2 + \|(V)_a\|_{L^2_\mu}^2.
\]
\item The Poincar\'e inequality in the angular variable $\theta \in [0,2\pi)$ implies that
\begin{equation}
\| (V)_a \|_{L^2_\mu} \leq 2 \| \mathcal{R} V \|_{L^2_\mu}.
\label{eq:theta:Poincare}
\end{equation}
\item $\| f g \|_{L^2_\mu} \leq \|f\|_{L^\infty} \|g\|_{L^2_\mu}$.
\end{enumerate}
\end{lemma}
\begin{proof}[Proof of Lemma~\ref{lem:useful:inner:product}]
Item \emph{(i)} follows by the very definition of the Gaussian weight $\mu$. 

By~\eqref{eq:rotation:operator}, \eqref{eq:f:mean:theta}, and integration by parts, we have that $\mean{\mathcal{R} V \cdot G}_\theta = -\mean{\p_\theta V^r G^r + \p_\theta V^\theta G^\theta + \p_\theta V^z G^z}_\theta = \mean{V^r \p_\theta G^r + V^\theta \p_\theta  G^\theta + V^z \p_\theta  G^z}_\theta = - \mean{V \cdot \mathcal{R} G}_\theta$. The first claim in item \emph{(ii)} now follows by  using definition~\eqref{eq:L2:mu:def}. The second claim follows by additionally appealing to item~\emph{(i)} of Lemma~\ref{lem:axisym:properties}.

By~\eqref{eq:L2:mu:def:scalar}, we have that $\mean{ (f)_a \mean{g}_\theta}_\theta = \mean{ (f)_a }_\theta \mean{g}_\theta = 0$, proving the first claim in item~\emph{(iii)}. The second claim follows by applying the first one to the $e_r$, $e_\theta$, and $e_z$ components of $V$ and $G$ separately. The claims about the orthogonal decomposition follow from \eqref{eq:f:mean:theta}--\eqref{eq:f:away:from:mean} and the bilinearity of the inner product.

In order to prove~\eqref{eq:theta:Poincare}, we appeal to~\eqref{eq:V:away:from:mean}, \eqref{eq:f:away:from:mean}, \eqref{eq:rotation:operator}, and the Poincar\'e-Wirtinger inequality in $\theta$ (applicable since $(f)_a$ has zero mean on the circle for any scalar function $f$), to conclude that for every fixed $(r,z) \in \RR_+ \times \RR$ we have   
\begin{align*}
\mean{(V)_a \cdot (V)_a}_\theta
&= 
\mean{(V^r)_a^2}_\theta
+
\mean{(V^\theta)_a^2}_\theta
+
\mean{(V^z)_a^2}_\theta
\\
&=
\tfrac{1}{2\pi} \|(V^r)_a\|_{L^2(0,2\pi)}^2 
+
\tfrac{1}{2\pi} \|(V^\theta)_a\|_{L^2(0,2\pi)}^2 
+
\tfrac{1}{2\pi} \|(V^z)_a\|_{L^2(0,2\pi)}^2 
\\
&\leq 
\tfrac{2}{\pi}\|\p_\theta (V^r)_a\|_{L^2(0,2\pi)}^2
+
\tfrac{2}{\pi}\|\p_\theta (V^\theta)_a\|_{L^2(0,2\pi)}^2
+
\tfrac{2}{\pi}\|\p_\theta (V^z)_a\|_{L^2(0,2\pi)}^2
\\
&= 
\tfrac{2}{\pi}\|\p_\theta  V^r \|_{L^2(0,2\pi)}^2
+
\tfrac{2}{\pi}\|\p_\theta  V^\theta \|_{L^2(0,2\pi)}^2
+
\tfrac{2}{\pi}\|\p_\theta  V^z \|_{L^2(0,2\pi)}^2
\\
&= 
4  \mean{\mathcal{R} V \cdot \mathcal{R} V}_\theta.
\end{align*}
The proof of~\eqref{eq:theta:Poincare} now follows from the above estimate and definition~\eqref{eq:L2:mu:def}.

The bound in item~\emph{(v)} follows from the definition~\eqref{eq:L2:mu:def:scalar} and H\"older's inequality.
\end{proof} 

\subsection{Estimates for the nonlinearity and the pressure gradient}
Inspecting the velocity equation~\eqref{eq:profile:a} we note that all terms except $(U\cdot \nabla) U$ and $\nabla P$ are linear in $U$. We estimate these terms as follows:
\begin{lemma}
\label{lem:source:bound}
Assume that the smooth incompressible vector field $U\colon \RR^3 \to \RR^3$ satisfies the bounds~\eqref{eq:decay} and~\eqref{eq:decay:grad}.
Define
\begin{equation}
\mathcal{N} := - (U \cdot \nabla) U - \nabla P,
\label{eq:source:def}
\end{equation}
where $P = R_i R_j (U^i U^j)$ (as in the proof of Lemma~\ref{lem:gradient}). For any $\eps \in (0,1]$, there exists a constant $C_\eps = C_\eps(\eps,C_{U,0})>0$ such that 
\begin{equation}
\| (\mathcal{N})_a \|_{L^2_\mu} \leq \eps + C_\eps   \|(U)_a\|_{L^2_\mu} + C_\eps   \| \nabla (U)_a \|_{L^2_\mu}  .
\label{eq:source:bound}
\end{equation}
\end{lemma}
\begin{proof}[Proof of Lemma~\ref{lem:source:bound}]
Let us first analyze the nonlinear term $(U\cdot \nabla) U$. Since $U = \mean{U}_\theta + (U)_a$, we may decompose
\[
(U \cdot \nabla) U =  (\mean{U}_\theta \cdot \nabla) \mean{U}_\theta  + (\mean{U}_\theta \cdot \nabla) (U)_a + ((U)_a  \cdot \nabla)\mean{U}_\theta  + ((U)_a  \cdot \nabla) (U)_a.
\]
Note that the term $(\mean{U}_\theta \cdot \nabla) \mean{U}_\theta$ is axisymmetric, meaning that its cylindrical components are independent of $\theta$.\footnote{To see this note that $(V\cdot \nabla) V =  ( (V\cdot\nabla)V^r - \tfrac 1r (V^\theta)^2 ) e_r + ( (V\cdot\nabla)V^\theta + \tfrac 1r V^r V^\theta) e_\theta + ( (V\cdot\nabla)V^z) e_z$, and also $V\cdot\nabla = V^r \p_r + \frac 1r V^\theta \p_\theta + V^z \p_z$.} 
Therefore, the projection operator $(\cdot)_a$ annihilates this term: 
\[
\bigl( (\mean{U}_\theta \cdot \nabla) \mean{U}_\theta \bigr)_a = 0.
\]
To estimate the remaining three terms, note that item~\emph{(iii)} of Lemma~\ref{lem:useful:inner:product}, gives $\|(V)_a\|_{L^2_\mu} \leq \|V\|_{L^2_\mu}$.  Moreover, using \eqref{eq:V:mean:theta}, \eqref{eq:decay}, \eqref{eq:decay:grad}, and items~\emph{(ii)} and~\emph{(iii)} of Lemma~\ref{lem:axisym:properties}  we have that $\| \mean{U}_\theta\|_{L^\infty} \leq C_{U,0}$, $\| (U)_a\|_{L^\infty} \leq 2 C_{U,0}$, and  $\|\nabla\mean{U}_\theta\|_{L^\infty} \leq C_{U,1}$. Using these facts and H\"older's inequality, we deduce
\begin{align*}
\| ((\mean{U}_\theta \cdot \nabla) (U)_a)_a \|_{L^2_\mu}
&\leq
\|  (\mean{U}_\theta \cdot \nabla) (U)_a  \|_{L^2_\mu}
\leq \| \mean{U}_\theta\|_{L^\infty}\| \nabla (U)_a  \|_{L^2_\mu}
\leq C_{U,0} \|  \nabla (U)_a  \|_{L^2_\mu},
\\
\| (((U)_a  \cdot \nabla)\mean{U}_\theta)_a \|_{L^2_\mu}
&\leq
\|  ((U)_a  \cdot \nabla)\mean{U}_\theta  \|_{L^2_\mu}
\leq \| (U)_a\|_{L^2_\mu} \|  \nabla \mean{U}_\theta\|_{L^\infty}
\leq
C_{U,1} \| (U)_a\|_{L^2_\mu},
\\
\| ( ((U)_a  \cdot \nabla) (U)_a)_a \|_{L^2_\mu}
&\leq
\|  ((U)_a  \cdot \nabla) (U)_a  \|_{L^2_\mu}
\leq \|(U)_a\|_{L^\infty} \| \nabla (U)_a  \|_{L^2_\mu}
\leq 2 C_{U,0} \|  \nabla (U)_a  \|_{L^2_\mu}.
\end{align*}
Summing these bounds we obtain
\begin{equation}
\label{eq:U:nabla:U:a}
\| ( (U\cdot \nabla) U  )_a\|_{L^2_\mu}
\leq 
C_{U,1} \|(U)_a\|_{L^2_\mu}
+ 
3 C_{U,0} \|\nabla (U)_a\|_{L^2_\mu}.
\end{equation}
This estimate is clearly consistent with~\eqref{eq:source:bound}. 

Next, we bound the pressure gradient. Note that we are seeking a bound in the Gaussian weighted $L^2_\mu$ norm, and so we do not have access to usual Calderon--Zygmund bounds. This is compensated for by the $\eps$-factor present on the right side of~\eqref{eq:source:bound}. Recall that $-\Delta P = \partial_{i} \partial_j(U^i U^j)$. Using item~\emph{(iv)} of Lemma~\ref{lem:axisym:properties}, we obtain 
\[
-\Delta (P)_a 
=  (\partial_{i} \partial_j(U^i U^j) )_a 
= ( \nabla \cdot \partial_j(U U^j))_a 
= \nabla \cdot ( \partial_j (U U^j) )_a
= \nabla \cdot ( (U\cdot\nabla) U)_a.
\]
Since $P$ decays (recall~\eqref{eq:decay:pressure}) so does $(P)_a$, and hence $(P)_a = (-\Delta)^{-1} \nabla \cdot ( (U\cdot\nabla) U)_a$,  where $(-\Delta)^{-1}$ denotes convolution with the Newtonian potential. Using item~\emph{(iii)} of Lemma~\ref{lem:axisym:properties}, we arrive at
\begin{equation}
\label{eq:nabla:P:a}
(\nabla P)_a = \nabla (P)_a = \nabla (-\Delta)^{-1} \nabla \cdot ( (U\cdot\nabla) U)_a,
\end{equation}
which is to say that  
\[
(\nabla P)_a^i = R_i R_j V^j,
\qquad \mbox{where} \qquad 
V^j =  ( (U\cdot \nabla) U )_a^j,
\] 
and $R_i$, $R_j$ are Riesz transforms. Note that we have previously (cf.~\eqref{eq:U:nabla:U:a})  bounded $\|V\|_{L^2_\mu}$,  and using   \eqref{eq:V:mean:theta}, \eqref{eq:decay}, \eqref{eq:decay:grad} we obtain that 
$|V(y)|\leq 2 C_{U,0} C_{U,1}(1+ |y|^3)^{-1}$. Since $\mu \leq 1$ and $R_i R_j$ is bounded on $L^2(\RR^3)$, we immediately have 
\[
\|R_i R_j V^j\|_{L^2_\mu} \leq \|R_i R_j V^j \|_{L^2} \leq \|V\|_{L^2}.
\] 
Then for $L_\eps \geq 1$, to be determined, we have 
\begin{align}
\|V\|_{L^2} 
&\leq \|{\bf 1}_{|y|\leq L_\eps }  V \|_{L^2} + \|{\bf 1}_{|y|>L_\eps }  V \|_{L^2} 
\notag\\
& \leq \|{\bf 1}_{|y|\leq L_\eps } \mu^{-\frac 12}\|_{L^\infty} \|\mu^{\frac 12} V \|_{L^2} 
+  2 C_{U,0} C_{U,1} \|{\bf 1}_{|y|>L_\eps }  |y|^{-3} \|_{L^2} 
\notag\\
& \leq e^{\frac 18 L_\eps ^2} \|V\|_{L^2_\mu} + 8 C_{U,0} C_{U,1} L_\eps^{-\frac 32}.
\label{eq:nabla:P:a:2}
\end{align}
Therefore, if we let
\[
L_\eps := \max \bigl\{ 1,  \bigl(\tfrac{24 C_{U,0} C_{U,1}}{\eps}\bigr)^{\frac 23} \bigr\}, 
\]
we obtain from~\eqref{eq:nabla:P:a}--\eqref{eq:nabla:P:a:2}, together with~\eqref{eq:U:nabla:U:a}, that 
\begin{equation}
\label{eq:nabla:P:a:3}
\| (\nabla P)_a \|_{L^2_\mu} \leq \eps + 9 e^{\frac 18 L_\eps^2} \bigl( C_{U,1} \|(U)_a\|_{L^2_\mu}
+ C_{U,0} \|\nabla (U)_a\|_{L^2_\mu} \bigr).
\end{equation}
Adding~\eqref{eq:U:nabla:U:a} and \eqref{eq:nabla:P:a:3}, and recalling that $C_{U,1} = C_{U,1}(C_{U,0})$, concludes the proof of~\eqref{eq:source:bound}.
\end{proof}

\subsection{Estimate for $\mathcal{R} U$}
Recall that $U$ solves the rotated backward self-similar Navier-Stokes equations~\eqref{eq:profile}. Using the operator $\mathcal{R}$ defined in~\eqref{eq:cal:R:def} and the nonlinear terms $\mathcal{N}$ defined in~\eqref{eq:source:def}, the $U$ equation may be rewritten as
\begin{subequations}
\label{eq:new:profile}
\begin{align}
\alpha \mathcal{R} U + \tfrac12 U + \bigl( -\Delta + \tfrac12 (y\cdot\nabla) \bigr)U &= \mathcal{N},
\label{eq:new:profile:a}\\
\nabla\cdot U &= 0.
\label{eq:new:profile:b}
\end{align}
\end{subequations}
The main result of this subsection is:
\begin{lemma}
\label{lem:alpha:coercive}
Let $\alpha\in \RR$ and let $U$ be a smooth solution of~\eqref{eq:new:profile} which satisfies the bounds~\eqref{eq:decay} and~\eqref{eq:decay:grad}. Then, we have 
\begin{equation}
\label{eq:alpha:coercive}
|\alpha| \, \| \mathcal{R} U\|_{L^2_\mu}^2
+ \tfrac 12 \| (U)_a \|_{L^2_\mu}^2
+ \|\nabla (U)_a\|_{L^2_\mu}^2 
\leq 
3
\| \mathcal{R} U\|_{L^2_\mu}
\| (\mathcal{N})_a \|_{L^2_\mu}.
\end{equation}
\end{lemma}
\begin{proof}[Proof of Lemma~\ref{lem:alpha:coercive}]
We begin by noting that the operators $\mathcal{R}$ and $-\Delta + \frac 12 y \cdot \nabla$ commute; this holds because $\mathcal{R}$ acts as $-\p_\theta$ on each cylindrical component of the vector fields it acts on,  and both $-\Delta$ and $y\cdot \nabla$ are differential operators with $\theta$-independent coefficients, when written in cylindrical coordinates.\footnote{\label{foot:Delta:V}Recall that for a scalar function $f$ we have $\Delta f = (\partial_{rr} + \frac 1r \partial_r) f + \frac{1}{r^2} \partial_{\theta\theta} f + \partial_{zz} f$, while for a vector field $V$ we have $\Delta V = (\Delta V^r - \frac{1}{r^2} V^r - \frac{2}{r^2} \p_\theta V^\theta) e_r + (\Delta V^\theta - \frac{1}{r^2} V^\theta + \frac{2}{r^2} \p_\theta V^r) e_\theta + (\Delta V^z) e_z$. Moreover, $y \cdot \nabla = r \p_r + z \p_z$.} Item~\emph{(ii)} of Lemma~\ref{lem:useful:inner:product} shows that $\mathcal{R}$ is skew-adjoint on $L^2_\mu$, whereas item~\emph{(i)} shows that $-\Delta + \frac 12 y \cdot \nabla$ is self-adjoint. Combined, these facts yield
\[
\langle \mathcal{R} U, (-\Delta + \tfrac 12 y \cdot \nabla) U\rangle_{L^2_\mu} = 0.
\]
Since we also have $\langle \mathcal{R} U, U\rangle_{L^2_\mu} = 0$, 
taking the $L^2_\mu$ inner product of~\eqref{eq:new:profile:a} with $
\mathrm{sgn}(\alpha) \mathcal{R}U$, and using the second identity in item~\emph{(ii)} of Lemma~\ref{lem:useful:inner:product}, we obtain
\begin{equation}
\label{eq:first:EE}
 |\alpha| \, \| \mathcal{R} U\|_{L^2_\mu}^2 
 = \mathrm{sgn}(\alpha) \langle \mathcal{R} U, \mathcal{N} \rangle_{L^2_\mu}
 = \mathrm{sgn}(\alpha) \langle \mathcal{R} U, (\mathcal{N})_a \rangle_{L^2_\mu}
 \leq \| \mathcal{R} U\|_{L^2_\mu} \| (\mathcal{N})_a  \|_{L^2_\mu}.
\end{equation}
Note however that bounding $(\mathcal{N})_a$ requires knowledge of $\nabla (U)_a$, which motivates our next computation.

We first apply the projection $(\cdot)_a$ to equation~\eqref{eq:new:profile:a}. 
We use Lemma~\ref{lem:axisym:properties} as follows: item~\emph{(i)} gives us $(\mathcal{R}U)_a = \mathcal{R}(U)_a$; the definition of $\Delta U$ in cylindrical coordinates (see Footnote~\ref{foot:Delta:V}) together with items~\emph{(iii)} and~\emph{(iv)} gives $(\Delta U)_a = \Delta (U)_a$;  and item~\emph{(v)} with $V = y = r e_r + z e_z$ so that $(V)_a = (y)_a = 0$, we obtain that $( y\cdot \nabla U)_a = y \cdot \nabla (U)_a$. With these identities, \eqref{eq:new:profile:a} gives
\[
\alpha \mathcal{R}(U)_a + \tfrac 12 (U)_a + \bigl( - \Delta + \tfrac 12 (y\cdot\nabla) \bigr) (U)_a = (\mathcal{N})_a.
\]
We take the $L^2_\mu$ inner product of this identity with $(U)_a$, appeal to items~\emph{(i)} and~\emph{(ii)} in Lemma~\ref{lem:useful:inner:product}, and deduce
\begin{equation*}
\tfrac 12 \| (U)_a \|_{L^2_\mu}^2 + \| \nabla (U)_a\|_{L^2_\mu}^2 
= \langle (\mathcal{N})_a,(U)_a\rangle_{L^2_\mu} .
\end{equation*}
By appealing to the angular Poincar\'e inequality~\eqref{eq:theta:Poincare}, we thus obtain
\begin{equation}
\label{eq:second:EE}
\tfrac 12 \| (U)_a \|_{L^2_\mu}^2 + \| \nabla (U)_a\|_{L^2_\mu}^2 
\leq  2 \| (\mathcal{N})_a \|_{L^2_\mu} \| \mathcal{R}U\|_{L^2_\mu}.
\end{equation}
Adding~\eqref{eq:first:EE} and~\eqref{eq:second:EE} concludes the proof.
\end{proof}

By combining Lemmas~\ref{lem:source:bound} and~\ref{lem:alpha:coercive} we may show that $\mathcal{R}U$ is small when $|\alpha|$ is large, in the following sense:
\begin{proposition}
\label{prop:large:alpha}
Let $\alpha\in \RR$ and let $U$ be a smooth solution of~\eqref{eq:new:profile} which satisfies the bounds~\eqref{eq:decay} and~\eqref{eq:decay:grad}. Let $\eps \in (0,1]$ be arbitrary. Then, there exists a constant $A_\eps = A(\eps,C_{U,0})\geq 1$, such that for all $|\alpha|\geq A_{\eps}$ we have 
\begin{equation}
|\alpha|  \, \| \mathcal{R} U\|_{L^2_\mu}
\leq  
\eps.
\label{eq:large:alpha:main}
\end{equation}
\end{proposition}
\begin{proof}[Proof of Proposition~\ref{prop:large:alpha}]
The bounds~\eqref{eq:source:bound} (with $\eps$ replaced by $\frac 16 \eps$), \eqref{eq:alpha:coercive}, and Young's inequality show that there exists $C_{\eps}= C_{\eps}(\eps,C_{U,0})>0$, which is in particular independent of $\alpha$, such that  
\begin{align*}
&|\alpha| \, \| \mathcal{R} U\|_{L^2_\mu}^2
+ \tfrac 12 \| (U)_a \|_{L^2_\mu}^2
+ \|\nabla (U)_a\|_{L^2_\mu}^2 
\notag\\
&\qquad \leq 
3
\| \mathcal{R} U\|_{L^2_\mu}
\bigl( \tfrac 16 \eps + C_{\eps} \|(U)_a\|_{L^2_\mu} + C_{\eps} \|\nabla (U)_a\|_{L^2_\mu}\bigr)
\notag\\
&\qquad \leq \tfrac 12  \eps \| \mathcal{R} U\|_{L^2_\mu} 
+ \tfrac 12 \|(U)_a\|_{L^2_\mu}^2 + \tfrac{9}{2} C_\eps^2 \| \mathcal{R} U\|_{L^2_\mu}^2
+ \|\nabla(U)_a\|_{L^2_\mu}^2 + \tfrac{9}{4} C_\eps^2 \| \mathcal{R} U\|_{L^2_\mu}^2.
\end{align*}
Absorbing the appropriate terms into the left side of the above inequality gives
\[
|\alpha| \, \| \mathcal{R} U\|_{L^2_\mu}^2
\leq  \tfrac 12 \eps \| \mathcal{R} U\|_{L^2_\mu}  +  \tfrac{27}{4} C_\eps^2  \| \mathcal{R} U\|_{L^2_\mu}^2.
\]
Letting $A_\eps :=  \frac{27}{2} C_\eps^2$, the above inequality shows that if $|\alpha| \geq A_\eps$, then $\frac 12 |\alpha|  \| \mathcal{R} U\|_{L^2_\mu}^2
\leq \tfrac 12 \eps \| \mathcal{R} U\|_{L^2_\mu} $, concluding the proof.
\end{proof}

\subsection{Proof of the main result for RSS solutions}
\label{sec:proof:RSS:large}
We conclude this section by giving the proof of our main result, when $|\alpha|\gg 1$.

\begin{proof}[Proof of Theorem~\ref{thm:main}: the large $|\alpha|$  case]
From~\eqref{eq:E:is:theta:derivative} and identity~\eqref{eq:vorticity:error}, we have that 
\[
\int_{\RR^3} |\Omega|^2 w \, dy = \tfrac 12 \alpha \int_{\RR^3} \p_\theta (|U|^2 + U\cdot y) \, w dy,
\]
where $w$ is the positive weight constructed in Proposition~\ref{prop:weight:construction}, with parameter $\eps \mapsto \frac 14$.

We first note that $\p_\theta |U|^2 = 2 U^r \p_\theta U^r + 2 U^\theta \p_\theta U^\theta + 2 U^z \p_\theta U^z$. Recalling the definition of $\mathcal{R}$ in~\eqref{eq:cal:R:def} we deduce the pointwise-in-$y$ bound
\[
\bigl| \p_\theta |U|^2 \bigr| \leq 6 |U| \, |\mathcal{R} U|.
\]
Similarly, since $\p_\theta(U \cdot y) = \p_\theta (r U^r + z U^z) = r \p_\theta U^r +  z \p_\theta U^z$, we have the pointwise-in-$y$ bound
\[
\bigl| \p_\theta (U\cdot y) \bigr| \leq 2 |y| \, |\mathcal{R} U|.
\] 
Combining the inequalities above with the $U$ bound in~\eqref{eq:decay}, the $w$ bound in~\eqref{eq:weight:decay} (since we fixed the small parameter at $\frac 14$, we have $M = M(C_{U,0})$), and the definition of $L^2_\mu$ in~\eqref{eq:L2:mu:def}, we arrive at 
\begin{align}
\int_{\RR^3} |\Omega|^2 w \, dy 
&\leq    |\alpha| \int_{\RR^3}   (3 |U|  + |y|) \,|\mathcal{R} U|  w dy 
\notag\\
&\leq    |\alpha| \int_{\RR^3}  w (3 C_{U,0}  + |y|) \mu^{-\frac 12} \, |\mathcal{R} U| \mu^{\frac 12}  dy
\notag\\
&\leq    |\alpha| \| \mathcal{R} U\|_{L^2_\mu} \, M
\left( \int_{\RR^3}  e^{-\frac 38 |y|^2}  (3 C_{U,0}  + |y|)^2 e^{\frac 14 |y|^2} dy\right)^{\frac 12}
\notag\\
&\leq    |\alpha| \| \mathcal{R} U\|_{L^2_\mu} \, M (1+ 3 C_{U,0})
\left( \int_{\RR^3}  e^{-\frac 18 |y|^2}  (1+ |y|)^2  dy\right)^{\frac 12}.
\label{eq:Omega:large:alpha:1}
\end{align}
One may verify that the integral term appearing in the last line of~\eqref{eq:Omega:large:alpha:1} is $\leq 100$. Thus, we deduce that there exists a constant $M^\prime = M^\prime(C_{U,0})>0$ (which is in particular independent of $\alpha$) such that 
\[
\| \Omega \sqrt{w} \|_{L^2}^2 \leq M^\prime |\alpha| \, \|\mathcal{R} U\|_{L^2_\mu}.
\]
Combining this bound with Proposition~\ref{prop:large:alpha} shows that for any $\eps \in (0,1]$, there exists $A_\eps = A_\eps(\eps,C_{U,0})$ such that 
\[
\| \Omega \sqrt{w} \|_{L^2}^2
\leq \eps M^\prime,
\qquad \mbox{whenever} \qquad 
|\alpha| \geq A_\eps.
\]
The proof is now concluded in analogy to the small $|\alpha|$ case.
From~\eqref{eq:weight:decay}, and with $\bar R = \bar R(C_{U,0}) >0$ as in Proposition~\ref{prop:local:vort} we deduce that there exists $m = m(C_{U,0})>0$ such that 
\[ 
\|\Omega\|_{L^2(B_{\bar R})}^2
\leq m^{-1} e^{\frac{5}{16} \bar R^2} \cdot \eps M^\prime
\]
whenever $|\alpha|\geq A_\eps$. Setting 
\[
\eps := \tfrac 12 
\min \bigl \{ (M^\prime)^{-1} \cdot m e^{-\frac{5}{16} \bar R^2} \cdot C_{\Omega}^{-2}, 1 \bigr \},
\]
so that $\eps$ is a parameter that only depends on $C_{U,0}$, we obtain that \eqref{eq:small:enstrophy} holds for all $|\alpha| \geq A_{\eps}$, and thus $U\equiv 0$ by Proposition~\ref{prop:local:vort}.
\end{proof}

\section{DSS and RDSS: proof of the main result for $\alpha$ small and $\lambda$ close to $1$}
\label{sec:DSS:RDSS}
Setting $\alpha=0$ in Theorem~\ref{thm:main:RDSS} recovers the statement of Theorem~\ref{thm:main:DSS}, and hence we only discuss the case of rotated discretely self-similar solutions, with small period $S$ and small (absolute) angular speed $|\alpha|$. The arguments presented here closely follow those in Sections~\ref{sec:pressure:and:derivative:bounds}--\ref{sec:alpha:small}, and thus we omit details which are identical. 

Without loss of generality, we let $\underline{\alpha} \leq 1$ and $\underline{\lambda} \leq e^{\frac 12}$, so that $|\alpha| \leq 1$ and $S \leq 1$, due to~\eqref{eq:RDSS:period}.
 
\subsection{Uniform-in-time Type~I bounds for higher derivatives}
We seek a statement analogous to Lemma~\ref{lem:gradient}, with bounds for $\nabla U, \nabla^2 U$, and $P$, which are \textit{uniform-in-time} for $s\in [0,S]$ and only depend on $C_{U,0}$. In the process, we also obtain bounds for $\p_s U$ and $\p_s \nabla U$. 

\begin{lemma}
\label{lem:many:gradients}
Under the assumptions of Theorem~\ref{thm:main:RDSS} (in particular $|\alpha|\leq 1$), the profile $U(\cdot,s)$ satisfies the pointwise bound~\eqref{eq:decay} for all $s\in [0,S]$, with the same constant $C_{U,0}>0$ from the Type~I upper bound~\eqref{eq:type:I}. The space-gradients of $U$ satisfy~\eqref{eq:decay:grad} and the pressure $P(\cdot,s)$ obeys estimate~\eqref{eq:decay:pressure} for every $s\in[0,S]$, with the same $(\alpha,S)$-independent constants $C_{U,1}, C_{U,2}, C_{P,0}$ as in Lemma~\ref{lem:gradient}. Moreover, there exists an $(\alpha,S)$-independent constant $C_{U,s} = C_{U,s}(C_{U,0})>0$ such that
\begin{equation}
\label{eq:ps:U:bound}
|\p_s U(y,s)| \le \frac{C_{U,s}}{1+|y|},
\qquad
|\nabla \p_s U(y,s)| \le \frac{C_{U,s}}{1+|y|^2},
\end{equation}
for all $(y,s)\in\RR^3\times[0,S]$.
\end{lemma}
\begin{proof}[Proof of Lemma~\ref{lem:many:gradients}]
The proof is nearly identical to that of Lemma~\ref{lem:gradient}, so we only emphasize the differences. Since $R(\alpha s) \in SO(3)$, using~\eqref{eq:RDSS:ansatz} we obtain that $|U(y,s)| = \sqrt{-t}\,|u(x,t)|$ for all $(y,s) \in \RR^3 \times [0,S]$. The Type~I bound~\eqref{eq:type:I} then gives $|U(y,s)| \leq C_{U,0} (1+|y|)^{-1}$ for all $(y,s)$, thereby proving~\eqref{eq:decay}.

The qualitative smoothness of $(U,P)$ and the pressure representation formula $P= R_i R_j(U^i U^j)$ are established exactly as in Lemma~\ref{lem:gradient}.

For the bounds on $\nabla U$ and $\nabla^2 U$, we run the interior-regularity argument of Lemma~\ref{lem:gradient}, making sure that the constants only depend on $C_{U,0}$. Fix $(y,s) \in \RR^3 \times [0,S]$ and set $x_* = R(\alpha s)\,e^{-\frac s2} y$, $t_* = -e^{-s}$, $r = \tfrac12\, e^{-\frac s2}\max\{|y|,1\}$. Then, the function\footnote{Here we are using the RDSS ansatz~\eqref{eq:RDSS:total} to extend $u$ backwards in time from $\RR^3\times [-1,0)$ to a solution of~\eqref{eq:NSE} on $\RR^3 \times (-\infty,0)$, while retaining the Type~I bound $|u(x,t)|\le C_{U,0}(|x|+\sqrt{-t})^{-1}$ there, with the same constant $C_{U,0}$.} 
$\tilde u(\tilde x,\tilde t) := r\, u\big(x_* + r \,\tilde x,\; t_* + r^2\,\tilde t\big)$ 
solves~\eqref{eq:NSE} on the unit parabolic cylinder $Q_1 = B_1\times(-1,0\,]$. For $(\tilde x,\tilde t) \in Q_1$ we have 
$(x,t) = (x_* + r \,\tilde x,\; t_* + r^2\,\tilde t )$ satisfies\footnote{To see this, note that $-t = - t_* - r^2\,\tilde t = e^{-s} - r^2\tilde t \ge e^{-s}$, hence $\sqrt{-t}\ge e^{- \frac s2}$. Similarly,  $|x| = |x_* + r \tilde x| \ge |x_*| - r = e^{-\frac s2} |y| - r$. If $|y|\ge 1$ then we obtain $|x|\ge r$. If $|y|<1$ then $\sqrt{-t}\ge e^{-\frac s2} = 2r$. In either case $|x|+\sqrt{-t}\ge r$.}  $|x|+\sqrt{-t}\ge r$, so that~\eqref{eq:type:I} gives $\|\tilde u\|_{L^\infty(Q_1)} \leq C_{U,0}$, a bound independent of $y$, $s$ and $\alpha$. A quantitative version of the interior regularity for Navier-Stokes~\cite{Serrin62}, gives for each $k\ge 1$ a constant $C_k = C_k(C_{U,0})$ with $|\nabla_{\tilde x}^k \tilde u(0,0)|\le C_k$. Undoing the rescaling we have $\nabla_x^k u(x_*,t_*) = r^{-(k+1)}\nabla_{\tilde x}^k\tilde u(0,0)$, and recall that by the definition of $(x_*,t_*)$ we have $U(y,s) = e^{- \frac s2}R(-\alpha s)u(x_*,t_*)$. Since each $y$-derivative contributes a factor $e^{- \frac s2}$, we deduce $|\nabla_y^k U(y,s)| = e^{-\frac{(k+1)s}{2}}\,|\nabla_x^k u(x_*,t_*)| = e^{-\frac{(k+1)s}{2}} r^{-(k+1)} \,|\nabla_{\tilde x}^k \tilde u(0,0)| \leq 2^{k+1} C_k \max\{|y|,1\}^{-(k+1)} $. Since $C_k$ only depends on $C_{U,0}$, this establishes the desired bound 
\begin{equation}
\label{eq:decay:all:k}
|\nabla_y^k U(y,s)| \leq    
\frac{C_{U,k}}{(1+|y|)^{k+1}},\qquad k\ge 0,
\end{equation}
for all $(y,s) \in \RR^3 \times [0,S]$, uniformly in $\alpha$.  For $k=1,2$ this is the claimed estimate~\eqref{eq:decay:grad}.

Since $P = R_iR_j\big(U^iU^j\big)$ plus a function of time, which without loss of generality we take to vanish identically, the bounds in~\eqref{eq:decay:all:k} and the singular-integral representation of $R_i R_j$ imply
\begin{equation}
\label{eq:decay:pressure:k}
|\nabla_y^k P(y,s)| \leq 
\frac{C_{P,k}}{(1+|y|)^{\,2-\sigma+k}}, \qquad k\ge 0,
\end{equation}
where $C_{P,k}$ depends on $C_{U,0}$, $k$, and $\sigma >0$, and is hence independent of $s$ and $\alpha$.

In order to prove the first bound in~\eqref{eq:ps:U:bound}, we use~\eqref{eq:profile:RDSS:a} to write 
$\p_s U = \Delta U - \alpha\big(JU - (Jy\cdot\nabla)U\big) - \tfrac12 U - \tfrac12 (y\cdot\nabla)U - (U\cdot\nabla)U - \nabla P$, and bound each term on the right side using~\eqref{eq:decay:all:k} and~\eqref{eq:decay:pressure:k}. The $\alpha$ independence of the upper bound is a consequence of the standing assumption $|\alpha|\leq 1$. For the second bound in~\eqref{eq:ps:U:bound} we differentiate~\eqref{eq:profile:RDSS:a} with respect to $y$, and proceed similarly.
\end{proof}

\subsection{The time-averaged profile and fluctuations from the mean}
We use the following notation $\mean{f}_s$ for the \emph{time-average} of a function $f$ over the period $[0,S]$, and denote by $\tilde f$ the associated \emph{fluctuation} from the mean:
\begin{equation}\label{eq:time:mean}
\mean{f}_s := \frac{1}{S}\int_0^S f(\cdot,s)\, ds,
\qquad
\tilde f := f - \mean{f}_s.
\end{equation}
In particular, $\mean{\tilde f}_s = 0$. An immediate consequence of this definition and of the bounds in~\eqref{eq:ps:U:bound} is that the fluctuation $\tilde U = U - \mean{U}_s$ is \emph{linearly} small in the period.

\begin{lemma} 
\label{lem:fluct}
Under the assumptions of Theorem~\ref{thm:main:RDSS} with $|\alpha|\le 1$, let $C_{U,s} = C_{U,s}(C_{U,0})>0$ be the constant from~\eqref{eq:ps:U:bound}. Then, the time-fluctuation of the profile satisfies
\begin{equation}
\label{eq:fluct:pointwise}
|\tilde U(y,s)| \le \frac{C_{U,s}\, S}{1+|y|},
\qquad
|\nabla \tilde U(y,s)| \le \frac{C_{U,s}\, S}{1+|y|^2},
\end{equation}
for all $(y,s) \in \RR^3 \times [0,S]$. In particular, we have 
\begin{equation}\label{eq:fluct:L2}
\| \tilde\Omega(\cdot,s) \|_{L^2(\RR^3)} \le 10 C_{U,s} \, S
\qquad\mbox{for every} \qquad s \in [0,S].
\end{equation}
Additionally, the \emph{time-averaged profile} $\mean{U}_s= \mean{U}_s(y)$ is smooth, divergence-free, and satisfies the bounds~\eqref{eq:decay} and~\eqref{eq:decay:grad}, with the same $C_{U,0}$-dependent constants.
\end{lemma}

\begin{proof}[Proof of Lemma~\ref{lem:fluct}]
Since $\mean{\tilde U}_s = 0$, we may write $\tilde U(\cdot,s) = \tfrac1S\int_0^S\big(U(\cdot,s) - U(\cdot,s^\prime)\big)\, ds^\prime = \tfrac1S\int_0^S\!\int_{s^\prime}^s \p_\tau U(\cdot,\tau)\, d\tau\, ds^\prime$, whence $|\tilde U(\cdot,s)| \le S\,\sup_\tau |\p_\tau U|$. The bounds~\eqref{eq:fluct:pointwise} now follow from~\eqref{eq:ps:U:bound}. 

Since $\tilde \Omega = \nabla \times \tilde U$, we have $\|\tilde\Omega(\cdot,s)\|_{L^2}\le \sqrt2\,\|\nabla\tilde U(\cdot,s)\|_{L^2} \le \sqrt{2} C_{U,s} S \, \big\|(1+|y|^2)^{-1}\big\|_{L^2(\RR^3)} \leq 10 C_{U,s} S$, which is~\eqref{eq:fluct:L2}.

Finally, the claims about $\mean{U}_s$ follow directly from the definition~\eqref{eq:time:mean} and from Lemma~\ref{lem:many:gradients}.
\end{proof}

\subsection{The time-dependent equation satisfied by the Bernoulli function}
We use the same definition for the head-pressure (Bernoulli function) as in~\eqref{eq:Bernoulli:def}, except that all fields are now time dependent: 
\[
\Pi(y,s) := P(y,s) + \tfrac 12 |U(y,s)|^2 + \tfrac 12 y \cdot U(y,s)
.
\]
Analogously to~\eqref{eq:L:def}, we denote by $L$ the elliptic operator with  time-dependent coefficients
\[
L := L(s) = - \Delta + \bigl( U(y,s) + \tfrac 12 y\bigr) \cdot \nabla \, ,
\]
and define a time-independent version associated to the time-averaged profile
\[
\bar L := - \Delta + \bigl( \mean{U}_s(y) + \tfrac 12 y\bigr) \cdot \nabla \, .
\]
The difference between the above two elliptic operators is a pure transport derivative: 
\[
L - \bar L = \tilde U(y,s) \cdot \nabla.
\]

With this notation, we derive a time-dependent version of~\eqref{eq:L:Pi:alpha}, by computing $L \Pi$; the resulting identity acquires exactly one new term when compared to the time-independent setting.
 Assume that $(U,P)$ are smooth solutions of~\eqref{eq:profile:RDSS} which are $[0,S]-$periodic in time. Then, by repeating the computation in~\eqref{eq:L:Pi:alpha} we obtain\footnote{The only difference is that the $\nabla P$ term carries the extra factor of $- \p_s U$, which in turn causes the new term $-(U + \tfrac 12 y) \cdot \p_s U$ to appear in \eqref{eq:L:Pi:alpha:s}.} that
\begin{equation}
\label{eq:L:Pi:alpha:s}
L\Pi + |\Omega|^2 = \alpha E - \big(U + \tfrac12 y\big)\cdot \p_s U,
\end{equation}
where $E$ is the error term  defined in~\eqref{eq:E:def}. 
Recalling~\eqref{eq:E:is:theta:derivative}, we may write~\eqref{eq:L:Pi:alpha:s} in the more informative form
\[
L \Pi + |\Omega|^2 = \tfrac 12 (\alpha \p_\theta - \p_s) (|U|^2 + U\cdot y).
\]
The most convenient form of the above identity involves replacing the operator $L$ with its time-independent version $\bar L$, leading to
\begin{equation}
\bar L \Pi + |\Omega|^2 = - \tilde U \cdot \nabla  \Pi + \tfrac 12 (\alpha \p_\theta - \p_s) (|U|^2 + U\cdot y).
\label{eq:L:Pi:alpha:s:good}
\end{equation}
At this point the proof strategy becomes clear: make the terms on the right side of~\eqref{eq:L:Pi:alpha:s:good} small either by taking $|\alpha|$ small or by taking the period to be small, and use the left side of~\eqref{eq:L:Pi:alpha:s:good} to deduce a weighted $L^2$ bound for the vorticity, where the weight is taken to lie in the kernel of  $\bar L^*$.

\subsection{A pointwise-in-time local enstrophy bound}\label{subsec:pwlocenst}
By Lemma~\ref{lem:fluct}, the time-averaged profile $\mean{U}_s$ is smooth, divergence-free, and satisfies~\eqref{eq:decay}--\eqref{eq:decay:grad}. Proposition~\ref{prop:weight:construction} 
thus applies to the operator $\bar L := -\Delta + (\mean{U}_s + \tfrac12 y)\cdot\nabla$ and its formal adjoint operator $\bar L^*$ (with respect to $L^2(\RR^3)$). With $\eps = \tfrac14$, Proposition~\ref{prop:weight:construction} yields  constants $m,M>0$ depending only on $C_{U,0}$, and a strictly positive $C^2$ weight $\bar w \colon \RR^3\to\RR_+$ with
\begin{equation}
\label{eq:bar:w}
\bar L^* \bar w = 0,
\qquad
m\, e^{-\frac{5}{16}|y|^2} \le \bar w(y) \le M\, e^{-\frac{3}{16}|y|^2},
\end{equation}
where $\bar L^* := -\Delta - \big(\mean{U}_s + \tfrac12 y\big)\cdot\nabla - \tfrac32$. We stress that $\bar w = \bar w(y)$ is \emph{independent of $s$}, and  the constants $m$ and $M$ are \emph{independent of $s$ and $\alpha$}.

We multiply \eqref{eq:L:Pi:alpha:s:good} by the weight $\bar w$ (fixed as in~\eqref{eq:bar:w}), and then integrate in $y$ over $\RR^3$, to obtain
\begin{align*}
\int_{\RR^3} |\Omega(\cdot,s)|^2 \, \bar w \, dy 
&= - \int_{\RR^3} \tilde U \cdot \nabla \Pi \, \bar w \, dy
+ \tfrac{\alpha}{2} \int_{\RR^3} \p_\theta \bigl( |U|^2 + U \cdot y \bigr) \, \bar w \, dy
- \tfrac 12 \tfrac{d}{ds} \int_{\RR^3} \bigl( |U|^2 + U \cdot y \bigr) \, \bar w \, dy,
\end{align*}
pointwise for $s\in [0,S]$. Here we have crucially used that $\bar w$ is independent of $s$. 

The \emph{periodicity in $s$ of the profile $U$} is now used to make the last term in the above identity disappear, upon integration in $s$:
\begin{align*}
\int_0^S\!\! \int_{\RR^3} |\Omega(\cdot,s)|^2 \, \bar w \, dy ds
&= - \int_0^S\!\! \int_{\RR^3} \tilde U \cdot \nabla \Pi \, \bar w \, dyds
+  \alpha \int_0^S\!\! \int_{\RR^3} E \, \bar w \, dyds.
\end{align*}
Here we have also recalled that the error term $E$ is given in~\eqref{eq:E:is:theta:derivative}. Using~\eqref{eq:E:bnd}, \eqref{eq:decay:all:k}, \eqref{eq:decay:pressure:k}, \eqref{eq:fluct:pointwise}, \eqref{eq:bar:w}, and recalling that $\nabla \Pi = \nabla P + \nabla U \cdot U  + \frac 12 U + \frac 12 \nabla U \cdot y$, we deduce from the above identity that
\begin{align}
\frac{1}{S} \! \int_0^S\!\! \int_{\RR^3} |\Omega(\cdot,s)|^2 \, \bar w \, dy ds
\leq  \frac{1}{S} \! \int_0^S\!\! \int_{\RR^3} |\tilde U| \, |\nabla \Pi| \, \bar w \, dyds
+   \frac{|\alpha|}{S} \! \int_0^S\!\! \int_{\RR^3} |E| \, \bar w \, dyds
\leq C^\prime \bigl( S + |\alpha| \bigr) ,
\label{eq:weighted:total}
\end{align}
where 
\[
C^\prime:= M  ( C_{P,1} + C_{U,0} + C_{U,1} + C_{U,0} C_{U,1})C_{U,s}  \int_{\RR^3} \tfrac{1}{(1+|y|)^2}  e^{-\frac{3}{16} |y|^2} \, dy  + C_E M \int_{\RR^3}   e^{-\frac{3}{16} |y|^2} dy 
\] 
is a constant that only depends on $C_{U,0}$; it is independent of $\alpha$ and $S$.

Restricting the space integral on the left side of~\eqref{eq:weighted:total} to $B_{\bar R}$ (where $\bar R = \sqrt{8 C_{U,1}}$ is as in Proposition~\ref{prop:local:vort}), and using the Gaussian lower bound on $\bar w$ from~\eqref{eq:bar:w}, we obtain the \emph{time-averaged} ball estimate
\[
\frac1S \! \int_0^S \|\Omega(\cdot,s)\|_{L^2(B_{\bar R})}^2\, ds
\leq  e^{\frac{5}{16}\bar R^2} m^{-1} \cdot 
\frac1S \! \int_0^S\!\!\int_{\RR^3}|\Omega|^2\bar w\, dy\, ds
\le C^\prime e^{\frac{5}{16}\bar R^2} m^{-1}  \bigl( S + |\alpha| \bigr)
.
\]
Jensen's inequality then gives 
\[
\|\mean{\Omega}_s\|_{L^2(B_{\bar R})}^2 
\leq \frac1S\int_0^S\|\Omega\|_{L^2(B_{\bar R})}^2 ds \le C^\prime e^{\frac{5}{16}\bar R^2} m^{-1}  ( S + |\alpha|  ).
\] 
Finally, writing $\Omega(\cdot,s) = \mean{\Omega}_s + \tilde\Omega(\cdot,s)$ and invoking the fluctuation bound~\eqref{eq:fluct:L2},  we conclude that for \emph{every} $s$,
\begin{equation*}
\|\Omega(\cdot,s)\|_{L^2(B_{\bar R})}
\le \|\mean{\Omega}_s\|_{L^2(B_{\bar R})} + \|\tilde\Omega(\cdot,s)\|_{L^2(\RR^3)}
\le \sqrt{C^\prime} e^{\frac{5}{32}\bar R^2} m^{-\frac 12}  \sqrt{S + |\alpha|} + 10 C_{U,s} S.
\end{equation*}
Since $S\leq 1$ and $|\alpha|\leq 1$, the above estimate then gives the pointwise-in-time local enstrophy bound
\begin{equation}
\|\Omega(\cdot,s)\|_{L^2(B_{\bar R})}
\leq C^{\prime\prime}  \sqrt{S + |\alpha|},
\label{eq:pointwise:ball}
\end{equation}
for all $s\in [0,S]$, where
\[
C^{\prime\prime}:= \sqrt{C^\prime} e^{\frac{5}{32}\bar R^2} m^{-\frac 12} + 10 C_{U,s}.
\]
Note that $C^\prime$, $\bar R$, $m$, and $C_{U,s}$ only depend on $C_{U,0}$, and thus so does $C^{\prime\prime}$.

\subsection{Proof of the main result for RDSS solutions when $|\alpha|+S$ is small}\label{subsec:mainRDSSsmall}
Analogously to~\eqref{eq:vort}, taking the curl of~\eqref{eq:profile:RDSS:a} gives the vorticity evolution, which is just~\eqref{eq:vort} with an extra $\p_s\Omega$ term on the left side:
\[
\p_s \Omega + \alpha \bigl( J\Omega-  Jy\cdot\nabla\Omega\bigr) + \Omega + \tfrac{1}{2}y\cdot\nabla\Omega -\Delta\Omega +U\cdot\nabla \Omega=\Omega\cdot\nabla U.   
\]
Integrating this evolution against $\Omega$ (justified in light of~\eqref{eq:decay:grad}) leads to
\begin{equation}\label{eq:enstrophy:s}
\tfrac12\tfrac{d}{ds}\|\Omega\|_{L^2(\RR^3)}^2 + \tfrac14\|\Omega\|_{L^2(\RR^3)}^2 + \|\nabla\Omega\|_{L^2(\RR^3)}^2 = \int_{\RR^3}\Omega^i \p_i U^j \Omega^j\, dy,
\end{equation}
for all $s \in [0,S]$. As in the proof of Proposition~\ref{prop:local:vort}, by~\eqref{eq:decay:grad} (which holds due to Lemma~\ref{lem:many:gradients}) we have that $|\nabla U(y,s)|\le\tfrac18$ for $|y|\ge\bar R = \sqrt{8C_{U,1}}$ and all $s \in [0,S]$. Absorbing the corresponding part of the vortex-stretching integral appearing on the right side of~\eqref{eq:enstrophy:s}, and estimating the remainder over $B_{\bar R}$ exactly as in~\eqref{eq:vort:est}, pointwise for $s \in [0,S]$ we obtain
\[
\tfrac12\tfrac{d}{ds}\|\Omega\|_{L^2(\RR^3)}^2
+ 
\tfrac18\Big(\|\Omega\|_{L^2(\RR^3)}^2 + \|\nabla\Omega\|_{L^2(\RR^3)}^2\Big)
\leq C_\Omega\, \|\Omega(\cdot,s)\|_{L^2(B_{\bar R})}\Big(\|\Omega\|_{L^2(\RR^3)}^2 + \|\nabla\Omega\|_{L^2(\RR^3)}^2\Big),
\]
where $C_{\Omega}>0$ is a universal constant. 
Integrating the above inequality in $s\in [0,S]$, using the periodicity in time of $\Omega$, and appealing to the pointwise-in-$s$ bound~\eqref{eq:pointwise:ball}, we arrive at 
\[
\int_0^S \!\! \Big(\|\Omega\|_{L^2(\RR^3)}^2 + \|\nabla\Omega\|_{L^2(\RR^3)}^2\Big) ds
\leq 8 C_\Omega\, C^{\prime\prime}  \sqrt{S + |\alpha| } \, 
\int_0^S \!\!  \Big(\|\Omega\|_{L^2(\RR^3)}^2 + \|\nabla\Omega\|_{L^2(\RR^3)}^2\Big) ds.
\]
Therefore, if $|\alpha|,S \leq 1$ satisfy 
\[
2 \log(\lambda) + |\alpha| = S+ |\alpha| <  \frac{1}{(8 C_\Omega\, C^{\prime\prime})^2},
\] 
with $C_{\Omega}$ a universal constant and $C^{\prime\prime}$ a constant that only depends on $C_{U,0}$, by the previous inequality we obtain
\[
\int_0^S \Big(\|\Omega(\cdot,s)\|_{L^2(\RR^3)}^2 + \|\nabla\Omega(\cdot,s)\|_{L^2(\RR^3)}^2\Big) ds = 0.
\]
Hence $\Omega\equiv 0$ on $\RR^3\times[0,S]$. Since $\nabla\cdot U(\cdot,s) = 0$, $\nabla\times U(\cdot,s) = \Omega(\cdot,s) = 0$, and $U(\cdot,s) \in L^4(\RR^3)$ for each $s$, we obtain that $U\equiv 0$ on $\RR^3\times[0,S]$. This concludes the proof of Theorem~\ref{thm:main:RDSS} when $|\alpha|$ is sufficiently small, and $\lambda>1$ is sufficiently close to $1$.

\section{DSS and RDSS: proof of the main result for $\alpha$ large and $\lambda$ close to $1$}
\label{sec:DSS:RDSS:Ben}

The proof of the large $|\alpha|$ case in Theorem~\ref{thm:main:RDSS} follows primarily by making minor adjustments to the proof of the small $|\alpha|$ case in Theorem~\ref{thm:main:RDSS}, and combining these arguments with some of the components of the large $|\alpha|$ portion of Theorem~\ref{thm:main}. 

We remark that if $\underline{\lambda}$ is close to $1$, the assumption $1<\lambda < \underline{\lambda}^{\frac{1}{1+\alpha^2}}$ appearing in Theorem~\ref{thm:main:RDSS} is equivalent to the statement that the period $S = 2 \log(\lambda)$ satisfies 
\begin{equation}
\label{eq:moral:of:the:story}
(1+\alpha^2) S < 2 \log \underline{\lambda} \ll 1.
\end{equation}
In other words, the condition on $\lambda$ ensures that the temporal period $S$ is sufficiently small relative to $\alpha$. This assumption is required for proving the following variation of Lemma~\ref{lem:many:gradients}, which is the main additional technical input required for the proof of the large $|\alpha|$ case in Theorem~\ref{thm:main:RDSS}.
\begin{lemma}
\label{lem:many:gradients2}
Under the assumptions of Theorem~\ref{thm:main:RDSS}, the profile $U(\cdot,s)$ satisfies the pointwise bound~\eqref{eq:decay} for all $s\in [0,S]$, with the same constant $C_{U,0}>0$ from the Type~I upper bound~\eqref{eq:type:I}. The space-gradients of $U$ satisfy~\eqref{eq:decay:grad} and the pressure $P(\cdot,s)$ obeys estimate~\eqref{eq:decay:pressure} for every $s\in[0,S]$, with the same $(\alpha,S)$-independent constants $C_{U,1}, C_{U,2}, C_{P,0}$ as in Lemma~\ref{lem:gradient}. Moreover, there exists an $(\alpha,S)$-independent constant $C_{U,s} = C_{U,s}(C_{U,0})>0$ such that
\begin{equation}
\label{eq:ps:U:bound2}
|\p_s U(y,s)| \le \frac{C_{U,s}S(1+|\alpha|)^2}{1+|y|},
\qquad
|\nabla \p_s U(y,s)| \le \frac{C_{U,s}S(1+|\alpha|)^2}{1+|y|^2},
\end{equation}
for all $(y,s)\in\RR^3\times[0,S]$.
\end{lemma}
\begin{proof}[Proof of Lemma~\ref{lem:many:gradients2}]
The starting point of the proof is the same as Lemma \ref{lem:many:gradients}. By the same argument, one can deduce the bounds
\begin{equation}
\label{eq:decay:all:k2}
|\nabla_y^k U(y,s)| \leq    
\frac{C_{U,k}}{(1+|y|)^{k+1}},\qquad |\nabla_y^k P(y,s)| \leq 
\frac{C_{P,k}}{(1+|y|)^{\,2-\sigma+k}}, \qquad k\ge 0,
\end{equation}
for all $(y,s) \in \RR^3 \times [0,S]$, uniformly in $\alpha$. Using these bounds, we deduce from the equation \eqref{eq:profile:RDSS:a} the bound 
\begin{equation*}
|\partial_s\nabla^k U(y,s)|\leq \frac{C_{U,s}(1+|\alpha|)}{1+|y|^{k+1}},
\qquad k \in \{ 0,1,2,3\}.    
\end{equation*}
Arguing as in Lemma \ref{lem:gradient}, these bounds then imply
\begin{equation*}
|\nabla^{k}\partial_s P (y,s)|\leq \frac{C_{U,s}(1+|\alpha|)}{(1+|y|)^{2-\sigma+k}},
\qquad k \in \{ 0,1,2\}.    
\end{equation*}
A-priori, these bounds are worse than \eqref{eq:ps:U:bound2} because we do not have the smallness factor $S$ to compensate for the large $\alpha$ parameter. However, by applying $\partial_s$ to \eqref{eq:profile:RDSS:a} and using the above bounds, along with \eqref{eq:decay:all:k2}, we obtain
\begin{equation*}
|\partial_s^2U(y,s)|\leq \frac{C_{U,s}(1+|\alpha|)^2}{1+|y|},\qquad |\nabla\partial_s^2U(y,s)|\leq \frac{C_{U,s}(1+|\alpha|)^2}{1+|y|^2}.   
\end{equation*}
Recalling that $U$ is $[0,S]$-periodic in time, Poincar\'e's inequality in $s$ then gives the bound \eqref{eq:ps:U:bound2}.
\end{proof}
The main purpose of the above lemma is to provide us with the following analogue of Lemma~\ref{lem:source:bound}.
\begin{lemma}\label{lem:source:boundRDSS}
Let $u$ satisfy the conditions of Theorem~\ref{thm:main:RDSS} with parameters $\alpha, \overline{\alpha},\lambda$, and $\underline{\lambda}$.  Assume that the smooth incompressible vector field $U\colon \RR^3 \times [0,S] \to \RR^3$ satisfies the bounds~\eqref{eq:decay:all:k2}. Define
\begin{equation}
\mathcal{N} := - (U \cdot \nabla) U - \nabla P-\partial_sU,
\label{eq:source:def2}
\end{equation}
where $P = R_i R_j (U^i U^j)$ (as in the proof of Lemma~\ref{lem:gradient}). For any $\eps \in (0,1]$, there exist constants $C_\eps = C_\eps(\eps,C_{U,0})>0$ and $\delta_{\eps}=\delta(\eps, C_{U,0})>0$ such that 
\begin{equation}
\| (\mathcal{N})_a \|_{L^2_\mu} \leq \eps + C_\eps   \|(U)_a\|_{L^2_\mu} + C_\eps   \| \nabla (U)_a \|_{L^2_\mu},
\label{eq:source:bound2}
\end{equation}
holds uniformly for all $s\in [0,S]$, 
provided that  $S (1+|\alpha|)^2 \leq \delta_\eps$.
\end{lemma}   
\begin{proof}[Proof of Lemma~\ref{lem:source:boundRDSS}]
The estimate for $-(U\cdot\nabla U)-\nabla P$ is identical to that in Lemma~\ref{lem:source:bound}, except that we take smallness factor $\frac{\eps}{2}$ instead of $\eps$. To deal with the remaining term $\partial_sU$, we simply apply Lemma \ref{lem:many:gradients2} to obtain (say), $\|(\partial_s U)_a \|_{L^2_{\mu}} \leq \|\partial_sU\|_{L^2_{\mu}} \leq C_{U,s} S (1+|\alpha|)^2 \| (1+|y|)^{-1}\|_{L^2_\mu}  \leq \frac{\eps}{2} $.    
\end{proof}
By arguing exactly as in the proof of Proposition \ref{prop:large:alpha}, we obtain smallness of $\alpha\mathcal{R}U$ in $L_{\mu}^2$,
\begin{proposition}
\label{prop:large:alpha2}
Let $u$ satisfy the conditions of Theorem~\ref{thm:main:RDSS}  with parameters $\alpha, \overline{\alpha},\lambda$ and $\underline{\lambda}$.  Suppose that the profiles $(U,P)$ satisfy the bounds~\eqref{eq:decay:all:k2}. For any $\eps \in (0,1]$, there exist constants $A_\eps = A(\eps,C_{U,0})\geq 1$ and $\delta_{\eps}=\delta(\eps,C_{U,0})$, such that for all $|\alpha|\geq A_{\eps}$, if $S (1+|\alpha|)^2 \leq \delta_\eps$ we have 
\begin{equation}
|\alpha|  \, \| \mathcal{R} U\|_{L^2_\mu}
\leq  
\eps,
\label{eq:large:alpha:main:2}
\end{equation}
uniformly for all $s\in[0,S]$.
\end{proposition}

To conclude the proof of the large $|\alpha|$ case in Theorem~\ref{thm:main:RDSS}, one simply follows  the same reasoning as in Section~\ref{sec:DSS:RDSS}; we only sketch the minor adjustments that are needed. Fix $0<\eps=\eps(C_{U,0})\ll1$, to be chosen below, and set
\[
\eta:=S(1+|\alpha|)^2 = 2\log (\lambda) (1+|\alpha|)^2.
\]
By first choosing $\overline{\alpha}$ sufficiently large, we may ensure that $|\alpha|\geq\overline{\alpha}\geq A_\eps$, where $A_\eps$ is as in Proposition~\ref{prop:large:alpha2}. We then choose $\underline{\lambda}>1$ sufficiently close to $1$, in terms of $\eps$ and $C_{U,0}$ only, as follows: the condition $1<\lambda<\underline{\lambda}^{\frac{1}{1+\alpha^2}}$ implies
\[
\eta
=2\log(\lambda)(1+|\alpha|)^2
\leq 2\log(\underline{\lambda})\frac{(1+|\alpha|)^2}{1+\alpha^2}
\leq 4\log(\underline{\lambda}),
\]
and hence (by decreasing $\underline{\lambda}$ if necessary) we may assume that $\eta\leq\delta_\eps$, as required in Proposition~\ref{prop:large:alpha2}, and may use $\eta\ll1$ as a smallness parameter in the proof, instead of $S$ (see~\eqref{eq:moral:of:the:story}).

After also ensuring that $S\leq1$, Lemma~\ref{lem:many:gradients2} and the proof of Lemma~\ref{lem:fluct} give
\[
|\tilde U(y,s)|\les \frac{ \eta}{1+|y|},
\qquad
|\nabla\tilde U(y,s)|\les \frac{ \eta}{1+|y|^2},
\qquad
\|\tilde\Omega(\cdot,s)\|_{L^2(\RR^3)}\les  \eta,
\]
uniformly for $s\in[0,S]$, where the implicit constants depend only on $C_{U,0}$. Then, we construct the weight $\bar w$ associated to the time-averaged profile $\mean{U}_s$, exactly as in~\eqref{eq:bar:w}. We multiply~\eqref{eq:L:Pi:alpha:s:good} by $\bar w$, integrate in space and then average over one temporal period. The $\p_s$-derivative term vanishes by periodicity, while the term involving $\tilde U\cdot\nabla\Pi$ is bounded as $\les \eta$, using the above $\tilde U$ estimates and~\eqref{eq:decay:all:k2}. For the remaining term, we use~\eqref{eq:E:is:theta:derivative}, the Gaussian upper bound for $\bar w$, and the argument in~\eqref{eq:Omega:large:alpha:1}, together with Proposition~\ref{prop:large:alpha2}, to obtain
\[
\frac1S\int_0^S\!\!\int_{\RR^3}|\Omega|^2\bar w\,dy\,ds
\les
\eta
+
\frac{|\alpha|}{S}\int_0^S
\|\mathcal{R}U(\cdot,s)\|_{L^2_\mu}\,ds
\les \eta+\eps.
\]
Using the Gaussian lower bound for $\bar w$, Jensen's inequality, and the above estimate for $\tilde\Omega$, we conclude (with $\bar R$ as in Proposition~\ref{prop:local:vort}) that
\[
\|\Omega(\cdot,s)\|_{L^2(B_{\bar R})}
\les  \sqrt{\eta+\eps}+ \eta,
\]
uniformly for $s\in[0,S]$; the implicit constant  depends only on $C_{U,0}$. To conclude, we first fix $\eps=\eps(C_{U,0})$ sufficiently small, and then choose $\underline{\lambda}=\underline{\lambda}(C_{U,0})$ sufficiently close to $1$, so that the right side of the above estimate is smaller than the universal threshold required in the argument of Subsection~\ref{subsec:mainRDSSsmall}. Integrating the enstrophy inequality~\eqref{eq:enstrophy:s} over one period and using the periodicity of $\Omega$, we deduce that $\Omega\equiv0$, completing the proof of the large $|\alpha|$ case in Theorem~\ref{thm:main:RDSS}.

\section{Approximate local self-similarity implies no Type~I blowup}
\label{sec:local}
The goal of this section is to prove Theorem~\ref{thm:local:theorem}. Recall that for a given $z_0 = (x_0,t_0)$, parabolic cylinders are denoted as
\[
Q_r(z_0) := B_r(x_0)\times(t_0 - r^2, t_0], \qquad Q_r := Q_r(0,0).
\]
The interior of a parabolic cylinder will be denoted as $Q_r^{\mathsf{o}}(z_0)$.

Given a smooth solution $(u,p)$ of~\eqref{eq:NSE} on $B_1\times[-1,0)$, we define similarity variables and profiles in the usual way
\begin{subequations}
\label{eq:local:ss:variables}
\begin{align}
&y = \frac{x}{\sqrt{-t}},   &&s = -\log(-t), \\
&U(y,s) := \sqrt{-t}\, u(x,t),   &&P(y,s) := (-t)\, p(x,t).
\end{align}
\end{subequations}
It follows that $(U, P)$ is well-defined on the set $\{(y,s) \in \RR^3 \times [0,\infty) \colon |y| < e^{\frac{s}{2}}\}$; on this set,\footnote{We do not make global-in-space or decay assumptions on the solutions; all objects in this section are defined only on the indicated regions of spacetime.} it solves the time-dependent self-similar Navier-Stokes system (cf.~\eqref{eq:profile:RDSS} with $\alpha = 0$):
\begin{subequations}
\label{eq:local:profile}
\begin{align}
\partial_s U + \tfrac12 U + \tfrac12 (y\cdot\nabla)U - \Delta U + (U\cdot\nabla)U + \nabla P &= 0,
\label{eq:local:profile:a}
\\
\nabla\cdot U &= 0.
\end{align}
\end{subequations}

\subsection{Consequences of the assumed local Type~I bounds}
We first record a quantitative interior estimate for bounded solutions, already used in rescaled form in the proof of Lemma~\ref{lem:gradient}.
\begin{lemma}[Interior regularity]
\label{lem:int:reg}
Let $r>0$, $A>0$, and let $(u,p)$ be a smooth solution to~\eqref{eq:NSE} on $Q_r(z_0)$ with $|u|\leq A \, r^{-1}$ on $Q_r(z_0)$. Then for every $\eps\in(0,1)$ and every integer $k\geq 0$,
\begin{equation*}
|\nabla_x^ku(x,t)|\leq C(k,\eps,A)\, r^{-k-1} ,  
\end{equation*}
for all $(x,t)\in Q_{r(1-\eps)}(z_0)$.
\end{lemma}
\begin{proof}[Proof of Lemma~\ref{lem:int:reg}] This result is well-known. 
Let $z_0 = (x_0,t_0)$. The function  $v(\tilde x, \tilde t ):= r\, u(x_0+r\, \tilde x, t_0+r^2 \,\tilde t )$ solves~\eqref{eq:NSE} on $Q_1$ with $\|v\|_{L^\infty(Q_1)}\leq A$, and satisfies $\nabla_{\tilde x}^k v = r^{k+1}\nabla^k_x u$. Using a quantitative version of Serrin's interior regularity result~\cite{Serrin62}, or equivalently, by iterating interior estimates for the non-stationary Stokes system~\cite[Lemma A.2]{ChenStrainTsaiYau08}, we obtain $\|\nabla_{\tilde x}^k v\|_{L^\infty(Q_{1-\eps})} \leq C(k,\eps,A)$. As remarked in footnote~\ref{foot:int:reg}, this bound is polynomial in $A$, and no a priori estimate on the pressure needs to be made. The claimed bound for the derivatives of $u$ now follows.
\end{proof}

\begin{lemma}[Local Type~I bounds for higher derivatives]
\label{lem:local:grad:bounds}
Let $(u,p)$ be a smooth solution to \eqref{eq:NSE} on $B_1\times[-1,0)$ satisfying the Type~I bound~\eqref{eq:local:typeI} in $Q_1^{\mathsf{o}}$, for some constant $C_u  >0$. Then, for every $\eps\in(0,\tfrac12]$ and every integer $k\geq0$,
\begin{equation}
\label{eq:grad:local:type1}
|\nabla^ku(x,t)|\leq \frac{C_{k,\eps}}{(-t)^{\frac{k+1}{2}}+|x|^{k+1}},
\qquad (x,t)\in Q_{1-\eps}^{\mathsf{o}},
\end{equation}
where $C_{k,\eps}$ depends only on $k$, $\eps$ and $C_u $.
\end{lemma}
\begin{proof}[Proof of Lemma~\ref{lem:local:grad:bounds}]
Fix $\eps\in(0,\tfrac12]$ and $(x,t)\in Q_{1-\eps}^{\mathsf{o}}$. Let
$c:= \min\{\tfrac{\eps}{4},\tfrac16\}$ and define $\rho:= c \, (|x|+\sqrt{-t})$. Then $Q_{2\rho}(x,t)\subset Q_1$. Indeed, since $|x|+\sqrt{-t}\leq 2$ we have $2\rho\leq 4c\leq\eps$, so $B_{2\rho}(x)\subset B_1$, while $(2\rho)^2\leq\eps^2\leq 1-(1-\eps)^2\leq 1-(-t)$.  Moreover, for $(\tilde x,\tilde t)\in Q_{2\rho}(x,t)$ we have $(-\tilde t)^{1/2} +|\tilde x|\geq\rho$. Indeed, if $\sqrt{-t}\geq|x|$ then $(-\tilde t)^{1/2} \geq\sqrt{-t}\geq\tfrac{\rho}{2c}\geq\rho$, and if $\sqrt{-t}<|x|$ then $|\tilde x|\geq |x|-2\rho\geq\rho(\tfrac1{2c}-2)\geq\rho$. We deduce from~\eqref{eq:local:typeI} that $|u(\tilde x,\tilde t)|\leq C_u /\rho$ on $Q_{2\rho}(x,t)$. 
Applying  Lemma~\ref{lem:int:reg} with $R=2\rho$, $A=2C_u $, and $\eps=\tfrac12$, gives $|\nabla^ku|\leq C_k\rho^{-k-1}$ on $Q_\rho(x,t)$, for $C_k$ depending only on $k$ and $C_u $. Evaluating at the top center $(x,t)$ yields~\eqref{eq:grad:local:type1}.
\end{proof}

\begin{corollary}[Profile bounds]
\label{cor:local:profile:bounds}
Assume the same hypotheses as in Lemma~\ref{lem:local:grad:bounds}. Let $U$ be as in~\eqref{eq:local:ss:variables}, and let $\Omega = \nabla\times U$. Then, we have
\begin{equation}
\label{eq:local:profile:bounds}
|\nabla^kU(y,s)|\leq \frac{C_{U,k}}{1+|y|^{k+1}},\qquad k\in\{0,1,2\},
\qquad |y|\leq \tfrac{3}{4}e^{\frac{s}{2}},\quad s\geq 1,
\end{equation}
for some constants $C_{U,k}$ which depend only on $C_u =C_{U,0}$. Letting $K:= \max\{C_{U,0}, C_{U,1}, C_{U,2},1\}$, we also have
\begin{equation}
\label{eq:local:vort:pointwise}
|\Omega(y,s)|\leq \frac{2K}{1+|y|^{2}},
\qquad
|\nabla\Omega(y,s)|\leq \frac{6K}{1+|y|^{3}},
\qquad |y|\leq \tfrac{3}{4}e^{\frac{s}{2}},\quad s\geq 1.
\end{equation}
\end{corollary}
\begin{proof}[Proof of Corollary~\ref{cor:local:profile:bounds}]
For $|y|\leq\tfrac34 e^{s/2}$ and $s\geq 1$, the point $(x,t) = (e^{-s/2}y, -e^{-s})$ has $|x|\leq\tfrac34$ and $-t\leq \frac{1}{e} \leq (\frac 34)^2$, so Lemma~\ref{lem:local:grad:bounds} with $\eps=\tfrac14$ gives
\[
|\nabla^k U(y,s)| = (-t)^{\frac{k+1}{2}}|\nabla^k u(x,t)|
\leq \frac{C_{k,1/4}\,(-t)^{\frac{k+1}{2}}}{(-t)^{\frac{k+1}{2}}+|x|^{k+1}}
= \frac{C_{k,1/4}}{1+|y|^{k+1}}.
\]
Then~\eqref{eq:local:vort:pointwise} follows from $|\Omega|^2 = |\nabla U|^2 - \tr((\nabla U)^2)\leq2|\nabla U|^2$ and from $\nabla\Omega$ being a combination of entries of $\nabla^2U$.
\end{proof}

\subsection{Small vorticity remains small and implies regularity}
The heart of the section is the following propagation lemma, which uses the local Type~I bound (via~\eqref{eq:local:vort:pointwise}) and the Biot--Savart law to propagate a locally small vorticity forward in time.

\begin{lemma}[Propagation of small vorticity]
\label{lem:local:propagation}
Under the hypotheses of Lemma~\ref{lem:local:grad:bounds} and Corollary~\ref{cor:local:profile:bounds}, there exists a universal constant $\vartheta \in (0,1]$ such that the following holds. For any $\theta_*\in(0,\vartheta]$, there exist $R_* = R_*(\theta_*,C_u )\geq2$, where $C_u $ is the constant from~\eqref{eq:local:typeI}, such that for any $R\geq R_*$ there exists $s_* = s_*(R) \geq1$  such that  
\begin{equation}
\label{eq:prop:hyp}
\|\Omega(\cdot,\bar s)\|_{L^2(B_{2R})}\leq\tfrac{1}{2} \theta_*
\quad\text{for some } \bar s\geq s_*
\quad\Longrightarrow\quad
\|\Omega(\cdot,s)\|_{L^2(B_{R})}\leq\theta_*
\quad\text{for all } s\geq\bar s.
\end{equation}
If~\eqref{eq:prop:hyp} holds, then also $\|\nabla U(\cdot,s)\|_{L^2(B_{R/2})} \leq 2 \theta_*$ for all $s\geq \bar s$.
\end{lemma}
\begin{proof}[Proof of Lemma~\ref{lem:local:propagation}]
Let $R\geq2$ be a parameter to be fixed below.
Let $\chi(y):=\bar\chi(y/R)$ with $\bar\chi$ a smooth radially non-increasing cutoff function such that $\bar\chi\equiv1$ on $B_1$ and $\bar\chi\equiv0$ on $B_2^\complement$. We consider times
\begin{equation}
\label{eq:prop:geometry}
s\geq\max\{1,\ 2\log(8 R)\},
\end{equation}
so that $B_{4R}\subseteq B_{\frac12 e^{\frac s2}}$ and $\supp\chi\subseteq B_{2R}$.\footnote{Note that $(U,\Omega)$ are well-defined in the cone $\{(y,s) \in \RR^3 \times [0,\infty) \colon |y| < e^{\frac{s}{2}}\}$; the bounds in Corollary~\ref{cor:local:profile:bounds} hold in the smaller set $\{(y,s) \in \RR^3 \times [1,\infty) \colon |y| \leq \frac 34 e^{\frac{s}{2}}\}$. }  

Taking the curl of~\eqref{eq:local:profile}, akin to~\eqref{eq:vort} (with the $\alpha$-term being replaced by the time derivative term) we obtain
\begin{equation}
\label{eq:local:vort:eq}
\partial_s\Omega - \Delta\Omega + \Omega + \tfrac12(y\cdot\nabla)\Omega + (U\cdot\nabla)\Omega = (\Omega\cdot\nabla)U.
\end{equation}
Upon defining the local enstrophy to be 
\[
\mathcal{F}^2 (s):=\int\chi^2|\Omega (\cdot,s)|^2\, dy, 
\]
by testing~\eqref{eq:local:vort:eq} against $\chi^2\Omega$ we obtain
\begin{equation}
\label{eq:local:enstrophy:identity}
\tfrac12 \tfrac{d}{ds} \mathcal{F}^2 + \tfrac14 \mathcal{F}^2 + \int\chi^2|\nabla\Omega|^2=  \int ((\chi\Omega)\cdot\nabla U) \cdot (\chi\Omega) + \mathcal{E},
\end{equation}
where the error term due to derivatives on the cutoff function is 
\begin{equation*}
\mathcal{E}:= \int\bigl(\tfrac12\Delta\chi^2 + \tfrac14 y\cdot\nabla\chi^2 + \tfrac12 U\cdot\nabla\chi^2\bigr)|\Omega|^2.
\end{equation*}
The error $\mathcal E$ is supported in the annulus $\{R\leq|y|\leq2R\}$, where due to~\eqref{eq:local:vort:pointwise} we have $|\Omega|^2\leq4K^2R^{-4}$ and $|U|\leq KR^{-1}$. Since $|\Delta\chi^2|\les R^{-2}$, $|y||\nabla\chi^2|\les 1$, and the annulus has volume $\les R^3$, we deduce that there exists a universal constant $C_{\mathcal{E}}  >0$ such that 
\begin{equation}
\label{eq:prop:influx}
|\mathcal{E}|\leq C_{\mathcal{E}} K^3 R^{-1}.
\end{equation}
This bounds the second term on the right side of~\eqref{eq:local:enstrophy:identity}.  In order to bound the first term on the right side of~\eqref{eq:local:enstrophy:identity}, we note that
\[
\Bigl| 
 \int ((\chi\Omega)\cdot\nabla U) \cdot (\chi\Omega)
\Bigr|
\leq \|\chi \Omega\|_{L^2} \|\nabla U\|_{L^4(B_{2R})} \|\chi \Omega\|_{L^4}
\leq \|\chi \Omega\|_{L^2} \|\nabla (\chi^\sharp U)\|_{L^4} \|\chi \Omega\|_{L^4}
,
\]  
where $\chi^\sharp(y) = \bar \chi(y/(2R))$. Since $B_{4R} \subset B_{\frac12 e^{\frac s2}}$ and $s\geq 1$, on this ball we may appeal to the bounds in~\eqref{eq:local:profile:bounds} and~\eqref{eq:local:vort:pointwise}, to deduce that there exists a universal constant $C^\prime>0$ such that\footnote{The simple bound is $\|\chi \Omega\|_{L^4} \les \|\chi \Omega\|_{L^2}^{1/4} \|\nabla (\chi \Omega)\|_{L^2}^{3/4}$ (from Gagliardo-Nirenberg) together with $\|\nabla (\chi \Omega)\|_{L^2} \les \|\chi \nabla \Omega\|_{L^2} + K R^{-3/2} $ (from~\eqref{eq:local:profile:bounds}--\eqref{eq:local:vort:pointwise} and the bound $\| (1 + |y|^2)^{-1}\|_{L^2(B_{2R}\setminus B_R)} \les R^{-1/2}$). Similarly, the standard div-curl estimate gives $\|\nabla (\chi^\sharp U)\|_{L^4} \les \|\nabla \cdot (\chi^\sharp U)\|_{L^4}  + \|\nabla \times (\chi^\sharp U)\|_{L^4} \les \|\nabla  \chi^\sharp \cdot U\|_{L^4} + \| \nabla \chi^\sharp \times U \|_{L^4}  + \| \chi^\sharp \Omega \|_{L^4} $; the first two terms are bounded using \eqref{eq:local:profile:bounds}--\eqref{eq:local:vort:pointwise}, while the last term is bounded just as $\|\chi \Omega\|_{L^4}$.}
\begin{equation}
\|\nabla (\chi^\sharp U)\|_{L^4} 
+ \|\chi \Omega\|_{L^4}
\leq 
C^\prime \|\chi \Omega\|_{L^2}^{\frac 14} \|\chi \nabla \Omega\|_{L^2}^{\frac 34} 
+ C^\prime \|\chi \Omega\|_{L^2} R^{-\frac 34} 
+ C^\prime K R^{-\frac 54}.
\label{eq:blahblahblah}
\end{equation}
Putting together the previous two bounds, using that $R\geq 2$, and appealing to $\epsilon$-Young, we obtain 
\begin{equation}
\label{eq:prop:influx:2}
\Bigl| 
 \int ((\chi\Omega)\cdot\nabla U) \cdot (\chi\Omega)
\Bigr| 
\leq 3 (C^\prime)^2 \mathcal{F} \bigl( \mathcal{F}^{2} + \|\chi \nabla \Omega\|_{L^2}^2
+  K^2 R^{-\frac 52} \bigr).
\end{equation}
With~\eqref{eq:local:enstrophy:identity}, \eqref{eq:prop:influx}, and \eqref{eq:prop:influx:2}, we arrive at
\begin{equation}
\label{eq:local:enstrophy:identity:new}
\tfrac12 \tfrac{d}{ds} \mathcal{F}^2 + \tfrac14 \mathcal{F}^2 + \|\chi \nabla \Omega\|_{L^2}^2
\leq 3 (C^\prime)^2 \mathcal{F} \bigl( \mathcal{F}^{2} + \|\chi \nabla \Omega\|_{L^2}^2
+  K^2 R^{-\frac 52} \bigr) + C_{\mathcal{E}} K^3 R^{-1},
\end{equation}
where $C^\prime$ and $C_{\mathcal{E}}$ are universal constants. 

Let $\vartheta:= \frac{1}{24} (C^\prime)^{-2}$. Let $\theta_* \leq \vartheta$ be arbitrary,  
and define $R_*=R_*(\theta_*,K)$ by
\[
R_* := \min \bigl\{R\geq 2 \colon K^2 R^{-\frac 52} + 8 C_{\mathcal{E}} K^3 R^{-1} \leq \tfrac 14 \theta_*^2 \bigr\}.
\]
Let $R\geq R_*$ be arbitrary, and according to \eqref{eq:prop:geometry} define $s_* = s_*(R)$ by
\[
s_*:= \max\{1,\ 2\log(8 R )\}.
\]

With these definitions, assume that there exists $\bar s \geq s_*$ such that $\|\Omega(\cdot,\bar s)\|_{L^2(B_{2R})} \leq \frac 12 \theta_*$. Then, by~\eqref{eq:local:enstrophy:identity:new} and continuity in time, there exists $\bar s^\prime > \bar s$ such that $\|\chi \Omega(\cdot,s)\|_{L^2} \leq  \theta_*$ for all $s\in[\bar s, \bar s^\prime]$. On this time interval, since $\theta_* \leq \vartheta$, and using the definition of $R_*$, we obtain from~\eqref{eq:local:enstrophy:identity:new} that
\begin{equation*}
\tfrac12 \tfrac{d}{ds} \mathcal{F}^2 + \tfrac18 \mathcal{F}^2 + \tfrac 78 \|\chi \nabla \Omega\|_{L^2}^2
\leq \tfrac 18 K^2 R^{-\frac 52}  + C_{\mathcal{E}} K^3 R^{-1}
\leq \tfrac 18 K^2 R_*^{-\frac 52}  + C_{\mathcal{E}} K^3 R_*^{-1}
\leq \tfrac{1}{32} \theta_*^2.
\end{equation*}
Integrating in time, for all $s\in[\bar s, \bar s^\prime]$ we have
\[
\mathcal{F}^2(s) \leq \mathcal{F}^2(\bar s)   + \tfrac{1}{4} \theta_*^2 \leq \tfrac{1}{2} \theta_*^2.
\]
By the standard continuity argument $\bar s^\prime$ may be taken to be arbitrarily large. Since $\chi\equiv 1$ on $B_{R}$, this proves~\eqref{eq:prop:hyp}.

To conclude the proof of the lemma, note that with $\chi^\flat(y):= \bar \chi(2 y /R)$, in analogy to~\eqref{eq:blahblahblah} we have $\|\nabla U\|_{L^2(B_{R/2})} \leq \|\nabla (\chi^\flat U)\|_{L^2} \leq \|\Omega\|_{L^2(B_{R})} + C^{\prime\prime} K R^{-1/2}$, where $C^{\prime\prime}>0$ is a universal constant. Upon increasing the value of $R_*$ if necessary, this concludes the proof of the Lemma.
\end{proof}

Combined with the $\eps$-regularity criterion of Caffarelli-Kohn-Nirenberg~\cite{CKN82}, Lemma~\ref{lem:local:propagation} implies the following regularity criterion.

\begin{proposition}[One time-slice regularity criterion for Type~I solutions]
\label{prop:one:time}
Consider a smooth solution $(u,p)$  to \eqref{eq:NSE} on $B_1\times[-1,0)$. Assume that $u$ satisfies the Type~I bound~\eqref{eq:local:typeI}, and assume that  $p \in L^\infty(A \times (-1,0))$, where $A$ is the annulus $A = \{\tfrac12<|x|<\tfrac34\}$.\footnote{Note that the pressure assumption is not quantitative and \emph{not sharp}; it merely ensures that the harmonic part of the pressure lies in $L^{3/2}(Q_1)$, by eliminating Serrin-type parasitic solutions. We could have alternatively assumed that: \emph{$(u,p)$ is a suitable weak solution of \eqref{eq:NSE} in $Q_1$, which satisfies the Type~I bound~\eqref{eq:local:typeI}}.}
There exists $\theta_*\in(0,1]$, $R_{**}\geq2$, and $s_{**}\geq1$, depending only on the constant $C_u $ from~\eqref{eq:local:typeI}, such that if
\begin{equation}
\label{eq:one:time:hyp}
\|\Omega(\cdot,\bar s)\|_{L^2(B_{2R_{**}})}\leq\tfrac{1}{2}\theta_*
\quad\text{for some }\bar s\geq s_{**}
\quad\Longrightarrow\quad
(0,0) \quad\text{is a regular point}.\footnote{Meaning, there exists $r>0$ such that  $u\in L^\infty(Q_r^{\mathsf o})$.}
\end{equation}
Equivalently, in the original $(x,t)$ variables, if there exists a single time $\bar t\in(-e^{-s_{**}},0)$,
such that 
\[
(-\bar t\, )^{1/2}\int_{B_{2R_{**}(-\bar t\, )^{1/2}}}|\omega(x,\bar t\,)|^2\,dx\leq \tfrac 14\theta_*^2 ,
\] 
then $(0,0)$ is a regular point.
\end{proposition}
\begin{proof}[Proof of Proposition~\ref{prop:one:time}]
Under the standing assumptions, $(u,p)$ is a suitable weak solution of the 3D Navier-Stokes equations in the parabolic cylinder $Q_{3/4}$, in the sense of~\cite{CKN82}. Since we assumed that $(u,p)$ are smooth on $B_1\times(-1,0)$, only integrability as $t\to 0^-$ needs to be checked. By~\eqref{eq:local:typeI} and Lemma~\ref{lem:local:grad:bounds} we have  $u \in L^\infty_t L^2_x(Q_{3/4})$ and $\nabla u \in L^2_{x,t}(Q_{3/4})$. To prove integrability of the pressure, fix $\varphi\in C_c^\infty(B_{7/8})$ with $\varphi\equiv1$ on $B_{3/4}$ and write $p = p_{\rm loc}+h$ on $B_{3/4}$, where $p_{\rm loc}(\cdot,t)$ is the Newtonian potential of $\partial_i\partial_j(\varphi\, u_iu_j)(\cdot,t)$, and   $\Delta h(\cdot,t)=0$ in $B_{3/4}$. By~\eqref{eq:local:typeI}, Lemma~\ref{lem:local:grad:bounds}, and standard Calderon-Zygmund estimates, we obtain that $p_{\rm loc} \in L^{3/2}_{x,t}(Q_{3/4})\cap L^\infty_{x,t}(A\times (-1,0))$, where $A =\{\tfrac12<|x|<\tfrac34\} $. The assumption $p \in L^\infty_{x,t}(A \times (-1,0))$ then gives $h\in L^\infty_{x,t}(A \times (-1,0))$, with $h(\cdot,t)$ harmonic in $B_{3/4}$. The fact that $h \in L^{3/2}_{x,t}(Q_{3/4})$ follows from the maximum principle. We are thus justified to invoke the $\epsilon$-regularity theorem for suitable weak solutions~\cite{CKN82,Lin98,LadySeregin99,Kukavica09}: there exists a universal constant $\eps_*>0$ such that if $\limsup_{r\to 0^+} \frac{1}{r} \int_{Q_r} |\nabla u|^2 \leq \eps_*$, then $(0,0)$ is a regular point. 

Let $\theta_* \in (0,1]$ be sufficiently small, in terms of universal constants, to be chosen later. Lemma~\ref{lem:local:propagation} fixes an  $R_* = R_*(\theta_*,C_u )\geq 2$. Pick any $R_{**} \geq R_*$, to be chosen later, and fix $s_{**}:=s_*(R_{**}) \geq 2 \log(8 R_{**})$, as in Lemma~\ref{lem:local:propagation}. As in~\eqref{eq:one:time:hyp}, assume that there exists $\bar s\geq s_{**}$ such that $\|\Omega(\cdot,\bar s)\|_{L^2(B_{2 R_{**}})} \leq \frac 12 \theta_*$. Then, Lemma~\ref{lem:local:propagation} guarantees that $\|\nabla U(\cdot,s)\|_{L^2(B_{R_{**}/2})} \leq 2 \theta_*$ for all $s\geq \bar s$.

Fix $r_0:= e^{-\bar s/2} \leq e^{-1} < \frac 34$, and let $r \in (0,r_0]$ and $t \in (-r^2,0)$. Using the self-similar transformation in~\eqref{eq:local:ss:variables}, with $s = -\log(-t) \geq - \log (r^2) \geq - \log(r_0^2) = \bar s$, we have that 
\[
\int_{B_r} |\nabla u(\cdot,t)|^2 dx
= 
(-t)^{-\frac 12} \int_{B_{r/\sqrt{-t}}} |\nabla U(\cdot,s)|^2 dy
=
e^{\frac s2} \int_{|y|\leq r   e^{s/2}}  |\nabla U(\cdot,s)|^2 dy.
\]
Note that $r e^{s/2} \leq r_0 e^{s/2} < \frac 34 e^{s/2}$, and thus we may appeal to the bound~\eqref{eq:local:profile:bounds} to estimate the right side of the above identity. Moreover, by construction we have $s\geq \bar s \geq 2 \log(8 R_{**})$; we may thus split the above integral into $|y|\leq \frac 12 R_{**}$ and $|y|\geq \frac 12 R_{**}$ to bound
\begin{equation}
\label{eq:Boston:sucks:1}
\int_{|y|\leq r   e^{s/2}}  |\nabla U(\cdot,s)|^2 dy
\leq  \| \nabla U(\cdot,s)\|_{L^2(B_{R_{**}/2})}^2 +  8 \pi C_{U,1}^2   R_{**}^{-1}
\leq 5 \theta_*^2,
\end{equation}
by choosing $R_{**} = R_{**}(\theta_*,C_u )= \max\{ R_*, 8 \pi C_{U,1}^2 \theta_*^{-2} \}$. Note that the bound~\eqref{eq:Boston:sucks:1} holds for all $s\geq \bar s$. Equivalently,~\eqref{eq:Boston:sucks:1} holds for all $t\in (-r_0^2,0)$, and thus also for all $t\in (-r^2,0)$. Using the bound~\eqref{eq:Boston:sucks:1} we may then compute
\begin{equation}
\label{eq:scaled:vorticity}
\frac{1}{r} \int_{Q_r}|\nabla u|^2\,dx\,dt
\leq 5 \theta_*^2 \cdot \frac{1}{r} \int_{-r^2}^{0} \frac{1}{\sqrt{-t}} dt
= 10 \, \theta_*^2 .
\end{equation}
The bound~\eqref{eq:scaled:vorticity} holds for all $r\leq r_0$. To conclude, we  pick $\theta_* \in (0,1]$ to be small enough to ensure $10 \, \theta_*^2 \leq \eps_*$, where $\eps_*>0$ is the universal constant from the $\epsilon$-regularity theorem of~\cite{CKN82}.
\end{proof}

\subsection{Proof of Theorem~\ref{thm:local:theorem}}
\label{sec:proof:local}
Fix $\bar s$ as in \eqref{eq:Ds:bound}, and let $\bar R := \frac 58 e^{\bar s/2}$. Since $\bar s\geq s_0$,  and $s_0$ is chosen large with respect to $C_u $ and $C_p$, we should think of $\bar R$ as being large with respect to $C_u $ and $C_p$.

For $\bar s \geq s_0 \geq 1$, $U(\cdot,\bar s)$ satisfies the bound ~\eqref{eq:local:profile:bounds}. From~\eqref{eq:local:pbound} and~\eqref{eq:local:ss:variables} we obtain that the pressure satisfies
$|P(y,s)|\leq C_p e^{-s}$ for all $\tfrac{1}{2} e^{\frac{s}{2}}\leq |y|\leq \tfrac{3}{4}e^{\frac{s}{2}}$, and all $s\geq 0$. Since by definition $\bar R \in (\frac 12 e^{\bar s/2},\frac 34 e^{\bar s/2})$, we deduce that 
\begin{equation}
\label{eq:local:pressure:annulus}
|P(y,\bar s)|\leq C_p e^{-\bar s},\qquad \mbox{for all} \qquad \tfrac 45 \bar R \leq |y|\leq \tfrac 65 \bar R.
\end{equation}
For the Bernoulli function $\Pi = P + \tfrac12|U|^2+\tfrac12 y\cdot U$, we deduce from~\eqref{eq:local:profile:bounds} and~\eqref{eq:local:pressure:annulus} that
\begin{align}
|\Pi(y,\bar s)| \leq C_p e^{-\bar s} + \tfrac{1}{2} C_u ^2 \bar R^{-2} + \tfrac 12 C_u  
\leq 2 C_u 
,\qquad \mbox{for all} \qquad |y|=\bar R,
\label{eq:local:Pi:bounds}
\end{align}
upon ensuring that $C_u  \bar R^{-2} \leq 3 C_u  e^{-\bar s} \leq 1$ and   $C_p e^{-\bar s} \leq C_u $ (both of these conditions are satisfied if $s_0$ is chosen to be sufficiently large with respect to $C_u $ and $C_p$). 

It will also be convenient  to obtain a bound for $P$ (and hence $\Pi$) in the entire ball $B_{\bar R}$; this bound is obtained as follows. Since $(U,P)(\cdot,\bar s)$ solve~\eqref{eq:local:profile} in the set $\{y \colon |y|<e^{\bar s/2} = \frac{8}{5} \bar R\}$, upon taking the divergence of~\eqref{eq:local:profile:a}, we obtain that $-\Delta P(\cdot,\bar s) = \mathrm{div}(U\cdot\nabla U)(\cdot,\bar s)$ in $B_{\frac{8}{5} \bar R}$. Let $\varphi$ be a  smooth radial cutoff with $\varphi\equiv 1$ on $B_{\bar R}$, $\varphi \equiv 0$ on $B_{\frac 65 \bar R}^\complement$, and $|\nabla \varphi| \leq {\bf 1}_{\bar R <|y|<\frac 65 \bar R} \cdot 10 \bar R^{-1}$. Multiplying the elliptic equation for $P(\cdot,\bar s)$ with $\varphi^2 P(\cdot,\bar s)$ and integrating by parts, we obtain $\|\nabla (P(\cdot,\bar s) \varphi)\|_{L^2}^2 \leq 2 \| (U \cdot\nabla U)(\cdot,\bar s) \varphi \|_{L^2}^2 + 3 \| P(\cdot,\bar s) \nabla \varphi\|_{L^2}^2$. By Gagliardo-Nirenberg we have $\dot{H}^1 \subset L^6$ in $\mathbb{R}^3$ (with sharp constant $(2/\pi)^{2/3} 3^{-1/2} < 1/2$), and so we with~\eqref{eq:local:profile:bounds} and~\eqref{eq:local:pressure:annulus} we deduce 
\begin{align}
\|P(\cdot,\bar s)\|_{L^6(B_{\bar R})} 
&\leq 
\|P(\cdot,\bar s) \varphi \|_{L^6(\mathbb{R}^3)}
\leq 
\tfrac 12
\|\nabla (P(\cdot,\bar s) \varphi) \|_{L^2(\mathbb{R}^3)}
\notag\\
&
\leq   \| (U \cdot\nabla U)(\cdot,\bar s) \varphi \|_{L^2(\mathbb{R}^3)}  + \tfrac 32 \| P(\cdot,\bar s) \nabla \varphi\|_{L^2(\mathbb{R}^3)} 
\notag\\
&
\leq C_{U,0} C_{U,1} \| \tfrac{1}{1+|y|} \tfrac{1}{1+|y|^2} \|_{L^2(B_{\frac 65 \bar R})} 
+ 10 C_p  \bar R^{-3} \bigl | \{ \tfrac 45 \bar R <|y|<\tfrac 65 \bar R \} \bigr|^{\frac 12}
\leq C_p^\prime
\label{eq:P:L6:bound}
\end{align}
where $C_p^\prime = C_p^\prime(C_u ,C_p) >1$ is an explicitly computable constant.

Next, we note that since $(U,P)$ solve~\eqref{eq:local:profile} in the set $\{(y,s) \colon |y|<e^{s/2}\}$, the Bernoulli identity (see~\eqref{eq:L:Pi:alpha:s} with $\alpha=0$) holds pointwise in this set; restricting this identity to time $\bar s$, we arrive at
\begin{equation}
\label{eq:Bernoulli:local}
\bar L\Pi(y,\bar s)+|\Omega(y,\bar s)|^2=-(U(y,\bar s)+\tfrac{1}{2}y)\cdot (\partial_s U)(y,\bar s),
\end{equation}
for $|y| < e^{\bar s/2}$, where 
\[
\bar L := -\Delta + \bigl(U(y,\bar s)+\tfrac12 y\bigr)\cdot\nabla.
\]
We consider the homogeneous Dirichlet eigenfunction $w_{\bar R}$ for the adjoint operator $\bar L^*$ on $B_{\bar R}$, as constructed in Lemma~\ref{lem:lambda:conv}, with associated principal eigenvalue $\lambda_{\bar R}\geq 0$. This construction applies verbatim since $\bar R < e^{\bar s/2}$, and $U(\cdot,\bar s)$ satisfies \eqref{eq:local:profile:bounds} in $\bar B_{\bar R}$. Recalling that $\bar L^* = -\Delta - (U(y,\bar s)+\frac12 y )\cdot\nabla - \frac 32$ (see~\eqref{eq:L*:def}), upon multiplying~\eqref{eq:Bernoulli:local}  by $w_{\bar R}$, integrating by parts,  using that $w_{\bar R}$ solves~\eqref{eq:Dirichlet:EV:L*:a}, and that $w_{\bar R} =0$ on $\partial B_{\bar R}$, we obtain that 
\begin{align}
& \int_{B_{\bar R}} \bigl| \Omega(\cdot,\bar s)\bigr|^2\, w_{\bar R}\, dy 
\notag 
\\
&=
-
\lambda_{\bar R} \int_{B_{\bar R}} \Pi(\cdot,\bar s) \, w_{\bar R}\, dy 
-
\int_{\partial B_{\bar R}} \Pi(\cdot,\bar s) \, \partial_n w_{\bar R}\, dy 
-
\int_{B_{\bar R}} (U(\cdot,\bar s)+\tfrac{1}{2}y)\cdot (\partial_s U)(\cdot,\bar s)
\, w_{\bar R}\, dy,
\label{eq:goodie}
\end{align}
where $\partial_n w_{\bar R} = \frac{y}{|y|} \cdot \nabla w_{\bar R}$  is the outward normal derivative of $w_{\bar R}$.

For the first term on the right side of~\eqref{eq:goodie}, we appeal to Lemma~\ref{lem:Gaussian:upper} (which is applicable if $\bar R$ is sufficiently large with respect to $C_u $) with $\eps =\frac 12$ to obtain the bound $w_{\bar R}(y) \leq M e^{-|y|^2/8}$ for some $M = M(C_u )>0$. Since $|\Pi(\cdot,\bar s)|\leq |P(\cdot,\bar s)| + \frac 12 C_u ^2 + \frac 12 C_u $, by using~\eqref{eq:P:L6:bound} we obtain that 
\begin{align}
\lambda_{\bar R} \int_{B_{\bar R}} \bigl| \Pi(\cdot,\bar s) \bigr| \, w_{\bar R}\, dy  
&\leq \tfrac 12 \lambda_{\bar R} (C_u ^2 + C_u ) M \| e^{-\frac 18 |y|^2}\|_{L^1(B_{\bar R})}
+ \lambda_{\bar R} M \|P(\cdot,\bar s)\|_{L^{6}(B_{\bar R})} \| e^{-\frac 18 |y|^2}\|_{L^{\frac 65}(B_{\bar R})}
\notag\\
&\leq  \lambda_{\bar R} \cdot 100 M ( C_u ^2 +  C_u  + C_p^\prime).
\label{eq:goodie:1}
\end{align}
For the second term on the right side of~\eqref{eq:goodie}, we notice that the outward normal derivative $\partial_n w_{\bar R}$  has a fixed sign on $\partial B_{\bar R}$; indeed, since $w_{\bar R} >0$ in $B_{\bar R}$ (see item (i) of Lemma~\ref{lem:lambda:conv}) and $w_{\bar R} = 0$ on $\partial B_{\bar R}$, we have that $\partial_n w_{\bar R} \leq 0$ on $\partial B_{\bar R}$. Moreover, integrating~\eqref{eq:Dirichlet:EV:L*:a} over $B_{\bar R}$ we obtain $- \int_{\partial B_{\bar R}} \partial_n w_{\bar R} = \lambda_{\bar R} \int_{B_{\bar R}} w_{\bar R}$, and thus $\| \partial_n w_{\bar R}\|_{L^1(\partial B_{\bar R})} = \lambda_{\bar R} \| w_{\bar R}\|_{L^1(B_{\bar R})}$. Combining this identity with the bound for $\Pi$ in~\eqref{eq:local:Pi:bounds}, and the Gaussian upper bound $w_{\bar R}(y) \leq M e^{-|y|^2/8}$, we deduce 
\begin{align}
\int_{\partial B_{\bar R}} \bigl| \Pi(\cdot,\bar s) \bigr| \, |\partial_n w_{\bar R}| \, dy  
&\leq \lambda_{\bar R} \cdot  300 C_u  M 
\label{eq:goodie:2}
\end{align}
Lastly, for the third term on the right side of~\eqref{eq:goodie}, we appeal to the $U$ bound in~\eqref{eq:local:profile:bounds}, the Gaussian upper bound for $w_{\bar R}$, and crucially, to the $\partial_s U$ bound in~\eqref{eq:Ds:bound}, and deduce 
\begin{align}
\int_{B_{\bar R}}  \bigl( |U(\cdot,\bar s)| +\tfrac{1}{2}|y| \bigr) \, \bigl| (\partial_s U)(\cdot,\bar s)\bigr| w_{\bar R} \, dy 
&\leq (1+C_u)  M \int_{B_{\bar R}} \bigl| (\partial_s U)(y,\bar s)\bigr| (1+|y|) e^{-\frac 18 |y|^2} \, dy
\notag\\
&\leq \delta_0  \cdot  600 (1+C_u)  M .
\label{eq:goodie:3}
\end{align}
By combining~\eqref{eq:goodie}--\eqref{eq:goodie:3} with the upper bound for $\lambda_{\bar R}$ established in item~(iv) of Lemma~\ref{lem:lambda:conv}, we deduce that there exist constants $K_1 = K_1(C_u ,C_p)\geq 1$ and $K_2 = K_2(C_u )\geq 1$ such that 
\begin{equation}
\int_{B_{\bar R}} \bigl| \Omega(\cdot,\bar s)\bigr|^2\, w_{\bar R}\, dy 
\label{eq:goodie:4}
\leq K_1  \bar R^{-2} + K_2 \delta_0 ,
\end{equation}
whenever $\bar R= \frac 58 e^{\bar s/2}$ is sufficiently large, in terms of $C_u $ and $C_p$.

In order to conclude the proof, we fix the constants $\theta_* \in (0,1]$, $R_{**} \geq 2$, and $s_{**}\geq 1$ from Proposition~\ref{prop:one:time}. \emph{We emphasize that  these constants only depend on $C_u $}. 
By Harnack's inequality for the uniformly elliptic operator $\bar L^* - \lambda_{\bar R}$ (see e.g.~\cite[Corollary 9.25]{GilbargTrudinger98}), there exists a constant $K_H = K_H(R_{**},C_u )>0$, such that 
\[
1 \leq \sup_{B_{2 R_{**}}} w_{\bar R} \leq K_H \inf_{B_{2 R_{**}}} w_{\bar R}
\]
for \emph{all} $\bar R \geq 4 R_{**}$. Using the above Harnack bound and \eqref{eq:goodie:4}, we obtain that 
\[
K_H^{-1} \|\Omega(\cdot,\bar s)\|_{L^2(B_{2R_{**}})}^2
\leq 
 K_1  \bar R^{-2} + K_2 \delta_0 
\]
for $\bar R= \frac 58 e^{\bar s/2}$  sufficiently large in terms of $C_u $ and $C_p$.
Upon letting $\delta_0  = K_H^{-1} K_2^{-1} \cdot \frac 18 \theta_*^2$ (a constant that only depends on $C_u $), and taking $\bar R$ to be large enough to also ensure that $K_H K_1 \bar R^{-2} \leq \frac 18 \theta_*^2$ (put differently, since $\bar s \geq s_0$, we   take $s_0$ large enough in terms of $C_u $ and $C_p$), we obtain that 
\[
\|\Omega(\cdot,\bar s)\|_{L^2(B_{2R_{**}})} \leq \tfrac 12 \theta_*. 
\]
By Proposition~\ref{prop:one:time} we conclude that $(0,0)$ is a regular point.

\section*{Acknowledgements}
The work of B.P. was partially supported by a fellowship in the Simons Society of Fellows.
The work of V.V. was partially supported by the Collaborative NSF grant DMS-2307681 and a Simons Investigator Award.  The authors thank Julien Guillod and Vladimir \v{S}ver\'ak for stimulating discussions.

\end{document}